\documentclass[11pt]{amsart}
\usepackage[babel]{csquotes}
\usepackage{enumitem}
\usepackage{amsmath,amsthm,amssymb,mathrsfs,amsfonts,verbatim,enumitem,color,leftidx}
\usepackage{mathabx}
\usepackage{etoolbox} 
\usepackage{bbm}
\usepackage[all,tips]{xy}
\usepackage{graphicx,ifpdf}
\usepackage{stmaryrd}
\usepackage{amssymb}
\usepackage{dsfont}

\ifpdf
\fi
\usepackage[colorlinks]{hyperref}
\hypersetup{
linkcolor=blue,          % color of internal links (change box color with linkbordercolor)
citecolor=green,        % color of links to bibliography
}

\usepackage{tikz}%pictures drawn in latex

\newtheorem{thm}{Theorem}[section]
\newtheorem{lem}[thm]{Lemma}
\newtheorem{cor}[thm]{Corollary}
\newtheorem{prop}[thm]{Proposition}

\theoremstyle{definition}
\newtheorem{defi}[thm]{Definition}

\theoremstyle{remark}
\newtheorem{rem}[thm]{Remark}

\numberwithin{equation}{section}
\definecolor{esperance}{rgb}{0.0,0.5,0.0}

\newcommand{\bp}{\mathbf{p}}

\newcommand{\bs}{\mathbf{s}}

\newcommand{\bx}{\mathbf{x}}
\newcommand{\by}{\mathbf{y}}

\newcommand{\R}{\mathbb{R}}
\newcommand{\N}{\mathbb{N}}
\newcommand{\Z}{\mathbb{Z}}

\newcommand{\s}{\sigma}
\newcommand{\de}{\delta}

\newcommand{\Ad}{\operatorname{Ad}}

\DeclareMathOperator{\diam}{diam}

\newcommand{\Bad}{\mb{Bad}}

\newcommand{\al}{\alpha}
\newcommand{\ga}{\gamma}
\newcommand{\Ga}{\Gamma}
\newcommand{\del}{\delta}
\newcommand{\Del}{\Delta}
\newcommand{\lam}{\lambda}

\newcommand{\eps}{\epsilon}

\newcommand{\Om}{\Omega}
\newcommand{\vphi}{\varphi}
\newcommand{\cA}{\mathcal{A}}
\newcommand{\cB}{\mathcal{B}}
\newcommand{\cC}{\mathcal{C}}
\newcommand{\cD}{\mathcal{D}}
\newcommand{\cE}{\mathcal{E}}

\newcommand{\cJ}{\mathcal{J}}

\newcommand{\cL}{\mathcal{L}}

\newcommand{\cP}{\mathcal{P}}
\newcommand{\cQ}{\mathcal{Q}}

\newcommand{\cS}{\mathcal{S}}

\newcommand{\bR}{\mathbb{R}}
\newcommand{\bZ}{\mathbb{Z}}
\newcommand{\bQ}{\mathbb{Q}}

\newcommand{\bN}{\mathbb{N}}

\newcommand{\bT}{\mathbb{T}}

\newcommand{\SL}{\operatorname{SL}}

\newcommand{\ASL}{\operatorname{ASL}}

\newcommand\ov[1]{\overline{#1}}

\newcommand\set[1]{\left\{#1\right\}}

\newcommand\idist[1]{\langle#1\rangle}

\newcommand\on[1]{\operatorname{#1}}

\newcommand\mb[1]{\mathbf{#1}}

\newcommand\crly[1]{\mathscr{#1}}

\newcommand{\wstar}{\overset{\on{w}^*}{\lra}}
\newcommand{\Supp}{\on{Supp}}

\newcommand{\defn}{\overset{\on{def}}{=}}
\newcommand{\lra}{\longrightarrow}

\newcommand{\onto}{\xymatrix{\ar@{>>}[r]&}}
\newcommand{\eq}[1]
{
\begin{equation*}
{#1}
\end{equation*}
}
\newcommand{\eqlabel}[2]
{
\begin{equation}
{#2}\label{#1}
\end{equation}
}
\makeatletter
\newcommand*{\rom}[1]{\expandafter\@slowromancap\romannumeral #1@}
\makeatother

\begin{document}

\title{Dimension estimates for badly approximable affine forms}

%\date{November 30, 2021}

\author{Taehyeong Kim}
\address{Taehyeong ~Kim. Department of Mathematical Sciences, Seoul National University, {\it kimth@snu.ac.kr}}
\author{Wooyeon Kim}
\address{Wooyeon Kim. Department of Mathematics, ETH Z\"{u}rich, {\it wooyeon.kim@math.ethz.ch}}
\author{Seonhee Lim}
\address{S.~Lim. Department of Mathematical Sciences and Resesarch Institute of Mathematics, Seoul National University,
{\it slim@snu.ac.kr}}

% \thanks will become a 1st page footnote.
\thanks{}

%\author{}


\def\thefootnote{}
   %%d 
\def\thefootnote{\arabic{footnote}}
\setcounter{footnote}{0}

\begin{abstract}
For given $\eps>0$ and $b\in\bR^m$, we say that a real $m\times n$ matrix $A$ is 
$\eps$-badly approximable for the
target $b$ if  
$$\liminf_{q\in\bZ^n, \|q\|\to\infty} \|q\|^n \idist{Aq-b}^m \geq \eps,$$
where $\idist{\cdot}$ denotes the distance from the nearest integral vector.
In this article, we obtain upper bounds for the Hausdorff dimensions of the set of $\epsilon$-badly approximable 
matrices for fixed target $b$ and the set of $\eps$-badly approximable targets for fixed matrix $A$. 
Moreover, we give a Diophantine condition of $A$ equivalent to the full Hausdorff dimension of the set of $\eps$-badly approximable targets for fixed $A$. 
The upper bounds are established by effectivizing entropy rigidity in homogeneous dynamics, 
which is of independent interest. 
For the $A$-fixed case, our method also works for the weighted setting where the supremum norms are replaced by
certain weighted quasinorms. 
%It was known that for almost every matrix $A$, the Hausdorff dimension of the set $Bad_A(\epsilon)$ of 
%$\epsilon$-badly approximable target $b$ is not full, and that for dimension 1, i.e. for real numbers $\alpha$, 
%$\dim_H Bad_\alpha(\epsilon)=1$ if and only if $\alpha$ is singular on average. 
%We show that $\dim_H Bad_A(\epsilon)=m$ if and only if $A$ is singular on average.
\end{abstract}

\maketitle
\section{Introduction}
\subsection{Main results} 
In classical Diophantine approximation, one wants to approximate an irrational number $\alpha$ by rationals $p/q$ 
for $p,q \in \mathbb{Z}$. Dirichlet theorem says that for every $N \in \mathbb{N}$, there exist 
$p,q \in \mathbb{Z}$ with $0<q<N$, such that
$$|q\alpha-p|<1/N < 1/q.$$
In this way, one can see classical Diophantine approximation as studying distribution of $q\alpha$ modulo 
$\mathbb{Z}$ near zero. Diophantine approximation for irrational numbers has been generalized to investigating 
vectors, linear forms, and more generally matrices, and have become classical subjects in metric number theory. 

In this article, we consider the inhomogenous Diophantine approximation: the distribution of $q\alpha$ modulo 
$\mathbb{Z}$ near a ``target" $b \in \mathbb{R}$. Although Dirichlet theorem does not hold anymore, 
there exist infinitely many $q\in\bZ$ such that 
\[
|q\al-b-p| < 1/|q| \quad\text{for some }p\in\bZ 
\]
for almost every $(\al,b)\in\bR^2$ and moreover,  
\[
\liminf_{p,q\in\bZ, |q|\to \infty} |q||q\al -b-p|=0 
\]
for almost every $(\al,b)\in\bR^2$ by inhomogeneous Khintchine theorem (\cite[Theorem \rom{2} in 
Chapter \rom{7}]{Cas57}). 

% Let $M_{m,n}(\bR)$ be the set of $m\times n$ real matrices, and let $\widetilde{M}_{m,n}(\bR)$ be the direct product of $M_{m,n}(\bR)$ and $\bR^m$. 
Similarly to numbers, for an $m \times n$ real matrix $A\in M_{m,n}(\bR)$, we study $Aq \in \R^m$  
modulo $\mathbb{Z}^m$ near the target $b \in \mathbb{R}^m$ for vectors $q \in \mathbb{Z}^n$.
In this general situation as well, using inhomogeneous Khintchine-Groshev theorem (\cite[Theorem1]{Sch64} 
or \cite[Chapter1, Theorem 15]{Spr79}), we have
%$$||q||^{n/m} \idist{Aq-b} <1,$$
%for infinitely many $q \in \mathbb{Z}^m$. Furthermore, 
$$\liminf_{q\in\bZ^n, \|q\|\to \infty} \|q\|^{n}\idist{Aq-b}^{m}=0$$
for almost every $(A,b) \in {M}_{m,n}(\bR) \times \bR^m$. 
Here, $\idist{v}: =\displaystyle\inf_{p\in \bZ^m} \|v-p\|$ denotes the distance from $v \in \bR^m$ 
to the nearest integral vector with respect to the supremum norm $\|\cdot \|$.

The exceptional set of the above equality is our object of interest. We will consider the exceptional set with weights 
in the following sense. Let us first fix, throughout the paper, an $m$-tuple and an $n$-tuple of positive reals 
$\mathbf{r}=(r_1,\cdots,r_m)$, $\bs=(s_1,\cdots,s_n)$ such that $r_1\geq \cdots \geq r_m$, $s_1\geq \cdots \geq s_n$, and $\displaystyle\sum_{1\leq i\leq m}r_i=1=\displaystyle\sum_{1\leq j\leq n}s_j$. The special case where $r_i=1/m$ and  $s_j=1/n$ for all $i=1,\dots,m$ 
and $j=1,\dots,n$ is called the unweighted case.

Define the $\mathbf{r}$-quasinorm of $\bx\in\bR^m$ and $\bs$-quasinorm of $\by\in\bR^n$ by
$$\|\bx\|_{\mathbf{r}}:=\max_{1\leq i\leq m}|x_i|^{\frac{1}{r_i}} \quad\textrm{and}\quad \|\by\|_{\mathbf{s}}:=\max_{1\leq j\leq n}|y_j|^{\frac{1}{s_j}}.$$
Denote $\idist{\bx}_\mathbf{r}: =\displaystyle\inf_{p\in \bZ^m} \|\bx-p\|_{\mathbf{r}}$. 
We call $A$ $\eps$-\textit{bad} for $b\in\mathbb{R}^m$ if
\eqlabel{eq1523}{
\liminf_{q\in\bZ^n, \|q\|_{\mb{r}} \to \infty} \|q\|_{\mathbf{s}}\idist{Aq-b}_{\mathbf{r}}\ge \eps
.}
Denote 
\begin{align*}
\mb{Bad}(\eps)&\defn\set{(A,b)\in {M}_{m,n}(\bR) \times \bR^m :A\textrm{ is $\eps$-bad for $b$}},\\ 
\mb{Bad}_A(\eps)&\defn\set{b\in\bR^m:A\textrm{ is $\eps$-bad for $b$}}, \;\;\mb{Bad}_A\defn\bigcup_{\eps>0}\mb{Bad}_A(\eps),\\
\mb{Bad}^b(\eps)&\defn\set{A\in M_{m,n}(\bR):A\textrm{ is $\eps$-bad for $b$}}, \;\; \mb{Bad}^b\defn\bigcup_{\eps>0}\mb{Bad}^b(\eps).
%\mb{Bad}^{b}(\eps)&\defn\bigcup_{b\in\bR^m}\mb{Bad}^b(\eps).
\end{align*}

The set $\mb{Bad}^0$ can be seen as the set of badly approximable systems of $m$ linear forms in $n$ variables.
This set is of Lebesgue measure zero \cite{Gro38}, but has full Hausdorff dimension $mn$ \cite{Sch69}. 
See \cite{PV02,KTV06,KW10} for the weighted setting.

For any $b$, $\mb{Bad}^b$ also has zero Lebesgue measure \cite{Sch} and full Hausdorff dimension for every $b$
\cite{ET}. Indeed, it is shown that $\mb{Bad}^b$ is a winning set \cite{ET}  and even a hyperplane winning set
\cite{HKS}, a property which implies full Hausdorff dimension. On the other hand, the set $\mb{Bad}_A$  also has full
Hausdorff dimension for every $A$ \cite{BHKV10}. See \cite{Har12,HM17,BM17} for the weighted setting. 
%but can have positive Lebesgue measure. \seon{reference needed} 

The sets $\mb{Bad}^b$ and $\mb{Bad}_A$ are unions of subsets $\mb{Bad}^b(\eps)$ and 
$\mb{Bad}_A(\eps)$ over $\eps>0$, respectively, thus a more refined question is whether the Hausdorff dimension
of $\mb{Bad}^b(\eps)$, $\mb{Bad}_A(\eps)$ could still be of full dimension. 
For the homogeneous case ($b=0$), the Hausdorff dimension $\mb{Bad}^0(\eps)$ is less than the full dimension 
$mn$ (see \cite{BK13, Sim} for the unweighted case and \cite{KM19} for the weighted case). 
Thus, a natural question is whether $\mb{Bad}^b(\eps)$ can have full Hausdorff dimension for some $b$. 
Our first main result says that in the unweighted case, $\mb{Bad}^b(\eps)$ cannot have full Hausdorff dimension 
for any $b$. We provide an effective bound on the dimension in terms of $\eps$ as well.

\begin{thm}\label{corb1} For the unweighted case, i.e. $r_i=1/m$ and  $s_j=1/n$ for all $i=1,\dots,m$ 
and $j=1,\dots,n$, there exist $c_0>0$ and $M_0>0$ depending only on $d$ such that for any $\eps>0$
and $b\in\bR^m$, $$\dim_H \mb{Bad}^b(\eps)\leq mn-c_0\eps^{M_0}.$$
\end{thm}

As for the set $\mb{Bad}_A(\eps)$, the third author, together with U. Shapira and N. de Saxc\'e, showed that
Hausdorff dimension of $\mb{Bad}_A(\eps)$ is less than the full dimension $m$ for almost every $A$ \cite{LSS}. 
In fact, it was shown that one can associate to $A$ a certain point $x_A$ in the space of unimodular lattices 
$\SL_d(\R)/\SL_d(\Z)$ such that if $x_A$ has no escape of mass on average for a certain diagonal flow 
(see Section~\ref{sec:1.2} for more details), which is satisfied by almost every point, then the Hausdorff dimension 
of $\mb{Bad}_A(\eps)$ is less than $m$.

In this article, we provide an effective bound on the dimension in terms of $\epsilon$ and a certain Diophantine property of $A$ as follows. 
We say that an $m\times n$ matrix $A$ is $\textit{singular on average}$ if for any $\eps>0$
$$\lim_{N\to\infty}\frac{1}{N}\left| \set{l\in\set{1,\cdots,N}:\exists q\in\Z^n \ \text{s.t.} \ 
\idist{Aq}_{\mb{r}}<\eps 2^{-l} \ \textrm{and} \ 0<\|q\|_{\mb{s}}<2^l}\right| =1.$$
\begin{thm}\label{thmEff1}
For any $A\in M_{m,n}(\bR)$ which is not singular on average, there exists a constant $c(A)>0$ depending on $A$
such that for any $\eps>0$, $\dim_H \mb{Bad}_{A}(\eps)\leq m-c(A)\frac{\eps}{\log(1/\eps)}.$
\end{thm}
Here, the constant $c(A)$, which depends on $\eta_A$ in Proposition \ref{prop5} and $H$ in \eqref{eqH}, encodes the quantitative singularity on average.

On the other hand, the third author, together with Y. Bugeaud, D. H. Kim and M. Rams, showed that 
in the one-dimensional case ($m=n=1$), $\Bad_\al (\eps)$ has full Hausdorff dimension for some $\eps>0$ if and 
only if $\al\in\bR$ is singular on average \cite{BKLR}.
We generalize this characterization to the general dimensional setting.
\begin{thm}\label{thmA1}
Let $A\in M_{m,n}(\bR)$ be a matrix. Then the following are equivalent:
\begin{enumerate}
\item\label{S1} For some $\eps > 0$, the set $\mb{Bad}_{A}(\eps)$ has full Hausdorff dimension.
\item\label{S3} $A$ is singular on average.
\end{enumerate}
\end{thm}
Note that the implication (\ref{S1}) $\implies$ (\ref{S3}) of Theorem \ref{thmA1} follows from 
Theorem \ref{thmEff1}. The other direction will be shown in Section \ref{sec6}.

\subsection{Discussion of the proofs}\label{sec:1.2}
We mainly use entropy rigidity in homogeneous dynamics, which is a principle that the measure with maximal 
entropy is invariant \cite{EL}. The main tool in \cite{LSS} is a relative version of entropy rigidity.
In this article, we effectivize this phenomenon in terms of static entropy and conditional measures. 
To use the effective version of the entropy rigidity, for each invariant measure, we construct a ``well-behaved" partition and a $\sigma$-algebra, well-behaved in the sense that the ``dynamical $\del$-boundary'' has small measure which is controlled uniformly (see Definition \ref{dynbdy} and Lemma \ref{Exceptional}).
We then compare the associated dynamical entropy and the static entropy. Section \ref{sec2} consists of these results
in the general setting of real Lie groups such as in \cite{EL}, which are of independent interest.

To describe the scheme of the proofs for main theorems, we consider more specific homogeneous space as follows.
For $d=m+n$, let us denote by $\ASL_d(\bR)=\SL_d(\bR)\ltimes\bR^d$ the set of area-preserving affine
transformations and denote by $\ASL_d(\bZ)=\SL_d(\bZ)\ltimes\bZ^d=\operatorname{Stab}_{\ASL_d(\bR)}(\bZ^d)$ 
the stabilizer of the standard lattice $\bZ^d$.
We view $\ASL_d(\bR)$ as a subgroup of $\SL_{d+1}(\bR)$ by $\ASL_d(\bR)=\set{\left(\begin{matrix}
g & v\\
0 & 1\\
\end{matrix}\right): g\in \SL_d(\bR), v\in\bR^d},$ and take a lift of the element $g\in\SL_d(\bR)$ to $\ASL_d(\bR)\subset \SL_{d+1}(\bR)$ by 
$g\longmapsto\left(\begin{matrix}
g & 0\\
0 & 1\\
\end{matrix}\right),$
denoted again by $g$.
For given weights $\mb{r}\in\bR^{m}_{>0}$ and $\mb{s}\in\bR^{n}_{>0}$, we consider the $1$-parameter diagonal subgroup 
$$\set{a_t=\mathrm{diag} (e^{r_1t},\cdots,e^{r_mt},e^{-s_1t},\cdots,e^{-s_nt})}_{t \in \mathbb R}$$ 
in $\SL_d(\bR)$ and  Let $a\defn a_1$ be the time-one map of the diagonal flow $a_t$.
We consider 
$$U=\set{ \left(\begin{matrix}
I_m & A & 0\\
0 & I_n & 0\\
0 & 0 & 1\\
\end{matrix}\right)
:A\in M_{m,n}(\bR)};\; \; 
W=\set{
\left(\begin{matrix}
I_m & 0 & b\\
0 & I_n & 0\\
0 & 0 & 1\\
\end{matrix}\right)
:b\in \bR^m},$$ both of which are unstable horospherical subgroups in $\ASL_d(\bR)$ for $a$.

The homogeneous spaces $\SL_d(\bR)/\SL_d(\bZ)$ and $\ASL_d(\bR)/\ASL_d(\bZ)$ can be seen 
as the space of unimodular lattices and the space of unimodular grids, i.e. unimodular lattices translated by a vector in $\bR^d$, respectively.
We say that a point $x\in \SL_d(\bR)/\SL_d(\bZ)$ has \emph{$\del$-escape of mass on average} (with respect to the diagonal flow $a_t$) if for any compact set $Q$ in $\SL_d(\bR)/\SL_d(\bZ)$,
$$\displaystyle\liminf_{N\to\infty}\frac{1}{N}|\set{\ell\in\set{1,\dots,N}: a_\ell x\notin Q}|\ge\del.$$
A point $x \in X$ has no escape of mass on average if it does not have $\del$-escape of mass on average for any $\del>0$.

For $A\in M_{m,n}(\bR)$ and $(A,b)\in {M}_{m,n}(\bR) \times \bR^m$, we associate points 
$$x_A:=
\left(\begin{matrix}
I_m & A\\
0 & I_n\\
\end{matrix}\right)\SL_d(\bZ)\quad \text{and}\quad y_{A,b}:=
\left(\begin{matrix}
I_m & A & -b\\
0 & I_n & 0\\
0 & 0 & 1\\
\end{matrix}\right)\ASL_d(\bZ),$$ respectively.
In \cite{LSS}, it was shown that $\dim_H \mb{Bad}_A(\eps)<m$ for all $\eps>0$ if $x_A$ is $\textit{heavy}$ which is a condition equivalent to no escape of mass on average. Note that $x_A$ is heavy for almost every $A\in M_{m,n}(\R)$. 
On the other hand, we remark that $A$ is singular on average if and only if 
the corresponding point $x_A$ has $1$-escape of mass on average (with respect to the diagonal flow $a_t$) by Dani's correspondence (see also \cite{KKLM}).
%
%For $m=n=1$, $A$ is singular on average if and only $\mb{Bad}_A(\eps)$ has full Hausdorff dimension for some $\eps>0$ \cite{BKLR}. 
%For  $(m,n)\neq (1,1)$, nothing was known about $\dim_H \mb{Bad}_A(\eps)$ for $A$ which has $\eta$-escape of mass on average for some $0<\eta<1$. Theorem~\ref{thmEff1} above deals with this case.

Now we give the outline of the proofs for Theorem \ref{corb1} and Theorem \ref{thmEff1}.
From the Dani correspondence, we characterize the Diophantine property 
$(A,b)\in \Bad(\eps)$ by the dynamical property that the orbit $(a_t y_{A,b})_{t\geq 0}$ is eventually 
in some target $\cL_\eps$ (see Subsection \ref{sec3.2}).
Using this characterization, we construct $a$-invariant measures with large dynamical entropies relative to
$W$ and $U$ (Proposition \ref{prop5} and Proposition \ref{prop2}), which are related to
the Hausdorff dimensions of $\Bad_A(\eps)$ and $\Bad^b(\eps)$, respectively.
Here, we use ``well-behaved" $\sigma$-algebra constructed in Proposition \ref{algebracst}.
Then we associate the dynamical entropies with the static entropies (Lemma \ref{algexiA}). 
Finally, we obtain effective upper bounds for the Hausdorff dimensions 
of $\Bad_A(\eps)$ and $\Bad^b(\eps)$ using an effective version of the variational principle (Proposition \ref{effEL}).

To treat $\Bad_A(\eps)$ and $\Bad^b(\eps)$ at the same time, we need to consider the entropy relative to arbitrary expanding 
closed subgroup $L$ normalized by $a$, which is more general than \cite{LSS}: in \cite{LSS}, the special case $L=W$ whose orbits stay in the compact fiber of $\ASL_d(\bR)/\ASL_d(\bZ) \to \SL_d(\bR)/\SL_d(\bZ)$ is considered. 

For $\Bad_A(\eps)$, we treat the case when $x_A$ has some escape of mass on average whereas $x_A$ has no escape of mass on average in \cite{LSS}. We need to consider $\cL_\eps\subset \ASL_d(\bR)/\ASL_d(\bZ)$, which is non-compact,
whereas in \cite{LSS}, for heavy $x_A$, it was enough to consider the set of fibers over a compact part of $\SL_d(\bR)/\SL_d(\bZ)$. In the case of $\Bad^b(\eps)$, as fixing $b$ does not determine the amount of excursion in the cusp, we need an additional step (Proposition \ref{KKLM'}) to control the measure near the cusp allowing a small amount of escape of mass.

Another new feature of this article is the use of the effective equidistribution of expanding translates under the diagonal action
on $\ASL_d(\bR)/\ASL_d(\bZ)$ and $\SL_d(\bR)/\Gamma_q$, where $\Gamma_q$ is a congruence subgroup of $\SL_d(\bZ)$,
in the case of $\Bad^b(\eps)$.
The former result is proved by the second author in \cite{Kim21}, and the latter result is a slight modification of \cite{KM21}.

Note that \cite{Kim21, KM21} hold in the weighted setting and the only reason we consider the unweighted setting for the 
Hausdorff dimension of $\Bad^b(\eps)$ is the covering estimate in Theorem \ref{KKLMcov} (Theorem 1.5 of \cite{KKLM}). 
%
%We remark that the set of matrices which are singular on average has Hausdorff dimension at most $mn-\frac{mn}{m+n}$ \cite{KKLM}. We also remark that the case where the best approximation vectors $Q_k$ of $A$ satisfies $\lim_{k \to \infty} |Q_k|^{1/k} = \infty$, then $\dim_H \mb{Bad}_A(\eps)=m$ for some $\eps>0$ \cite{BKLR}. Thus the only remaining open question is whether $\dim_H   \mb{Bad}_A(\eps)$ can be full is the case where $A$ is singular on average but $ |Q_k|^{1/k}$ is bounded, which is proved in Theorem~\ref{thmA1}.

The article is organized as follows. In Section \ref{sec2}, we introduce entropy, relative entropy, and a general setup. In this general setup, we construct a partition with a well-behaved ``dynamical $\del$-boundary'' and a $\sigma$-algebra in a quantitative sense. From this construction, we compare the dynamical entropy and the static entropy. Finally, we prove an effective version of the variational principle for relative entropy in the spirit 
of \cite[7.55]{EL}. In Section \ref{sec3}, we introduce preliminaries for the proofs of dimension upper bounds 
including properties of dimensions with respect to quasi-metrics. We also reduce badly approximable properties to dynamical properties in the space of grids in $\bR^{m+n}$. 
In Section \ref{sec:entropyboundA} and Section \ref{sec:entropyboundb}, we construct $a$-invariant measures on 
$\ASL_d(\bR)/\ASL_d(\bZ)$ with large relative entropy and estimate dimension upper bounds of Theorem \ref{thmEff1} and Theorem \ref{corb1} using the effective variational principle. We conclude the paper with Section \ref{sec6}, characterizing the singular on average property in terms of best approximations and show (\ref{S3})$\implies$(\ref{S1}) part in Theorem \ref{thmA1} using a modified version of Bugeaud-Laurent sequence in \cite{BL}.

\section{Effective version of entropy rigidity}\label{sec2}
In this section, we will establish an effective version of entropy rigidity in \cite[Section 7]{EL}.
There have been effective uniqueness results along the line of \cite{EL} in various settings: \cite{Pol} for toral automorphisms, \cite{Kad} for hyperbolic maps on Riemannian manifolds, \cite{Ruh} on p-adic homogeneous spaces, and 
\cite{Kha} for a $p$-adic diagonal action in the $S$-arithmetic setting.
However, in all of the above results as well as in \cite{KLP}, there exists a partition compatible with the given map or flow
in the sense that images under the iteration have boundaries of small measure with respect to any invariant measure
of interest.  

In our setting of a diagonal action on a quotient of real Lie groups, one of the main technical difficulty is that there is no such  partition for all the invariant measures we consider. We thus construct a partition $\cP$ for each invariant measure $\mu$ and control the $\mu$-measure of its ``dynamical $\del$-boundary'' $E_\del$ constructed out of images of thickenings 
of the boundary $\cP$. The value $\mu(E_\del)$ is bounded above uniformly over the partition $\cP$ and the measure $\mu$. See Lemma \ref{Exceptional}.

\subsection{Entropy and relative entropy}\label{sec2.1}
In this subsection, we recall the definitions of the entropy and the relative entropy for $\sigma$-algebras which we use in the later sections. We refer the reader to \cite[Chapter 1 \& 2]{ELW} for basic properties of the entropy.

\begin{defi}
Let $(X,\cB,\mu,T)$ be a measure-preserving system on a Borel probability space, and let $\cA, \cC \subseteq\cB$ be sub-$\sigma$-algebras. 
Suppose that $\cC$ is countably generated.
Note that there exists an $\cA$-measurable conull set $X'\subset X$ and a system $\set{\mu_x^\cA|x\in X'}$ of measures on $X$, 
referred to as \emph{conditional measures}, given for instance by \cite[Theorem 2.2]{ELW}. 
The \emph{information function} of $\cC$ given $\cA$ with respect to $\mu$ is defined by
$$I_\mu(\cC|\cA)(x)=-\log\mu_x^\cA([x]_{\cC}),$$
where $[x]_\cC$ is the atom of $\cC$ containing $x$. 

\begin{enumerate}
    \item  The \emph{conditional (static) entropy of $\cC$ given $\cA$} is defined by
$$H_{\mu}(\cC|\cA):=\int_X I_\mu(\cC|\cA)(x)d\mu(x),$$
which is the average of the information function. If the $\sigma$-algebra $\cA$ is trivial, then we denote by $H_{\mu}(\cC)=H_{\mu}(\cC|\cA)$, 
which is called the \emph{(static) entropy of $\cC$}. Note that the entropy of the countable partition $\xi=\set{A_1,A_2,\dots}$ of $X$ is given by
$$H_\mu(\xi)=H(\mu(A_1),\dots)=-\displaystyle\sum_{i\ge 1}\mu(A_i)\log\mu(A_i)\in[0,\infty],$$
where $0\log0=0$.

\item Let $\cA\subseteq\cB$ be a sub-$\sigma$-algebra such that $T^{-1}\cA = \cA$. For any countable partition $\xi$ of $X$, let
$$h_\mu(T,\xi):=\displaystyle\lim_{n\to\infty}\frac{1}{n}H_\mu (\xi_0^{n-1})
=\displaystyle\inf_{n\ge 1}\frac{1}{n}H_\mu(\xi_0^{n-1}),$$
$$h_\mu(T,\xi|\cA):=\displaystyle\lim_{n\to\infty}\frac{1}{n}H_\mu(\xi_0^{n-1}|\cA)
=\displaystyle\inf_{n\ge 1}\frac{1}{n}H_\mu(\xi_0^{n-1}|\cA),$$
where $\xi_0^{n-1}= \bigvee_{i=0}^{n-1}T^{-i}\xi$.
Then the \emph{(dynamical) entropy of $T$} is
$$h_\mu(T):=\displaystyle\sup_{\xi:H_\mu(\xi)<\infty}h_\mu(T,\xi).$$
Moreover, the \emph{conditional (dynamical) entropy of $T$ given $\cA$} is
$$h_\mu(T|\cA):=\displaystyle\sup_{\xi:H_\mu(\xi)<\infty}h_\mu(T,\xi|\cA).$$
\end{enumerate}

\end{defi}

\subsection{General setup}\label{sec2.2}
Let $G$ be a closed real linear group (or connected, simply connected real Lie group) and let $\Ga<G$ be a lattice.
We consider the quotient $Y=G/\Ga$ with a $G$-invariant probability measure $m_Y$ and call it 
Haar measure on $Y$. Let $d_G$ be a right invariant metric on $G$, which induces the metric $d_Y$ on 
the space $Y=G/\Ga$, which is locally isometric to $G$.
Let $r_y$ be the maximal injectivity radius at $y\in Y$, which is the supremum of $r>0$ such that the
map $g\mapsto gy$ is an isometry from the open $r$-ball $B^G_r$ around the identity in $G$ 
onto the open $r$-ball $B^Y_r(y)$ around $y\in Y$.
For any $r>0$, we denote $$Y(r):=\{y \in Y : r_y \geq r\}.$$ 
It follows from the continuity of the injectivity radius that $Y(r)$ is compact.
Let us denote 
$$r_{\max} := \inf\{r>0 : r_y \leq r \text{ for all } y \in Y\}. $$
Since $\Gamma$ is a lattice, $r_{\max}<\infty$. Hence we now assume that $r_{\max}\leq 1$ 
by rescaling the right invariant metric $d_G$ on $G$.  
Note that for any $r>1$, $Y(r) = \varnothing$.
For any closed subgroup $L<G$, we consider the right invariant metric $d_L$ by restricting $d_G$ on $L$, and similarly denote by $B^L_r$ the open $r$-ball around the identity in $L$. 
In this section, we fix an element $a\in G$ which is $\Ad$-diagonalizable over $\bR$. 
Let $G^{+} = \set{g\in G|a^k g a^{-k}\to id \ \textrm{as} \ k\to -\infty}$ be the unstable (resp. stable) horospherical subgroup associated to $a$ (resp. $a^{-1}$), which is always a closed subgroup of $G$ in our setting.

\subsection{Construction of $a^{-1}$-descending, subordinate algebra and its entropy properties}\label{sec2.3}
In this subsection, our goal is to strengthen results of \cite[\S 7]{EL} for our quantitative purposes.
%\seon{Check this sentence.}
%\seon{Do we need stable in the folllowing definition?}
\begin{defi}[7.25. of \cite{EL}]\label{algdef}
Let $G^+ <G$ be the unstable horospherical subgroup associated to $a$. Let $\mu$ be an $a$-invariant measure on $Y$ and $L<G^+$ be a closed subgroup normalized by $a$.
\begin{enumerate}
\item We say that a countably generated $\s$-algebra $\cA$ is \emph{subordinate to $L$} (mod $\mu$) if for $\mu$-a.e. $y$, there exists $\de > 0$ such that
\eqlabel{Subordef}{B^{L}_\de\cdot y \subset [y]_{\cA} \subset B^{L}_{\de^{-1}}\cdot y.}
\item
We say that $\cA$ is \emph{$a^{-1}$-descending} if $(a^{-1})^{-1}\cA= a\cA \subseteq \cA$.
\end{enumerate}
\end{defi}

For each $L<G^+$ and $a$-invariant ergodic probability measure $\mu$ on $Y$, there exists a countably generated $\sigma$-algebra $\cA$ which is $a^{-1}$-descending and subordinate to $L$ \cite[Proposition 7.37]{EL}. We will prove that such a $\sigma$-algebra can be constructed so that we also have an explicit upper bound of the measure of the set violating \eqref{Subordef} for fixed $\del>0$. In order to prove an effective version of the variational principle later, we need this quantitative estimate independent of $\mu$.

We first introduce some notations that will be used in this subsection.
For a subset $B\subset Y$ and $\del>0$, we denote by $\partial_\del B$ the $\del$-neighborhood of the boundary of $B$, i.e.
$$\partial_\del B\defn\set{y\in Y: \displaystyle\inf_{z\in B}d_Y(y,z)+\displaystyle\inf_{z\notin B}d_Y(y,z)<\del}.$$
We also define the neighborhood of the boundary of a countable partition $\cP$ by
$$\partial_\del \cP\defn\displaystyle\bigcup_{P\in\cP}\partial_\del P.$$
Here, we deal with the entropy with respect to $a^{-1}$, so we write for any extended integers $\ell \leq \ell'$ in $\mathbb{Z}\cup\set{\pm\infty}$,
$$\cP_{\ell}^{\ell'}=\displaystyle\bigvee_{k=\ell}^{\ell'} a^{k}\cP,$$
for a given partition $\cP$ of $Y$. We will use this notation also for $\sigma$-algebras. 

We first construct a finite partition which has small measures on neighborhoods of the boundary. The following lemma is the main ingredient of the effectivization in this section. A key feature is that the measure estimate below is independent of $\mu$.
\begin{lem}\label{partitioncst}
There exists a constant $0<c<\frac{1}{10}$ depending only on $G$ such that the following holds. 
Let $\mu$ be a probability measure on $Y$. For any $r>0$ and any measurable subset $\Omega\subset Y(2r)$, 
there exist a measurable subset $K\subset Y$ and a partition $\cP=\set{P_1,\cdots,P_N}$ of $K$ such that
\begin{enumerate}
    \item\label{partprop1}  $\Omega\subseteq K\subseteq B^G_{\frac{11}{10}r}\Omega$,
    \item\label{partprop2} For each $1\leq i\leq N$, there exists $z_i\in B^G_{\frac{r}{10}}\Omega$ such that $$B_{\frac{r}{5}}^G\cdot z_i\subseteq P_i\subseteq B_{r}^G\cdot z_i,\qquad K=\bigcup_{i=1}^{N}B_{r}^G\cdot z_i,$$
    \item\label{partprop3} $\mu(\partial_\del\cP)\leq \left(\frac{\del}{r}\right)^{\frac{1}{2}}\mu(B^G_{\frac{12}{10}r}\Omega)$ for any $0<\del<cr$.
\end{enumerate}
\end{lem}

\begin{proof}
Choose a maximal $\frac{9}{10}r$-separated set $\set{y_1,\cdots,y_N}$ of $\Omega$. \vspace{0.3cm}\\ 
%Note that $\Omega\subseteq \bigcup_{i}B_{\frac{9}{10}r}^G\cdot y_i\subseteq B^G_{\frac{9}{10}r} \Omega$.
\textbf{Claim}\; There exist a constant $0<c<\frac{1}{10}$ depending only on $G$, and $\{g_i\}_{i=1}^N \subset B^G_{\frac{r}{10}}$ such that for $z_i=g_iy_i$ and for any $0<\del<cr$,
\eqlabel{eqcst1}{\sum_i\left(\mu(\partial_\del (B_r^G\cdot z_i))+\mu(\partial_\del (B_{\frac{r}{2}}^G\cdot z_i))\right)\leq \left(\frac{\del}{r}\right)^{\frac{1}{2}}\mu(B^G_{\frac{12}{10}r}\Omega).}
\begin{proof}[Proof of Claim]
To prove this claim, we randomly choose each $g_i$ with the independent uniform distribution on $B^G_{\frac{r}{10}}$. 
For $0<\del<\frac{r}{10}$ fixed, we have
\eq{
\begin{aligned}
\mathbb{E}&\left(\sum_i\mu(\partial_\del (B_r^G\cdot z_i))\right)=\sum_i\frac{1}{m_G(B^G_{\frac{r}{10}})}\int_{B^G_{\frac{r}{10}}}\int_{Y}\mathds{1}_{B^G_{r+\del}\cdot g_iy_i\setminus B^G_{r-\del}\cdot g_iy_i}(y)d\mu(y) dm_G(g_i)\\
&\asymp\sum_i\frac{1}{r^{\dim G}}\int_Y m_G\left(\set{g_i\in B^G_{\frac{r}{10}}: r-\del\leq d(g_iy_i,y)<r+\del}\right)d\mu(y)\\
&\ll \sum_i\frac{1}{r^{\dim G}}\int_{B^G_{\frac{11}{10}r+\del}\cdot y_i}\del r^{\dim G-1}d\mu
\leq \frac{\del}{r}\int_{B_{\frac{12}{10}r}^G\Om}\sum_i\mathds{1}_{B_{\frac{12}{10}r}\cdot y_i}(y)d\mu(y).
\end{aligned}
}
For any $y\in B_{\frac{12}{10}r}^G\Om$, the number of $y_i$'s contained in $B_{\frac{12}{10}r}^G \cdot y$ is at most $\left(\frac{33}{9}\right)^{\dim G}$ 
since $B_{\frac{9}{20}r}^G \cdot y_i$'s are disjoint and contained in $B_{\frac{33}{20}r}^G \cdot y$. It implies that $\sum_i\mathds{1}_{B_{\frac{12}{10}r}\cdot y_i}(y)\leq 4^{\dim G}$ for any $y\in B_{\frac{12}{10}r}^G\Om$. It follows that
$$\mathbb{E}\left(\sum_i\mu(\partial_\del (B_r^G\cdot z_i))\right)\ll \frac{\del}{r}\int_{B_{\frac{12}{10}r}^G\Om}4^{\dim G}d\mu(y)\ll \frac{\del}{r}\mu(B_{\frac{12}{10}r}^G\Om),$$
where the implied constant depends only on $G$.

Applying the same argument for $\partial_\del (B_{\frac{r}{2}}^G\cdot z_i)$ instead of $\partial_\del (B_r^G\cdot z_i)$,
$$\mathbb{E}\left(\sum_i\left(\mu(\partial_\del (B_r^G\cdot z_i))+\mu(\partial_\del (B_{\frac{r}{2}}^G\cdot z_i))\right)\right)\ll\frac{\del}{r}\mu(B^G_{\frac{12}{10}r}\Omega).$$
It follows from Chebyshev's inequality that 
$$\mathbb{P}\left(\sum_i\left(\mu(\partial_\del (B_r^G\cdot z_i))+\mu(\partial_\del (B_{\frac{r}{2}}^G\cdot z_i))\right)\ge \frac{1}{2}\left(\frac{\del}{r}\right)^{\frac{1}{2}}\mu(B^G_{\frac{12}{10}r}\Omega)\right)\ll \left(\frac{\del}{r}\right)^{\frac{1}{2}}.$$
Hence, we have
\eqlabel{eqpcst}{
\begin{split}
\mathbb{P}&\left(\bigcap_{k\geq 0}\left\{\sum_i\left(\mu(\partial_{2^{-k}\del} (B_r^G\cdot z_i))+\mu(\partial_{2^{-k}\del} (B_{\frac{r}{2}}^G\cdot z_i))\right)< \frac{1}{2}\left(\frac{2^{-k}\del}{r}\right)^{\frac{1}{2}}\mu(B^G_{\frac{12}{10}r}\Omega)\right\}\right)\\&>1-O\left( \left(\frac{\del}{r}\right)^{\frac{1}{2}}\right).
\end{split}
}
Thus, there exists $0<c<\frac{1}{10}$ so that the right-hand side of \eqref{eqpcst} is positive for any $\del<cr$. It follows that we can find $\set{g_i}_{i=1}^{N}$ such that $z_i=g_iy_i$'s satisfy \eqref{eqcst1} for any $0<\del<cr$. 
\end{proof}

Let $c>0$ and $\{g_i\}_{i=1}^N \subset B^G_{\frac{r}{10}}$ be as in \textbf{Claim}. The set $\set{z_i = g_i y_i}_{i=1}^N$ is $\frac{7}{10}r$-separated since $\set{y_i}_{i=1}^N$ is $\frac{9}{10}r$-separated. Let $K:=\bigcup_{i=1}^N B_r^G\cdot z_i$. Since $B_{\frac{9}{10}r}^G\cdot y_i\subseteq B_r^G\cdot z_i\subseteq B_{\frac{11}{10}r}^G\cdot y_i$, we have
$$\Omega\subseteq\bigcup_{i=1}^N B_{\frac{9}{10}r}^G\cdot y_i\subseteq K\subseteq \bigcup_{i=1}^N B_{\frac{11}{10}r}^G\cdot y_i\subseteq B_{\frac{11}{10}r}^G \Omega.$$
Now we define a partition $\cP$ of $K$ inductively as follows:
$$P_i=B^G_r\cdot z_i\setminus\left(\displaystyle\bigcup_{j=1}^{i-1}P_j\cup\displaystyle\bigcup_{j=i+1}^{N}B^G_{\frac{r}{2}}\cdot z_j\right)$$
for $1\leq i\leq N$.
It is clear from the construction that $B_{\frac{r}{5}}^G\cdot z_i\subseteq P_i\subseteq B_{r}^G\cdot z_i$ and $z_i\in B^G_{\frac{r}{10}}\Omega$ for $1\leq i\leq N$. We also observe that the $\del$-neighborhood of $\cP$ is contained in $\displaystyle\bigcup_{i=1}^N \left(\partial_\del(B_r^G\cdot z_i) \cup \partial_\del(B_\frac{r}{2}^G\cdot z_i)\right)$. Hence it follows from \textbf{Claim} that for any $0<\del<cr$,
$$\mu(\partial_\del\cP)\leq \sum_i\left(\mu(\partial_\del (B_r^G\cdot z_i))+\mu(\partial_\del (B_{\frac{r}{2}}^G\cdot z_i))\right)\leq \left(\frac{\del}{r}\right)^{\frac{1}{2}}\mu(B^G_{\frac{12}{10}r}\Omega).$$
\end{proof}

\begin{rem}\label{basepoint}
Under the same setting in Lemma \ref{partitioncst}, let $y\in\Omega$ be given. 
If we start the proof of the lemma with a maximal $\frac{9}{10}r$-separated set $\{y=y_1,\dots,y_N\}$ of $\Omega$,
then we can conclude not only \eqref{partprop1}, \eqref{partprop2}, and \eqref{partprop3} in the lemma 
but also $y \in B_{\frac{r}{10}}^G \cdot z_1 \subset B_{\frac{r}{5}}^G\cdot z_1 \subset P_1$.
Hence, $y \notin \partial \cP$, which will be used in proofs of Proposition \ref{prop5} and Proposition \ref{prop2}.
\end{rem}

%
%\begin{lem}\label{thicklem}
%For any $r>\del>0$, we have
%\[ B_{\del}^G Y(r) \subset Y(r-\del)\quad \text{and}\quad B_\del^G Y(r)^c \subset Y(r+\del)^c. \]
%\end{lem}
%\begin{proof}
%For any $g\in B_\del^G$ and $y\in Y(r)$, we need to show $r_{gy}\geq r-\del$.
%Suppose that $r_{gy}<r-\del$. Then there exist $g_1, g_2 \in B_{r-\del}^G$ such that $g_1gy =g_2 gy$, or equivalently, 
%$g^{-1}g_2^{-1}g_1 g y=y$. But it follows from $y\in Y(r)$ that $g^{-1}g_2^{-1}g_1g \notin B_r^{G}$, hence using the triangular inequality and the right invariance of $d_G$, 
%\[
%\begin{split}
%r &\leq d_G(g^{-1}g_2^{-1}g_1g,id) \leq d_G(g_1g,id)+d_G(g_2g,id)\\
%&\leq d_G(g_1,id) + d_G(g_2,id) + 2d_G(g,id) < r,
%\end{split}
%\] which is a contradiction. This concludes the first assertion. The second assertion follows similarly.
%\end{proof}

We need the following thickening properties.  
It can easily be checked that for any $r>\del>0$, we have
\eqlabel{thicklem}{
B_{\del}^G Y(r) \subset Y(r-\del)\qquad \text{and}\qquad B_\del^G Y(r)^c \subset Y(r+\del)^c.
}

Using Lemma \ref{partitioncst} inductively, we have the following partition of $Y$ 
with its subpartition having small boundary measures. 
Recall that $Y(r) = \varnothing$ for any $r>1$ by our choice of the right invariant metric $d_G$ on G.
\begin{lem}\label{partitioncst'}
Let $0<r_0\leq 1$ be given and $\mu$ be a probability measure on $Y$. 
There exists a partition $\set{K_k}_{k=1}^{\infty}$ of $Y$ such that for each $k\geq 1$, the following statements hold:
\begin{enumerate}
    \item $K_k\subseteq Y(2^{-k})\setminus Y(2^{-k+2});$
    \item there exist a partition $\cP_k=\set{P_{k1},\cdots,P_{kN_k}}$ of $K_k$ and a point 
    $z_i\in B^G_{\frac{1}{10}r_0 2^{-k-1}}K_k$ for each $1\leq i\leq N_k$ satisfying
    $$B_{\frac{1}{5}r_0 2^{-k-1}}^G\cdot z_i\subseteq P_{ki}\subseteq B_{r_0 2^{-k-1}}^G\cdot z_i;$$
    \item $\mu(\partial_\del\cP_k)\leq (r_0^{-1} 2^{k+4}\del)^{\frac{1}{2}}\mu(Y(2^{-k-1})\setminus Y(2^{-k+3}))$ for any 
    $0<\del<cr_02^{-k-2}$, where $c>0$ is the constant in Lemma \ref{partitioncst}. 
\end{enumerate}
\end{lem}
\begin{proof}
We will construct $\set{K_k}_{k\ge1}$ and $\set{\cP_k}_{k\ge 1}$ using Lemma \ref{partitioncst} inductively. 
For each $k\geq 1$, let us say that $K_k\subset Y$ and $\cP_k$ satisfy $(\spadesuit_k)$ if they satisfy the three conditions in the statement.
We will also need auxiliary bounded sets $K_k'\subset Y$'s and corresponding partitions $\cP_k'$'s during the inductive process. 
Let us say that $K_k'$ and a partition $\cP_k'$ of $K_k'$ satisfy $(\clubsuit_k)$ if they satisfy the following three conditions.
\begin{enumerate}
    \item $Y(2^{-k+1})\setminus\bigcup_{j=1}^{k-1}K_j\subseteq K_k'\subseteq B^G_{\frac{11}{10}r_02^{-k-1}}(Y(2^{-k+1})
    \setminus\bigcup_{j=1}^{k-1}K_j)$,
    \item For each $1\leq i\leq N_{k}$, there exists $z_{ki}\in B^G_{\frac{1}{10}r_0 2^{-k-1}}K_k'$ such that 
    $$B_{\frac{1}{5}r_0 2^{-k-1}}^G\cdot z_{ki}\subseteq P_{ki}'\subseteq B_{r_0 2^{-k-1}}^G\cdot z_{ki}\quad\text{and}
    \quad K_k'=\bigcup_{i=1}^{N}B_{r_02^{-k-1}}^G\cdot z_{ki},$$
    \item $\mu(\partial_\del\cP_k')\leq (r_0^{-1}2^{k+1}\del)^{\frac{1}{2}}\mu(Y(2^{-k})\setminus Y(2^{-k+3}))$ for any $0<\del<cr_02^{-k-1}$.
\end{enumerate}
Here, $\bigcup_{j=1}^{0} K_j$ means the empty set.

Let us start with the initial step.
We first choose $\Omega_1 = Y(1)$ and apply Lemma \ref{partitioncst} with $r=r_02^{-2}$ 
and $\Omega=\Omega_1 \subset Y(\frac{r_0}{2})$. Then we have a subset $K_1' \subset Y$ and 
a partition $\cP_1'$ of $K_1'$ satisfying (1), (2) of $(\clubsuit_{1})$, and 
$\mu(\partial_\del\cP_1')\leq (r_0^{-1} 2^2\del)^{\frac{1}{2}}\mu(B_{\frac{12}{10}r_02^{-2}}^G \Omega_1)$ 
for any $0<\del<cr_02^{-2}$. 
It follows from \eqref{thicklem} that $B_{\frac{12}{10}r_02^{-2}}^G Y(1) \subset Y(\frac{1}{2})$, 
which implies (3) of $(\clubsuit_{1})$ since $Y(4)=\varnothing$. 
Note that $K_1' \subset B_{\frac{11}{10}r_02^{-2}}^G Y(1)\subset Y(\frac{1}{2})$. 

Now let $\Omega_2 = Y(\frac{1}{2}) \setminus K_1'$ and apply Lemma \ref{partitioncst} again with $r=r_02^{-3}$ and 
$\Omega=\Omega_2 \subset Y(\frac{r_0}{4})$. We have a subset $K_2' \subset Y$ and a partition $\cP_2'$ of $K_2'$
satisfying $\Omega_2 \subset K_2' \subset B_{\frac{11}{10}r_0 2^{-3}}^G \Omega_2$, (2) of $(\clubsuit_{2})$, and 
$\mu(\partial_\del\cP_2')\leq (r_0^{-1} 2^{3}\del)^{\frac{1}{2}}\mu(B_{\frac{12}{10}r_0 2^{-3}}^G \Omega_2)$ 
for any $0<\del<cr_02^{-3}$. Set $K_1 = K_1' \setminus K_2 '$, then (1) of  $(\clubsuit_{2})$ and (1) of $(\spadesuit_1)$ 
follow since $Y(2)=\varnothing$. 
Since $K_1' \supset Y(1)$, it follows from \eqref{thicklem} that 
$B_{\frac{12}{10}r_02^{-3}}^G \Omega_2 \subset Y(\frac{1}{4})\setminus Y(2)$, which implies (3) of $(\clubsuit_{2})$.
Define a partition $\cP_1 = \{P_{11},\dots, P_{1N_1}\}$ from 
$\cP_1' = \{ P_{11}',\dots, P_{1N_1}'\}$ by $P_{1i}=P_{1i}' \setminus K_{2}'$ for each $1\leq i\leq N_1$.
For each $1\leq i\leq N_1$ and $y\in B^G_{\frac{1}{5}r_02^{-2}}\cdot z_{1i}$, observe that $y\notin K_{2}'$ 
since $B^G_{r_02^{-2}}\cdot z_{1i}\subset K_1'$ and $K_{2}'\subset B^G_{\frac{11}{10}r_02^{-3}}\Omega_2 \subset
B^G_{\frac{11}{10}r_0 2^{-3}}(Y\setminus K_1')$. Hence, $B^G_{\frac{1}{5}r_02^{-2}}\cdot z_{1i}\subset P_{1i}$ holds, 
so (2) of $(\spadesuit_1)$ follows. Since $P_{1i}=P_{1i}' \setminus K_{2}'$ for each $1\leq i\leq N_1$, we have
\[
\begin{split}
\mu(\partial_\del \cP_1) &\leq \mu(\partial_\del \cP_1') + \mu(\partial_\del \cP_2') 
\leq (r_0^{-1}2^2 \del)^{\frac{1}{2}}\mu(Y(2^{-1})\setminus Y(2^2)) + (r_0^{-1}2^3 \del)^{\frac{1}{2}}\mu(Y(2^{-2})\setminus Y(2)) \\
&\leq (r_0^{-1}2^5 \del)^{\frac{1}{2}}\mu(Y(2^{-2})\setminus Y(2^2))
\end{split}
\] 
for any $0<\del<cr_02^{-3}$.
Hence (3) of $(\spadesuit_1)$ follows.

Our desired disjoint sets $\set{K_k}_{k\ge 1}$ and partitions $\set{\cP_k}_{k\ge 1}$ will be obtained by applying this process repeatedly.\vspace{0.3cm}\\
\textbf{Claim}\; 
For $k\ge 2$, suppose that we have disjoint bounded sets $K_j$ of $Y$ and corresponding partitions 
$\cP_j$ satisfying $(\spadesuit_j)$ for $j=1,\dots, k-1$, 
and a subset $K_k' \subset Y$ and a partition $\cP_k'$ satisfying $(\clubsuit_k)$. 
Then we can find $K_k\subseteq K_k'$ and a partition $\cP_k$ of $K_k$ satisfying $(\spadesuit_k)$, 
and $K_{k+1}' \subset Y$ and a partition $\cP_{k+1}'$ of $K_{k+1}'$ 
satisfying $(\clubsuit_{k+1})$.

\begin{proof}[Proof of Claim]
Note that $K_k' \subset B^G_{\frac{11}{10}r_02^{-k-1}}Y(2^{-k+1})\subset Y(2^{-k})$ and 
$K_j \subset Y(2^{-j}) \subset Y(2^{-k})$ for each $j=1,\dots, k-1$.
 Let $\Omega_{k+1}=Y(2^{-k})\setminus (\bigcup_{j=1}^{k-1} K_j\cup K_k')$ and apply Lemma \ref{partitioncst} 
 with $r=r_02^{-k-2}$ and $\Omega=\Omega_{k+1} \subset Y(r_02^{-k-1})$. Then there exist 
 $K_{k+1}'\subset Y$ and a partition $\cP_{k+1}'=\set{P_{(k+1)1}',\cdots,P_{(k+1)N_{k+1}}'}$ of $K_{k+1}'$ satisfying 
$\Omega_{k+1} \subset K_{k+1}' \subset B_{\frac{11}{10}r_02^{-k-2}}^G \Omega_{k+1}$, (2) of $(\clubsuit_{k+1})$, and 
$\mu(\partial_\del\cP_{k+1}')\leq (r_0^{-1}2^{k+2}\del)^{\frac{1}{2}}\mu(B_{\frac{12}{10}r_02^{-k-2}}^G \Omega_{k+1})$ 
for any $0<\del<cr_02^{-k-2}$. We set $K_k=K_k'\setminus K_{k+1}'$, then (1) of  $(\clubsuit_{k+1})$ follows. 
Since $\bigcup_{j=1}^{k-1} K_j \supset Y(2^{-k+2})$ and $K_k \subset K_k' \subset Y(2^{-k}) \setminus \bigcup_{j=1}^{k-1}K_j$,
(1) of $(\spadesuit_k)$ follows. It follows from $\bigcup_{j=1}^{k-1} K_j \cup K_k' \supset Y(2^{-k+1})$ and 
\eqref{thicklem} that $B_{\frac{12}{10}r_02^{-k-2}}^G \Omega_{k+1} \subset Y(2^{-k-1})\setminus Y(2^{-k+2})$,
which implies (3) of $(\clubsuit_{k+1})$.
Define a partition $\cP_k=\set{P_{k1},\cdots, P_{kN_k}}$ from $\cP_k'=\set{P_{k1}',\cdots,P_{kN_k}'}$ by
 $P_{ki}=P_{ki}'\setminus K_{k+1}'$ for any $1\leq i\leq N_k$. For each $1\leq i\leq N_k$ and 
 $y\in B^G_{\frac{1}{5}r_02^{-k-1}}\cdot z_{ki}$, observe that $y\notin K_{k+1}'$ since 
 $B^G_{r_02^{-k-1}}\cdot z_{ki}\subseteq K_k'$ and $K_{k+1}'\subseteq B^G_{\frac{11}{10}r_02^{-k-2}}\Omega_{k+1}
 \subset B^G_{\frac{11}{10}r_02^{-k-2}}(Y\setminus K_k')$. 
 Hence, $B^G_{\frac{1}{5}r_02^{-k-1}}\cdot z_{ki}\subset P_{ki}$ holds, so (2) of $(\spadesuit_k)$ follows. 
Since $P_{ki}=P_{ki}' \setminus K_{k+1}'$ for each $1\leq i\leq N_k$, we have
\[
\begin{split}
\mu(\partial_\del \cP_k) &\leq \mu(\partial_\del \cP_k') + \mu(\partial_\del \cP_{k+1}') \\
&\leq (r_0^{-1}2^{k+1}\del)^{\frac{1}{2}}\mu(Y(2^{-k})\setminus Y(2^{-k+3})) +
(r_0^{-1}2^{k+2}\del)^{\frac{1}{2}}\mu(Y(2^{-k-1})\setminus Y(2^{-k+2})) \\
&\leq (r_0^{-1}2^{k+4}\del)^{\frac{1}{2}}\mu(Y(2^{-k-1})\setminus Y(2^{-k+3}))
\end{split}
\] for any $0<\del<cr_02^{-k-2}$.
Hence (3) of $(\spadesuit_k)$ follows.
\end{proof}
This claim concludes the proof of Lemma \ref{partitioncst'}.
\end{proof}

By \cite[Lemma 7.29 and 7.45]{EL}, there are constants $\alpha>0$ and $d_0>0$ depending on $a$ and $G$ such that for every $r\in(0,1]$, 
\eqlabel{excEq}{a^{-k}B_{r}^{G^+}a^{k}\subset B_{d_0 e^{-k\alpha}r}^{G}}
for any $k\in\bZ$. It implies that $a^{k}B_{r}^{G}a^{-k}\subset B_{d_0 e^{k\alpha}r}^{G}$ for $k\geq 0$.

The following lemma is a quantitative strengthening of \cite[Lemma 7.31]{EL}. We remark that the constants below are independent of $\mu$ and $\cP$ while the ``dynamical $\del$-boundary'' $E_\del$ depends on $\mu$.

\begin{defi}\label{dynbdy}
We define the dynamical $\del$-boundary of the partition $\cP$
by $$E_\del=\bigcup_{k=0}^\infty a^k\partial_{d_0e^{-k\alpha}\del}\cP.$$
\end{defi}

\begin{lem}\label{Exceptional}
Given $0<r_0\leq 1$ and an $a$-invariant probability measure $\mu$ on $Y$, 
let $\set{K_j}_{j\ge 1}$ and $\set{\cP_j}_{j\ge 1}$ be the sets and the partitions we constructed in Lemma \ref{partitioncst'}. 
Set the countable partition $\cP=\bigcup_{j=1}^\infty \cP_j$ of $Y$. 
Then there exist $C_1,C_2>0$ such that the following holds:
Let $c>0$ and $d_0>0$ be the constants in Lemma \ref{partitioncst} and \eqref{excEq}. 
    For any $0<\del<\min((\frac{cr_0}{16d_0})^{2},1)$, the dynamical $\del$-boundary $E_\del\subset Y$ satisfies 
    $$\mu(E_{\del})<\mu(Y\setminus Y(C_1\del^{\frac{1}{2}}))+C_2\del^{\frac{1}{4}}$$ and 
$B^{G^+}_{\del}\cdot y\subset [y]_{\cP_0^{\infty}}$ for any $y\in Y\setminus E_\del$. 
%
%\begin{enumerate}
%    \item\label{partupper} For any $P\in\cP$ there exists $j\ge 1$ such that $P\subseteq Y(2^{-j})\setminus Y(2^{-j+2})$.
%     Moreover, there exists $z\in P$ such that $$B^G_{\frac{1}{5}r_02^{-j-1}}\cdot z\subseteq P
%     \subseteq B^G_{r_02^{-j-1}}\cdot z.$$
%    \item\label{partlower} Let $c>0$ and $d_0>0$ be the constants in Lemma \ref{partitioncst} and \eqref{excEq}. 
%    For any $0<\del<\min((\frac{cr_0}{16d_0})^{2},1)$, there exists $E_\del\subset Y$ such that 
%    $$\mu(E_{\del})<\mu(Y\setminus Y(C_1\del^{\frac{1}{2}}))+C_2\del^{\frac{1}{2}}$$ and 
%$B^{G^+}_{\del}\cdot y\subset [y]_{\cP_0^{\infty}}$ for any $y\in Y\setminus E_\del$. 
%\end{enumerate}
Here, the constants $C_1,C_2 $ depend only on $r_0$, $a$, and $G$. 
\end{lem}
\begin{proof}
%Then $\cP$ is a countable partition of $Y$ and the condition (1) directly follows from Lemma \ref{partitioncst'}.
We split $E_\del$ into two subsets
$$E_\del'=\bigcup_{k=0}^\infty a^k\left(\bigcup_{i= 2+\lceil\frac{\alpha}{\log 2}k-\frac{\log\del}{2\log2}\rceil}^\infty\partial_{d_0e^{-k\alpha}\del}\cP_i\right),\qquad
E_\del''=\bigcup_{k=0}^\infty a^k\left(\bigcup_{i= 1}^{1+\lceil\frac{\alpha}{\log 2}k-\frac{\log\del}{2\log2}\rceil}\partial_{d_0e^{-k\alpha}\del}\cP_i\right).$$

We claim that $E_\del'\subset Y\setminus Y((d_0+d_0^2)\del^{\frac{1}{2}})$. 
To see this, let $y\in E_\del'$. Then there exist $k\geq 0$ and 
$P\in \cP_i$ for some $i \geq 2+\lceil\frac{\alpha}{\log 2}k-\frac{\log\del}{2\log2}\rceil$ 
such that $y\in a^k\partial_{d_0e^{-k\alpha}\del}P$. 
By Lemma \ref{partitioncst'}, $P \subset K_i \subset Y(2^{-i})\setminus Y(2^{-i+2}) \subset Y(2^{-i+2})^c$.
It follows from \eqref{thicklem} that 
\eqlabel{eqbeforea}{
\partial_{d_0e^{-k\alpha}\del}P \subset B_{d_0e^{-k\alpha}\del}^G P \subset B_{d_0e^{-k\alpha}\del}^G Y(2^{-i+2})^c
\subset Y(2^{-i+2}+d_0e^{-k\alpha}\del)^c.
}
Using \eqref{excEq}, for any $0<r<1,$ $a^k Y(r)^c \subset Y(d_0e^{k\alpha}r)^c.$
Since $e^{k\alpha}2^{-i+2}\leq \del^{\frac{1}{2}}$, combining with \eqref{eqbeforea}, 
\[
a^k\partial_{d_0e^{-k\alpha}\del}P \subset a^k Y(2^{-i+2}+d_0e^{-k\alpha}\del)^c \subset Y((d_0+d_0^2)\del^{\frac{1}{2}})^c.
\]
This proves the claim.
It follows that 
\eqlabel{Edel'est}{
\mu(E_\del') \leq \mu(Y\setminus Y(C_1 \del^{\frac{1}{2}}))
}
where $C_1=d_0 +d_0^{2}$ is a constant depending only on $a$ and $G$.

Next we estimate $\mu(E_\del'')$. 
It follows from the $a$-invariance of $\mu$ that
\eqlabel{Edel''est1}{
\mu(E_\del'') \leq\sum_{k=0}^{\infty}\sum_{i=1}^{1+\lceil\frac{\alpha}{\log 2}k-\frac{\log\del}{2\log2}\rceil}
\mu(\partial_{d_0e^{-k\alpha}\del}\cP_i)= \sum_{i=1}^{\infty}\sum_{k=k_i}^{\infty}\mu(\partial_{d_0e^{-k\alpha}\del}\cP_i),
}
where $k_i\in\bN$ denotes the smallest number of $k$ such that 
$1+\lceil\frac{\alpha}{\log 2}k-\frac{\log\del}{2\log2}\rceil \geq i$. 
Note that $k_i\geq \frac{\log 2}{\alpha}(i-2)+\frac{\log\del}{2\alpha}$.

On the other hand, by Lemma \ref{partitioncst'} we have
\eqlabel{Edel''est2}{\mu(\partial_{d_0e^{-k\alpha}\del}\cP_i)\leq (r_0^{-1}2^{i+4}d_0e^{-k\alpha}\del)^{\frac{1}{2}}
\mu(Y(2^{-i-1})\setminus Y(2^{-i+3}))}
for any $k\geq k_i$, since $d_0e^{-k\alpha}\del\leq d_02^{-i+2}\del^{\frac{1}{2}}<cr_02^{-i-2}$.
Hence, we obtain from \eqref{Edel''est1} and \eqref{Edel''est2}
\eqlabel{Edel''est3}{\begin{aligned}
\mu(E_\del'')&\leq \sum_{i=1}^{\infty}\sum_{k=k_i}^{\infty}\mu(\partial_{d_0e^{-k\alpha}\del}\cP_i)
\leq \sum_{i=1}^{\infty}\sum_{k=k_i}^{\infty}(r_0^{-1}2^{i+4}d_0e^{-k\alpha}\del)^{\frac{1}{2}}
\mu(Y(2^{-i-1})\setminus Y(2^{-i+3}))\\
&=\sum_{i=1}^{\infty}(r_0^{-1}2^{i+4}e^{-k_i\alpha}\del)^{\frac{1}{2}}(1-e^{-\al/2})^{-1}\mu(Y(2^{-i-1})\setminus Y(2^{-i+3}))\\
&\leq r_0^{-\frac{1}{2}}2^{3}\del^{\frac{1}{4}}(1-e^{-\al/2})^{-1}\sum_{i=1}^{\infty}\mu(Y(2^{-i-1})\setminus Y(2^{-i+3}))\leq C_2 \del^{\frac{1}{4}},
\end{aligned}}
where $C_2 = 2^{5}r_0^{-\frac{1}{2}}(1-e^{-\al/2})^{-1}$ is a constant depending only on $r_0$, $a$, and $G$.
%In the last inequality we used the fact that $Y(2^{-i-1})\setminus Y(2^{-i+3})$'s can be overlapped at most four times. 
Combining \eqref{Edel'est} and \eqref{Edel''est3}, we finally have
$$\mu(E_\del)<\mu(Y\setminus Y(C_1\del^{\frac{1}{2}}))+C_2\del^{\frac{1}{4}}$$
and the constants $C_1,C_2>0$ depend only on $r_0$, $a$, and $G$.

It remains to check that $B^{G^+}_\del\cdot y\subset[y]_{\cP_0^\infty}$ for any $y\in Y\setminus E_{\del}$. 
Let $h\in B^{G^+}_{\del}$ and suppose $[hy]_{\cP_0^{\infty}}\neq[y]_{\cP_0^{\infty}}$. 
Then there is some $k\geq 0$ such that $a^{-k}hy$ and $a^{-k}y$ belong to different elements of the partition $\cP$. 
Since $a^{-k}ha^{k}\in a^{-k}B_{\del}^{G^+}a^k \subset B_{d_0e^{-k\alpha}\del}^G$ by \eqref{excEq}, we have
\[d_Y (a^{-k}hy, a^{-k}y) \leq d_G (a^{-k}ha^{k},id) \leq d_0e^{-k\alpha}\del.\]
It follows that both $a^{-k}hy$ and $a^{-k}y$ belong to $\partial_{d_0e^{-k\alpha}\del}\cP$, 
hence $y\in E_\del$.
It concludes that $B^{G^+}_{\del}\cdot y\subset [y]_{\cP_0^{\infty}}$ for any $y\in Y\setminus E_\del$.
\end{proof}

The following proposition is a quantitative version of \cite[Proposition 7.37]{EL}. Given $a$-invariant measure $\mu$, the proposition provides a $\sigma$-algebra which is $a^{-1}$-descending and subordinate to $L$ in the following quantitative sense.

\begin{prop}\label{algebracst}
Let $0<r_0\leq 1$ be given, $\mu$ be an $a$-invariant probability measure on $Y$,
and $L<G^+$ be a closed subgroup normalized by $a$.
There exists a countably generated sub-$\sigma$-algebra $\cA^L$ of Borel $\sigma$-algebra of $Y$ satisfying 
\begin{enumerate}
\item\label{algebracstprop1} $a\cA^L \subset \cA^L$, that is, $\cA^L$ is $a^{-1}$-descending,
\item\label{algebracstprop2} $[y]_{\cA^L}\subset B_{r_02^{-k+1}}^{L}\cdot y$ for any $y\in Y(2^{-k})\setminus Y(2^{-k+2})$ with $k\geq 1$,
\item\label{algebracstprop3} if $0<\del<\min((\frac{cr_0}{16d_0})^{2},1)$, then $B_\del^{L}\cdot y\subset [y]_{\cA^L}$ 
for any $y\in Y(\del)\setminus E_\del$, where $c, d_0>0$ are the constants in Lemma \ref{partitioncst} and \eqref{excEq}, 
and $E_\del$ is the dynamical $\del$-boundary defined in Lemma \ref{Exceptional}.
\end{enumerate} 
In particular, the $\sigma$-algebra $\cA^L$ is $L$-subordinate modulo $\mu$.
\end{prop}

\begin{proof}
For a given $a$-invariant probability measure $\mu$ on $Y$, let $\cP$ be the countable partition of $Y$ constructed in Lemma \ref{Exceptional}. 
We will construct a countably generated $\sigma$-algebra $\cP^L$ by taking $L$-plaques in each $P\in \cP$ as atoms of $\cP^L$.
Then $\cA^L:=(\cP^L)_0^\infty$ will be the desired $\sigma$-algebra.

For each $P\in \cP$, by Lemma \ref{partitioncst'}, there exist $j\geq 1$ and $z\in P$ 
such that $P\in Y(2^{-j})\setminus Y(2^{-j+2})$ and $B_{\frac{1}{5}r_02^{-j-1}}^G \cdot z \subseteq P \subseteq 
B_{r_02^{-j-1}}^G \cdot z$. We can find $B_P \subset G$ with $\diam(B_P)\leq r_02^{-j}$ such that $P = \pi_Y (B_P)$,
where $\pi_Y : G \to Y$ is the natural quotient map.
Let $\cB_{G/L}$ be the Borel $\sigma$-algebra of the quotient $G/L$. Note that since $L$ is closed, $\cB_{G/L}$ is countably generated. 
Define the $\sigma$-algebra
$$\cP^L=\sigma\left(\left\{\pi_Y(B_P \cap S): P\in\cP,\ S\in \cB_{G/L}\right\}\right).$$
Then $\cP^L$ is a refinement of $\cP$ such that atoms of $\cP^L$ are open $L$-plaques, 
i.e. for any $y\in P \in \cP$, $[y]_{\cP^L}=[y]_\cP \cap B_{r_02^{-j}}^L\cdot y=V_y\cdot y$, 
where $V_y\subset B_{r_02^{-j}}^L$ is an open bounded set. 

It is clear that $\cP^L$ is countably generated, hence $\cA^L=(\cP^L)_0^\infty$ is also countably generated. 
By construction, we have $a\cA^L=(\cP^L)_1^\infty\subset \cA^L$, which proves the assertion \eqref{algebracstprop1}. 

For any $y\in Y(2^{-k})\setminus Y(2^{-k+2})$ with $k\geq 1$, take $P\in\cP$ such that $y\in P$.  
By Lemma \ref{partitioncst'}, there exist $j\geq 1$ and $z\in P$ such that 
$P\in Y(2^{-j})\setminus Y(2^{-j+2})$ and $P \subseteq B_{r_02^{-j-1}}^G \cdot z$.
Observe that $2^{-j+2} > 2^{-k}$ and $2^{-j} < 2^{-k+2}$, that is, $j-2 < k < j+2 $. Hence we have 
$$[y]_{\cA^L}\subset [y]_{\cP^L}=V_y\cdot y \subset B_{r_02^{-j}}^L\cdot y \subset B_{r_02^{-k+1}}^L \cdot y,$$
which proves the assertion \eqref{algebracstprop2}.

For a given $0<\del<\min((\frac{cr_0}{16d_0})^{2},1)$ and $y\in Y(\del)\setminus E_\del$, assume that $z=hy$ with $h\in B_{\del}^L$. By Lemma \ref{Exceptional}, 
$B_\del^{G^+}\cdot y\subset[y]_{\cP_0^\infty}$. Hence it follows that for any $k\geq 0$, 
$a^{-k}y$ and $a^{-k}z$ belong to the same atom $P_k \subset \cP$. Then we have
\[
a^{-k}y,\; a^{-k}z=a^{-k}ha^{k}\cdot(a^{-k}y)\; \in P_k.
\]
Note that for any $y\in Y(\del)$ the map $B_{\del}^{G^+}\ni g \mapsto g y$ is injective, hence
the map $a^{-k}B_{\del}^{G^+}a^k \ni g \mapsto ga^{-k}y$ is injective.
Since $a^{-k}ha^{k} \in a^{-k}B_{\del}^L a^k$, $a^{-k}y$ and $a^{-k}z$ belong to the same atom of $\cP^L$. 
This proves the assertion \eqref{algebracstprop3}.
\end{proof}

As in \cite[Lemma 3.4]{LSS}, we need to compare the dynamical entropy and the static entropy. 
In \cite{LSS}, the $\sigma$-algebra $\pi^{-1}(\cB_X)$ is used to deal with the entropy relative to $X$, where
$\cB_X$ is the Borel $\sigma$-algebra of $X$. In order to deal with the entropy relative to the general closed subgroup
$L<G^+$ normalized by $a$, we consider the following tail $\sigma$-algebra with respect to $\cA^L$ in 
Proposition \ref{algebracst}: Denote by
\eqlabel{eq:tailalg}{
\cA_\infty^L :=\bigcap_{k=1}^{\infty}a^{k}\cA^L= \bigcap_{k=1}^{\infty}\left(\cP^L\right)_{k}^{\infty}.
}
This tail $\sigma$-algebra may not be countably generated but it satisfies strict $a$-invariance, i.e. $a\cA_\infty^L=\cA_\infty^L=a^{-1}\cA_\infty^L$.
 
\begin{lem}\label{atomfree}
Let $0<r_0\leq 1$ be given, $\mu$ be an $a$-invariant probability measure on $Y$,  
$L<G^+$ be a closed subgroup normalized by $a$, and $\cA^L$ be as in Proposition \ref{algebracst}. 
Then the $\sigma$-algebra $(\cA^L)_{-\infty}^\infty$ is the Borel $\sigma$-algebra of $Y$ modulo $\mu$. 
\end{lem}
\begin{proof}
Let $\cP^L$ be as in the proof of Proposition \ref{algebracst}.
Since $(\cA^L)_{-\infty}^\infty=(\cP^L)_{-\infty}^{\infty}$ and $Y=\bigcup_{k\geq 1}Y(2^{-k})\setminus Y(2^{-k+2})$,
it is enough to show that for each $k\geq 1$ and for $\mu$-a.e. $y\in Y(2^{-k})\setminus Y(2^{-k+2})$, 
we have $[y]_{(\cP^L)_{-\infty}^{\infty}}=\{y\}$.

For fixed $k\geq 1$, it follows from Poincar\'{e} recurrence (e.g. see \cite[Theorem 2.11]{EW}) that for $\mu$-a.e. 
$y\in Y(2^{-k})\setminus Y(2^{-k+2})$, there exists an increasing sequence $(k_i)_{i\geq 1} \subset \bN$ such that 
\[
a^{k_i}y\in Y(2^{-k})\setminus Y(2^{-k+2}) \quad \text{and}\quad k_i \to \infty\ \text{as}\ i\to \infty.
\]
By Proposition \ref{algebracst}\eqref{algebracstprop2}, it follows that for each $i\geq 1$
$$[a^{k_i}y]_{\cA^L}=[a^{k_i}y]_{(\cP^L)_0^{\infty}}\subset B_{r_02^{-k+1}}^L \cdot a^{k_i}y.$$
Since
$[a^{k_i}y]_{(\cP^L)_0^{\infty}}=a^{k_i}[y]_{a^{-k_i}(\cP^L)_0^{\infty}}=a^{k_i}[y]_{(\cP^L)_{-k_i}^{\infty}}$,
using \eqref{excEq}, we have
$$[y]_{(\cP^L)_{-k_i}^{\infty}} \subset a^{-k_i}B_{r_02^{-k+1}}^L \cdot a^{k_i}y =a^{-k_i}B_{r_02^{-k+1}}^L a^{k_i}\cdot y
\subset B^L_{d_0e^{-\alpha k_i}r_02^{-k+1}}\cdot y.$$ 
Taking $i\to\infty$, we conclude that $[y]_{(\cP^L)_{-\infty}^{\infty}}=\{y\}$.
\end{proof}

\begin{prop}\label{algexiA}
Let $0<r_0 \leq 1$ be given, $\mu$ be an $a$-invariant probability measure on $Y$,
$L<G^+$ be a closed subgroup normalized by $a$, $\cA^L$ be as in Proposition \ref{algebracst}, 
and $\cA^L_\infty$ be as in \eqref{eq:tailalg}. Then we have
\eqlabel{entropycompare}{h_\mu(a|\cA_\infty^L)=h_{\mu}(a^{-1}|\cA_\infty^L) = H_{\mu}(\cA^L|a\cA^L).}
Moreover, \eqref{entropycompare} holds for almost every ergodic component of $\mu$.
\end{prop}
\begin{proof}
Let $\cP^L$ be as in the proof of Proposition \ref{algebracst}.
Since $\cP^L$ is countably generated, we can take an increasing sequence of finite partitions $(\cP_k^L)_{k\geq 1}$ of $Y$
such that $\cP_k^L \nearrow \cP^L$.
By Lemma \ref{atomfree}, we have $\cB_Y=(\cP^L)_{-\infty}^{\infty}=\bigvee_{k=1}^{\infty}(\cP_k^L)_{-\infty}^{\infty}$ 
modulo $\mu$, where $\cB_Y$ is the Borel $\sigma$-algebra of $Y$.
It is clear that $(\cP_k^L)_{-\infty}^{\infty}\subseteq (\cP_{k+1}^L)_{-\infty}^{\infty}$ for all $k\in\bN$.
Hence it follow from Kolmogorov-Sina\u{\i} Theorem \cite[Proposition 2.20]{ELW} that
\[
h_{\mu}(a^{-1}|\cA_\infty^L)=\lim_{k\to\infty}h_{\mu}(a^{-1},\cP_k^L|\cA_\infty^L).
\]
Using the future formula \cite[Proposition 2.19 (8)]{ELW}, we have
\[
\lim_{k\to\infty}h_{\mu}(a^{-1},\cP_k^L|\cA_\infty^L)=\lim_{k\to\infty}H_{\mu}(\cP_k^L|(\cP_k^L)_1^\infty \vee \cA_\infty^L)
\]
It follows from monotonicity and continuity of entropy \cite[Proposition 2.10, 2.12, and 2.13]{ELW} that for any fixed $k\geq 1$
\[
\lim_{\ell\to\infty}H_{\mu}(\cP_k^L|(\cP_\ell^L)_1^\infty \vee \cA_\infty^L)\leq H_{\mu}(\cP_k^L|(\cP_k^L)_1^\infty \vee \cA_\infty^L)
\leq \lim_{\ell\to\infty}H_{\mu}(\cP_\ell^L|(\cP_k^L)_1^\infty \vee \cA_\infty^L),
\]
hence we have 
\[
H_{\mu}(\cP_k^L|(\cP^L)_1^\infty \vee \cA_\infty^L)\leq H_{\mu}(\cP_k^L|(\cP_k^L)_1^\infty \vee \cA_\infty^L)
\leq H_{\mu}(\cP^L|(\cP_k^L)_1^\infty \vee \cA_\infty^L).
\]
Taking $k\to\infty$, it follows that
\[
\lim_{k\to\infty}H_{\mu}(\cP_k^L|(\cP_k^L)_1^\infty \vee \cA_\infty^L)
=H_{\mu}(\cP^L|(\cP^L)_1^\infty \vee \cA_\infty^L)= H_{\mu}(\cA^L|a\cA^L),
\]
which concludes \eqref{entropycompare}.

Note that $\cB_Y=(\cP^L)_{-\infty}^{\infty}=\bigvee_{k=1}^{\infty}(\cP_k^L)_{-\infty}^{\infty}$ 
modulo almost every ergodic component of $\mu$. Thus following the same argument as above, we can conclude
\eqref{entropycompare} for almost every ergodic component of $\mu$. 
\end{proof}

The quantity $H_\mu(\cA^L|a\cA^L)$ is called \textit{empirical entropy} and is the average of the \textit{conditional information function}
$$I_\mu(\cA^L|a\cA^L)(x)=-\log \mu_x^{a\cA^L}([x]_\cA),$$
and indeed the \textit{entropy contribution} of $L$ (see \cite[7.8]{EL} for definition). 

\subsection{Effective variational principle}\label{sec2.4}
This subsection is to effectivize the variational principle. 
We first recall the following ineffective variational principle. Combining \cite[Proposition 7.34]{EL} and \cite[Theorem 7.9]{EL}, 
we have the following upper bound of an empirical entropy (or entropy contribution), and the entropy rigidity.
\begin{thm}[\cite{EL}]\label{thmEL}
Let $L<G^{+}$ be a closed subgroup normalized by $a$, and let $\mathfrak{l}$ denote the Lie algebra of $L$. Let $\mu$ be an $a$-invariant ergodic probability measure on $Y$. If $\cA$ is a countably generated sub-$\sigma$-algebra of the Borel $\sigma$-algebra which is $a^{-1}$-descending and $L$-subordinate, then $$H_\mu(\cA|a\cA)\leq \log|\det(Ad_a|_\mathfrak{l})|$$
and equality holds if and only if $\mu$ is $L$-invariant.
\end{thm}

Let $L<G^+$ be a closed subgroup normalized by $a$, $m_L$ be the Haar measure on $L$, and $\mu$ be an $a$-invariant probability measure on $Y$. Let $\cA$ be a countably generated sub-$\sigma$-algebra of Borel $\sigma$-algbera which is $a^{-1}$-descending and $L$-subordinate modulo $\mu$. Note that for any $j\in\bZ_{\geq 0}$, the sub-$\sigma$-algebra $a^j \cA$ is also countably generated, $a^{-1}$-descending, and $L$-subordinate modulo $\mu$.

For $y\in Y$, denote by $V_y \subset L$ the shape of the $\cA$-atom at $y\in Y$ so that $V_y \cdot y=[y]_{\cA}$. It has positive $m_L$-measure for $\mu$-a.e. $y\in Y$ since $\cA$ is $L$-subordinate modulo $\mu$. 
Note that for any $j\in\bZ_{\geq 0}$, we have $[y]_{a^{j}\cA}=a^{j}V_{a^{-j}y}a^{-j}\cdot y$.

As in \cite[7.55]{EL} which is the proof of \cite[Theorem 7.9]{EL}, let us define $\tau_{y}^{a^j \cA}$ for $\mu$-a.e $y\in Y$ to be the normalized push forward of $m_L|_{a^j V_{a^{-j}y}a^{-j}}$ under the orbit map, i.e.,
\[
\tau_{y}^{a^j \cA}=\frac{1}{m_L (a^j V_{a^{-j}y}a^{-j})}m_L|_{a^j V_{a^{-j}y}a^{-j}}\cdot y,
\] 
which is a probability measure on $[y]_{a^j\cA}$. 

The following proposition is an effective version of Theorem \ref{thmEL}. 

\begin{prop}\label{effEL}
Let $L<G^+$ be a closed subgroup normalized by $a$ and $\mu$ be an $a$-invariant ergodic probability measure on $Y$. 
Fix $j\in\bN$ and denote by $J\ge 0$ the maximal entropy contribution of $L$ for $a^j$, that is,
\[ J= \log|\det(Ad_{a^j}|_\mathfrak{l})|.\]
Let $\cA$ be a countably generated sub-$\sigma$-algebra of Borel $\sigma$-algbera which is $a^{-1}$-descending and $L$-subordinate. 
Suppose there exist a measurable subset $K\subset Y$ and a symmetric measurable subset $B\subset L$ such that 
$[y]_{\cA}\subset B\cdot y$ for any $y\in K$. 
Then we have $$H_\mu(\cA|a^{j}\cA)\leq J+\int_Y \log\tau_y^{a^{j}\cA}\big((Y\setminus K) \cup B\Supp\mu\big)d\mu(y).$$
\end{prop}
\begin{proof}
%For $y\in Y$, denote by $V_y\subset L$ the shape of $\cA$-atom so that $V_y\cdot y=[y]_{\cA}$ a.e. 
By for instance \cite[Theorem 5.9]{EL}, for $\mu$-a.e. $y\in Y$, $\mu_y^{a^{j}\cA}$ is a probability measure on $[y]_{a^{j}\cA}=a^{j}V_{a^{-j}y}a^{-j}\cdot y$, and $H_\mu(\cA|a^{j}\cA)$ can be written as
\eq{H_\mu(\cA|a^{j}\cA)=-\int_Y\log \mu_y^{a^{j}\cA}([y]_{\cA})d\mu(y).}
%Let $m_L$ be the Haar measure of $L$. For $\mu$-a.e. $y\in Y$, $\tau_y^{a^{j}\cA}$ is the normalized push-forward of $m_L|_{a^{j}V_{a^{-j}y}a^{-j}}$, that is, $$\tau_y^{a^{j}\cA}=\frac{1}{m_L(a^{j}V_{a^{-j}y}a^{-j})}m_L|_{a^{j}V_{a^{-j}y}a^{-j}}\cdot y.$$ 
Note that $m_L(a^jBa^{-j})=e^{J}m_L(B)$ for any measurable $B\subset L$. Let $$p(y):= \mu_y^{a^{j}\cA}([y]_{\cA}) \qquad\text{and}\qquad p^{Haar}(y):= \tau_y^{a^{j}\cA}([y]_{\cA}).$$ Then we have
$$p^{Haar}(y)=\frac{m_L(V_y)}{m_L(a^jV_{a^{-j}y} a^{-j})}=\frac{m_L(V_y)}{m_L(V_{a^{-j}y})}e^{-J},$$
hence, applying the ergodic theorem, we have $-\int_Y\log p^{Haar}(y)d\mu(y)=J$. 

Now we estimate an upper bound of $H_\mu(\cA|a^{j}\cA)-J$ following the computation in \cite[7.55]{EL}.
% Since
%\eqlabel{entdif}{
%\begin{aligned}
%J-H_\mu(\cA|a^{j}\cA)&=-\int_Y (\log p^{Haar}(y)-\log p(y))\mu(y),
%\end{aligned}
%}
%we will estimate the value of $\log p^{Haar}(y)-\log p(y)$ for each $y\in Y$. 
Following \cite[7.55]{EL}, we can partition $[y]_{a^{j}\cA}$ into a countable union of $\cA$-atoms as follows:
$$[y]_{a^{j}\cA}=\bigcup_{i=1}^{\infty}[x_i]_\cA\cup N_y,$$
where $N_y$ is a null set with respect to $\mu_y^{a^{j}\cA}$. 
Note that $\mu_y^{a^{j}\cA}$ is supported on $\Supp\mu$ for $\mu$-a.e $y$.
Since $B\subset L$ is symmetric, if $x_i\in K\setminus B\Supp\mu$, then 
$[x_i]_{\cA}\subset B\cdot x_i\subset K\setminus\Supp\mu$, hence we have $\mu^{a^{j}\cA}_y([x_i]_\cA)=0$. 
If $x_i \in (Y\setminus K) \cup B\Supp\mu$ and $[x_i]_\cA \nsubset (Y\setminus K)\cup B\Supp\mu$, then
there exists $x_i' \in [x_i]_\cA$ such that $x_i' \in K\setminus B\Supp\mu$, hence $\mu^{a^{j}\cA}_y([x_i]_\cA)=\mu^{a^{j}\cA}_y([x_i']_\cA)=0$.
Thus we denote by $Z$ the set of $x_i$'s in $(Y\setminus K) \cup B\Supp\mu$ such that $[x_i]_\cA \subset (Y\setminus K)\cup B\Supp\mu$.
It follows that
%Then for $\mu$-a.e. $y$,
%\eqlabel{logest}{\begin{aligned}
%\int_Y\log p^{Haar}(y)-\log p(y)d\mu(y)&=\int_Y(\log p^{Haar}(y)-\log p(y))d\mu^{a^{j}\cA}_y (y)\\
%&=\displaystyle\sum_{i=1}^{\infty}\log\left(\frac{\tau^{a^{j}\cA}_y([x_i]_\cA)}{\mu^{a^{j}\cA}_y([x_i]_\cA)}\right)\mu^{a^{j}\cA}_y([x_i]_\cA)\\
%&=\displaystyle\sum_{x_i\in Z}\log\left(\frac{\tau^{a^{j}\cA}_y([x_i]_\cA)}{\mu^{a^{j}\cA}_y([x_i]_\cA)}\right)\mu^{a^{j}\cA}_y([x_i]_\cA).
%\end{aligned}}
%By the convexity of $\log t$ and taking $t_i=\frac{\tau^{a^{j}\cA}_y([x_i]_\cA)}{\mu^{a^{j}\cA}_y([x_i]_\cA)}$ in the last line, we have
%\eqlabel{convex}{
%\begin{aligned}
%\displaystyle\sum_{x_i\in Z}\log(t_i)\mu_y^{a^{j}\cA}([x_i]_{\cA})
%&\leq\log\left(\displaystyle\sum_{x_i\in Z}\tau_y^{a^{j}\cA}([x_i]_{\cA})\right)\\
%&\leq\log \tau_y^{a^{j}\cA}(Z).
%\end{aligned}
%}
%Combining \eqref{entdif}, \eqref{logest}, and \eqref{convex}, we obtain the desired inequality.
\[
\begin{split}
H_{\mu}(\cA|a^j \cA)-J &= -\int_Y \left(\log p(z) - \log p^{Haar}(z) \right)d\mu(z)\\
&= \int_Y \int_Y \left(\log p^{Haar}(z) - \log p(z) \right)d\mu_y^{a^j \cA}(z)d\mu(y)\\
&= \int_Y \sum_{x_i \in Z} \int_{z\in [x_i]_\cA} \left(\log p^{Haar}(z) - \log p(z) \right)d\mu_y^{a^j \cA}(z)d\mu(y)\\
&= \int_Y \sum_{x_i \in Z} \log\left(\frac{\tau^{a^{j}\cA}_y([x_i]_\cA)}{\mu^{a^{j}\cA}_y([x_i]_\cA)}\right)\mu^{a^{j}\cA}_y([x_i]_\cA)d\mu(y)\\
&\leq \int_Y \log\left(\sum_{x_i\in Z}\tau_y^{a^{j}\cA}([x_i]_{\cA})\right) d\mu(y)\\
&\leq \int_Y \log \tau_y^{a^{j}\cA}((Y\setminus K) \cup B \Supp\mu) d\mu(y).
\end{split}
\]
The second last inequality follows from the convexity of the logarithm. 
This proves the proposition.
\end{proof}

In particular, if $\cA$ is of the form $a^{k}\cA^L$ for $k\in\bZ$, then Proposition \ref{effEL} still holds without assuming the ergodicity of $\mu$.

\begin{cor}\label{effELcor}
Let $0<r_0 \leq 1$ be given, $\mu$ be an $a$-invariant probability measure on $Y$,
$L<G^+$ be a closed subgroup normalized by $a$, and $\cA^L$ be as in Proposition \ref{algebracst}. 
Then Proposition \ref{effEL} holds for $\cA$ of the form $a^{k}\cA^L$ for $k\in\bZ$.
\end{cor}
\begin{proof}
Writing the ergodic decomposition $\mu=\int \mu_{z}^{\cE}d\mu(z)$, we have
$$h_\mu(a^j|\cA^L_\infty)=\int h_{\mu_z^\cE}(a^j|\cA^L_\infty)d\mu(z),$$
where $\cA^L_\infty$ is the $\sigma$-algebra as in \eqref{eq:tailalg}.
By Proposition \ref{algexiA}, we also have 
$$H_\mu(\cA^L|a^j\cA^L)=\int H_{\mu_z^\cE}(\cA^L|a^j\cA^L)d\mu(z).$$
It follows from the $a$-invariance of $\mu$ and $\mu_z^\cE$ that 
$$H_\mu(\cA|a^j\cA)=\int H_{\mu_z^\cE}(\cA|a^j\cA)d\mu(z).$$
Applying Proposition \ref{effEL} for each $\mu_{z}^{\cE}$ we obtain
\eq{\begin{aligned}
H_\mu(\cA|a^j\cA)  = \int H_{\mu_z^\cE}(\cA|a^j\cA)d\mu(z)
&\leq J+\int_Y\int_Y \log\tau_y^{a^{j}\cA}(B^2\Supp\mu_z^{\cE})d\mu_z^\cE(y)d\mu(z)\\
&\leq J+\int_Y \log\tau_y^{a^{j}\cA}(B^2\Supp\mu)d\mu(y).
\end{aligned}}
\end{proof}

\section{Preliminaries for the upper bound}\label{sec3}
From now on, we fix the following notations:
$$d=m+n,\; G=\ASL_d(\bR),\; \Ga=\ASL_d(\bZ),\; \text{and}\; Y=G/\Ga.$$
%$\|g-id\| \leq d_G(g,id)$ for $g$ in the sufficiently small ball $B_r^G(id)$,
%where $\|\cdot\|$ is the supremum norm on $M_{d+1,d+1}(\bR)$.
We use all notations in Subsection \ref{sec2.2} with this setting. 
In particular, we choose a right invariant metric $d_G$ on $G$ so that $r_{max} \leq 1$.
Denote by $d_\infty$ the metric on $G$ induced from the max norm on $M_{d+1,d+1}(\bR)$.
Since $d_G$ and $d_\infty$ are locally bi-Lipschitz, there are constants $0<r_0<1$ and $C_0 \geq 1$
such that for any $x,y\in B_{r_0}^G$
\eqlabel{eqbilip}{
\frac{1}{C_0} d_\infty (x,y) \leq d_{G}(x,y)\leq C_0 d_\infty (x,y).
} 
Note that $r_0$ and $C_0$ depend only on $G$. In the rest of the article, all the statements from Lemma \ref{partitioncst'} 
to Proposition \ref{algexiA} will be applied to this $r_0$.

Recall the notations $a_t$, $a=a_1$, $U$, and $W$ in the introduction.
Then the subgroups $U$ and $W$ are closed subgroups in $G^+$ normalized by $a$,
where $G^+$ is the unstable horospherical subgroup associated to $a$. 
Denote by $\mathfrak{u}$ and $\mathfrak{w}$ the Lie algebras of $U$ and $W$, respectively.
We now consider the following quasinorms on $\mathfrak{u}=\bR^{mn}=M_{m,n}(\bR)$ and $\mathfrak{w}= \bR^{m}$: 
For $A\in M_{m,n}(\bR)$ and $b\in\bR^m$, define
$$\|A\|_{\mb{r}\otimes\mb{s}}=\max_{\substack{1\leq i\leq m\\1\leq j\leq n}}
|A_{ij}|^{\frac{1}{r_i +s_j}}\quad \text{and}\quad \|b\|_{\mb{r}}=\max_{1\leq i\leq m}
|b_i|^{\frac{1}{r_i}}.$$ 
We call these quasinorms $\mb{r}\otimes\mb{s}$-quasinorm and $\mb{r}$-quasinorm, respectively.

We remark that for $A, A' \in M_{m,n}(\bR)$ and $b, b' \in\bR^m$,
using the convexity of functions $s\mapsto s^{\frac{1}{r_i + s_j}}$ and $s\mapsto s^{\frac{1}{r_i}}$, 
\eqlabel{quasitriang}{
\begin{split}
\|A+A'\|_{\mb{r}\otimes\mb{s}} &\leq 2^{\frac{1-(r_m+s_n)}{r_m +s_n}} 
(\|A\|_{\mb{r}\otimes\mb{s}}+\|A'\|_{\mb{r}\otimes\mb{s}});\\
\|b+b'\|_{\mb{r}} &\leq 2^{\frac{1-r_m}{r_m}} (\|b\|_{\mb{r}}+\|b'\|_{\mb{r}}).
\end{split}
}
It also satisfies that
$$\|\Ad_{a_t} A\|_{\mb{r}\otimes\mb{s}}=e^t\|A\|_{\mb{r}\otimes\mb{s}}\quad\text{and}
\quad \|\Ad_{a_t} b\|_{\mb{r}}=e^t\|b\|_{\mb{r}},$$
for any $A\in M_{m,n}(\bR)$ and $b\in \bR^m$.

By a \emph{quasi-metric} on a space $Z$, we mean a map $d_Z: Z \times Z \to \mathbb R_{\geq0}$ which is 
a symmetric, positive definite map such that, for some constant $C$, for all $x,y\in Z$, 
$d_Z (x,y)\leq C(d_Z (x,z)+d_Z(z,y))$. 
The $\mb{r}\otimes\mb{s}$-quasinorm (resp. $\mb{r}$-quasinorm) induces 
the quasi-metric $d_{\mb{r}\otimes\mb{s}}$ (resp. $d_{\mb{r}}$) on $\mathfrak{u}$ (resp. $\mathfrak{w}$). 
Note that the logarithm map is defined on $U$ and $W$, hence the quasi-metric $d_{\mb{r}\otimes\mb{s}}$ 
(resp. $d_{\mb{r}}$) induce the quasi-metric on $U$ (resp. $W$) via the logarithm map.
For simplicity, we keep the notations 
$d_{\mb{r}\otimes\mb{s}}$ and $d_{\mb{r}}$ for the quasi-metrics on $U$ and $W$, respectively.
We similary denote by $B^{U,\mb{r}\otimes\mb{s}}_r$ (resp. $B^{W,\mb{r}}_r$)
the open $r$-ball around the identity in $U$ (resp. $W$) with respect to the quasi-metric $d_{\mb{r}\otimes\mb{s}}$ 
(resp. $d_{\mb{r}}$).
For any $y\in Y$, we also denote by $d_{\mb{r}\otimes\mb{s}}$ (resp. $d_{\mb{r}}$) the induced quasi-metric on the fiber $B_{r_y}^U\cdot y$ (resp. $B_{r_y}^W\cdot y$). 

As in Theorem \ref{thmEL}, we can explicitly compute the maximum entropy contribitions for $L=U$ and $W$. For $L=U$, the restricted adjoint map is the expansion $\Ad_a:(A_{ij})\mapsto (e^{r_i+s_j}A_{ij})$ of $A\in M_{m,n}(\bR)$, hence
\eq{\log|\det(Ad_a|_\mathfrak{u})|=\displaystyle\sum_{i=1}^m \sum_{j=1}^n (r_i+s_j)=m+n.} For $L=W$, the restricted adjoint map is the expansion $Ad_a:(b_i)\mapsto (e^{r_i}b_{i})$ of $b\in \bR^m$, hence
\eq{\log|\det(Ad_a|_\mathfrak{w})|=\displaystyle\sum_{i=1}^m r_i=1.}

Denote by $X=\SL_d(\bR)/\SL_d(\bZ)$ and by $\pi:Y\to X$ the natural projection 
sending a translated lattice $x +v$ to the lattice $x$. 
Equivalently, it is defined by $\pi\left(
\left(\begin{matrix}
g & v\\
0 & 1\\
\end{matrix}\right)\Gamma\right)
=g \SL_d(\bZ)$ for $g\in \SL_d(\bR)$ and $v\in \bR^d.$
We also use the following notation:
$ w(v)=\left(\begin{matrix}
I_d & v\\
0 & 1\\
\end{matrix}\right)$ for $v\in \bR^d$.

\subsection{Dimensions}\label{sec3.1}
Let $Z$ be a space endowed with a quasi-metric $d_Z$.
For a bounded subset $S\subset Z$, the lower Minkowski dimension $\underline{\dim}_{d_Z} S$ with respect to the quasi-metric $d_Z$ is defined by
\[ 
\underline{\dim}_{d_Z} S \defn \liminf_{\delta\to 0} \frac{\log N_{\delta}(S)}{\log1/\delta},\]
where $N_{\delta}(S)$ is the maximal cardinality of a $\delta$-separated subset of $S$ for $d_Z$.
%If $S$ is unbounded, we let $\underline{\dim}_{d_Z} S = \sup\{\underline{\dim}_{d_Z} S\cap K\ ;\ K\ \mbox{compact}\}$.

In the begining of this section, we consider Lie algebras $\mathfrak{u}$ and $\mathfrak{w}$ endowed with $\mb{r}\otimes\mb{s}$-quasinorm and $\mb{r}$-quasinorm, which induce the quasi-metrics
$d_{\mb{r}\otimes\mb{s}}$ and $d_{\mb{r}}$ on $\mathfrak{u}$ and $\mathfrak{w}$, repectively.

Now, for subsets $S \subset \mathfrak{u}=\bR^{mn}$ and $S' \subset \mathfrak{w}=\bR^m$, we denote the lower Minkowski dimensions of these subsets as follows:
$$\underline{\dim}_{\mb{r}\otimes\mb{s}} S \defn \underline{\dim}_{d_{\mb{r}\otimes\mb{s}}} S, \qquad \underline{\dim}_{\mb{r}} S' \defn \underline{\dim}_{d_{\mb{r}}} S'.$$
We will also consider Hausdorff dimensions $\dim_H S$ and $\dim_{H} S'$, always defined with respect to the standard metric.
%We refer the reader to \cite{falconer} for general properties of Minkowski or Hausdorff dimensions, such as the inequality

%Following \cite{LSS}, we will relate dimension $\underline{\dim}_M$ to entropy, and further to Hausdorff dimension using $\underline{\dim}_{\mb{r}\otimes\mb{s}}$ and $\underline{\dim}_{\mb{r}}$ via the following lemma.

\begin{lem}\cite[Lemma 2.2]{LSS}\label{relating dimensions}
For subsets $S \subset \mathfrak{u}$ and $S' \subset \mathfrak{w}$,
\begin{enumerate}
\item $\underline{\dim}_{\mb{r}\otimes\mb{s}} \mathfrak{u} = \sum_{i,j}(r_{i}+s_{j}) = m+n$ and $\underline{\dim}_{\mb{r}} \mathfrak{w} = \sum_{i}r_{i}=1$,
\item $\underline{\dim}_{\mb{r}\otimes\mb{s}} S \geq (m+n) - (r_1 +s_1)(mn - \dim_H S)$,
\item $\underline{\dim}_{\mb{r}} S' \geq 1 - r_1 (m-\dim_H S' )$.
\end{enumerate}
\end{lem}

\subsection{Correspondence with dynamics}\label{sec3.2} 
For $y=\left(\begin{matrix} g & v \\ 0 & 1 \\ \end{matrix}\right)\Ga \in Y$ with $g\in \SL_d(\bR)$ and 
$v\in \bR^d$, denote by $\Lambda_y$ the corresponding unimodular grid $g\bZ^d + v$ in $\bR^d$.
We denote the $(\mathbf{r},\mathbf{s})$-quasinorm of $v=(\bx,\by)\in \bR^m\times\bR^n$ by 
$\|v\|_{\mathbf{r},\mathbf{s}}=\max\{\|\bx\|_{\mathbf{r}}^{\frac{d}{m}},\|\by\|_{\mathbf{s}}^{\frac{d}{n}}\}$.
Let
$$\cL_\eps\defn\set{y\in Y : \forall v\in\Lambda_y, \|v\|_{\mathbf{r},\bs}\geq\eps},$$
which is a (non-compact) closed subset of $Y$. 
Following \cite[Section 1.3]{Kle99}, we say that the pair $(A,b)\in M_{m,n}(\bR)\times \bR^m$ is 
\textit{rational} if there exists some $(p,q) \in \bZ^m\times \bZ^n$ such that 
$Aq-b+p=0$, and \textit{irrational} otherwise.

\begin{prop}\label{prop1}
For any irrational pair $(A,b)\in M_{m,n}(\bR)\times \bR^m$, 
$(A,b)\in \mb{Bad}(\eps)$ if and only if the $a_t$-orbit of the point $y_{A,b}$ is eventually in $\cL_\eps$, 
i.e., there exists $T\ge 0$ such that $a_t y_{A,b}\in \cL_\eps$ for all $t\ge T$.
\end{prop}

\begin{proof}
Suppose that there exist arbitrarily large $t$'s satisfying $a_t y_{A,b}\notin \cL_\eps$. Denote
$e^{\mathbf{r} t}:= \textrm{diag}(e^{r_1t},\cdots,e^{r_mt})\in M_{m,m}(\bR)$ and 
$e^{\mathbf{s}t}:=\textrm{diag}(e^{s_1t},\cdots,e^{s_nt})\in M_{n,n}(\bR).$ 
%As
%$$a_t x_{A,b}=
%\left(\begin{matrix}
%e^{\mathbf{r} t}I_m & 0 & 0 \\
%0 & e^{-\bs t}I_n & 0\\
%0 & 0 & 1\\
%\end{matrix}\right)
%y_{A,b}=
%\left(\begin{matrix}
%e^{\mathbf{r} t}I_m & e^{\mathbf{r} t}A & -e^{\mathbf{r} t}b\\
%0 & e^{-\bs t}I_n & 0\\
%0 & 0 & 1\\
%\end{matrix}\right)y_{O, 0},
%$$
Then the vectors in the grid $\Lambda_{a_t y_{A,b}}$ can be represented as 
$$a_t \left(\left(\begin{matrix} I_m & A \\ 0 & I_n \\ \end{matrix}\right) 
\left(\begin{matrix} p \\ q \\ \end{matrix}\right)
+\left(\begin{matrix} -b \\ 0 \\ \end{matrix}\right)\right)
=\left(\begin{matrix} e^{\mathbf{r} t}(Aq+p-b)\\ e^{-\bs t}q\\ \end{matrix}\right)$$
for $(p,q)\in\bZ^m\times\bZ^n$.
Therefore $a_t x_{A,b}\notin \cL_\eps$ implies that for some $q\in \bZ^n$, 
\eqlabel{eqDani}{e^{t}\idist{Aq-b}_{\mathbf{r}}<\eps^{\frac{m}{d}} \quad 
\text{and}\quad e^{-t}\|q\|_{\bs}<\eps^{\frac{n}{d}},}
thus $\|q\|_{\bs}\idist{Aq-b}_{\mathbf{r}}<\eps$. Since $\idist{Aq-b}_\mb{r}\neq0$ for all $q$, we use the condition $\idist{Aq-b}_{\mathbf{r}}<e^{-t}\eps^{\frac{m}{d}}$ for arbitrarily large $t$ to conclude that $\|q\|_{\bs}\idist{Aq-b}_{\mathbf{r}}<\eps$ holds for infinitely many $q$'s. This is a contradiction to the assumption that $(A,b)\in \mb{Bad}(\eps)$.

On the other hand, if $(A,b)\notin \Bad(\eps)$, then since $(A,b)$ is irrational, there are infinitely many $q\in\bZ^n$ such that $\|q\|_{\bs}\idist{Aq-b}_{\mathbf{r}}<\eps$. Thus we can choose arbitrarily large $t$ so that \eqref{eqDani} hold, which contradicts to the assumption that the $a_t$-orbit of the point $y_{A,b}$ is eventually in $\cL_\eps$.
\end{proof}

\begin{rem}\label{rmkdim}
We claim that for a fixed $b\in \bR^m$, the subset $\mb{Bad}_{0}^b(\eps)$ of rational $(A,b)$'s in $\mb{Bad}^b(\eps)$ 
is a subset of $\mb{Bad}^0(\eps).$ Indeed, if $A\in \mb{Bad}^b(\eps)$ for some $b$ and 
$(A,b)$ is rational, then $\idist{Aq_0-b}_{\mathbf{r}}=0$ for some $q_0\in\bZ^m$ and 
$\displaystyle\liminf_{\|q\|_{\mb{s}}\to \infty} \|q\|_{\mb{s}}\idist{Aq-b}_{\mathbf{r}}\ge \eps$, 
thus $\displaystyle\liminf_{\|q\|_{\mb{s}}\to \infty} \|q\|_{\bs}\idist{A(q-q_0)}_{\mathbf{r}}\ge \eps$. Therefore, we have 
$$\dim_H\mb{Bad}_{0}^{b}(\eps)\leq\dim_H\mb{Bad}^0(\eps)
=mn-c_{m,n}\frac{\eps}{\log 1/\eps}<mn$$
for some constant $c_{m,n}>0$ \cite{KM19}. 
For a fixed $A\in M_{m,n}(\bR)$, the subset of $\mb{Bad}_A(\eps)$ such that $(A,b)$ is rational is of the form $Aq+p$ for some $q,p\in\bZ^m$ thus has Hausdorff dimension zero.
\end{rem}
In the rest of the article, we will focus on the elements $y_{A,b}$ that are eventually in $\cL_\eps$.

\subsection{Covering counting lemma}\label{sec3.3}

To construct measures of large entropy in Proposition \ref{prop5} and Proposition \ref{prop2}, we will need 
the following counting lemma, which is a generalization of \cite[Lemma 2.4]{LSS}. 

Here, we consider two cases: $L=U$ and $L=W$. 
%Fix a standard basis $\{e_i: i=1,\dots, \dim \mathfrak{l}\}$ on $\mathfrak{l}$.
Denote by $\mb{c}=(c_1,\dots,c_{\dim \mathfrak{l}})$ either $\mb{r}\otimes\mb{s}$ (for $L=U$) and $\mb{r}$ (for $L=U$),
and denote by $\|\cdot\|_{\mb{c}}$ either $\|\cdot\|_{\mb{r}\otimes\mb{s}}$ (for $L=U$) and 
$\|\cdot\|_{\mb{r}}$ (for $L=W$). 
Let $J_L$ be the maximal entropy contribution for $L$. Recall that $J_U=m+n$ and $J_W=1$. 
%For $y\in Y$, let $r>0$ be smaller than injectivity radius at $y$. We will use the same notations $d_L$ and $d_{\mb{c}}$ on 
%$B_r^L \cdot y$ induced from $d_L$ and $d_{\mb{c}}$ on $L$, respectively.

Before stating the main result of this subsection, we fix the following notations.
Fix a ``cusp part'' $Q_\infty^0\subset X$ that is a connected subset such that 
$X\smallsetminus Q_\infty^0$ has compact closure. 
Set $Q_{\infty} = \pi^{-1}(Q_\infty^0)$ and denote by $r(Q_\infty)>0$ the infimum of injectivity radius on $Y\smallsetminus Q_\infty$. 
For any $D>J_L$, choose large enough $T_{D}\in\bN$ such that for all $i=1,\dots,\dim\mathfrak{l}$,
\eqlabel{eqct}{
\lceil e^{c_{i}T_{D}} \rceil \leq e^{c_{i}T_{D}}e^{\frac{D-J_L}{\dim \mathfrak{l}}}.
} 
For $r_0>0$ and $C_0 \geq 1$ from \eqref{eqbilip},
fix $0<r_D=r_D(Q_\infty^0)<\min(r_0,1/2)$ small enough so that
\eqlabel{eqsmallrad}{
B_{2^{\frac{1}{\min\mb{c}}}C_0 r_D^{\frac{1}{\max\mb{c}}}T_D}^{L,\mb{c}} \subset B_{\min(r_0,\frac{1}{2}r(Q_\infty))}^{L} \quad \text{and}\quad 
B_{r_D}^G (Y\smallsetminus Q_\infty) \subset Y(\frac{1}{2}r(Q_\infty)).
} 

\begin{lem}\label{CovLem}
For any $D>J_L$, we fix the above notations. Let $y\in Y\smallsetminus Q_\infty$ and $I=\{t\in\bN\ |\ a_ty\in Q_\infty\}$.
For any non-negative integer $T$, let
\[ E_{y,T} = \{ z\in B_{r_D}^{L}\cdot y \ |\ 
\forall t\in\{1,\dots,T\}\smallsetminus I,\, d_Y(a_ty,a_tz)\leq r_D\}.\]
The set $E_{y,T}$ can be covered by $Ce^{D|I\cap\{1,\dots,T\}|}$ $d_{\mb{c}}$-balls of radius 
$r_D^{\frac{1}{\max\mb{c}}}e^{-T}$,
where $C$ is a constant depending on $Q_\infty^0$ and $D$, but independent of $T$.
\end{lem}
\begin{proof}
For $s\in\{0,\dots, T_{D}-1\}$ and $k\in\bZ_{\geq 0}$, let us denote by $I_{s,k}(T_{D})= \{s,s+T_{D},\dots,s+kT_{D}\}$ and
\[
E_{y,k}^{s}= \{z\in B_{r_D}^{L}\cdot y : \forall t \in I_{s,k}(T_{D}) \smallsetminus I, d_{Y}(a_{t}y,a_{t}z)\leq r_D \}.
\]
Following the proof of \cite[Lemma 2.4]{LSS} with $E_{y,k}^{s}$ instead of $E_{y,T}$, we obtain the following claim:\vspace{0.3cm}\\
\textbf{Claim}\; The set $E_{y,k}^{s}$ can be covered by $C_{s}e^{(J_L(T_{D}-1)+D)|I\cap I_{s,k}(T_{D})|}$ $d_{\mb{c}}$-balls
of radius $C_0 r_D^{\frac{1}{\max\mb{c}}}e^{-(s+kT_{D})}$, where $C_{s}$ is a constant depending on 
$Q_\infty^0$, $D$ and $s$, but independent of $k$. 
\begin{proof}[Proof of Claim]
We prove the claim by induction on $k$. Since the number of $d_{\mb{c}}$-balls of radius $C_0 r_D^{\frac{1}{\max\mb{c}}} e^{-s}$ 
needed to cover $B_{r_D}^L \cdot y$ is bounded by a constant $C_s$ depending on $Q_\infty^0$, $D$ and $s$, the claim holds for $k=0$.

Suppose that $E_{y,k-1}^s$ can be covered by $N_{k-1}=C_s e^{(J_L(T_{D}-1)+D)|I\cap I_{s,k-1}(T_{D})|}$ 
$d_{\mb{c}}$-balls $\{B_j:j=1,\dots, N_{k-1}\}$ of radius $C_0 r_D^{\frac{1}{\max\mb{c}}}e^{-(s+(k-1)T_{D})}$. 
By the inequality \eqref{eqct}, any $d_{\mb{c}}$-ball of radius 
$C_0 r_D^{\frac{1}{\max\mb{c}}}e^{-(s+(k-1)T_D)}$ can be covered by 
\[
\begin{split}
\prod_{i=1}^{\dim\mathfrak{l}} \left\lceil  \frac{e^{-(s+(k-1)T_D) c_i }}{e^{-(s+kT_D) c_i }} \right\rceil
&= \prod_{i=1}^{\dim\mathfrak{l}}\lceil e^{T_D c_i} \rceil 
\leq \prod_{i=1}^{\dim\mathfrak{l}} e^{c_{i}T_{D}}e^{\frac{D-J_L}{\dim \mathfrak{l}}}\\
&= e^{J_L T_D}e^{D-J_L}= e^{J_L(T_D-1)+D},
\end{split}
\] $d_{\mb{c}}$-balls of radius $C_0 r_D^{\frac{1}{\max\mb{c}}}e^{-(s+kT_D)}$. Thus if $s+kT_D \in I$, 
then $E_{y,k}^s$ can be covered by $N_k=e^{J_L(T_D-1)+D}N_{k-1}$ $d_{\mb{c}}$-balls of radius 
$C_0 r_D^{\frac{1}{\max\mb{c}}}e^{-(s+kT_D)}$.

Suppose that $s+kT_D \notin I$. 
Since $E_{y,k}^s \subset E_{y,k-1}^s$, the set $\{E_{y,k}^s \cap B_j : j=1,\dots, N_{k-1}\}$ covers $E_{y,k}^s$. 
We now claim that for any $x_1,x_2\in E_{y,k}^s \cap B_j$
$$d_{L,\mb{c}}(x_1,x_2)\leq 2^{\frac{1}{\min\mb{c}}}C_0 r_D^{\frac{1}{\max\mb{c}}}e^{-(s+kT_D)}.$$
Indeed, since $B_j$ is a $d_{L,\mb{c}}$-ball of radius $C_0 r_D^{\frac{1}{\max\mb{c}}}e^{-(s+(k-1)T_{D})}$
and $x_1,x_2 \in B_j \subset B_{r_D}^L \cdot y$, there are $h\in B_{r_D}^L$ and 
$h_1,h_2\in B_{C_0 r_D^{\frac{1}{\max\mb{c}}}e^{-(s+(k-1)T_{D})}}^{L,\mb{c}}$ such that $x_1=h_1 hy$ and $x_2=h_2 hy$.
It follows from $s+kT_d \notin I$ and $x_1,x_2\in E_{y,k}^s$ that $a^{s+kT_d}y \subset Y \smallsetminus Q_\infty$ and 
$d_Y(a^{s+kT_D} y, a^{s+kT_D} x_\ell)\leq r_D$ for $\ell=1,2$, hence by \eqref{eqsmallrad} we have 
$a^{s+kT_D} x_1 \in B_{r_D}^G (Y \smallsetminus Q_\infty) \subset Y(\frac{1}{2} r(Q_\infty))$
and $d_Y(a^{s+kT_D} x_1, a^{s+kT_D} x_2)\leq 2r_D$.
Observe that by \eqref{eqsmallrad},
\[
\begin{split}
a^{s+kT_D}h_1 h_2^{-1} a^{-(s+kT_D)} &\subset a^{s+kT_D}B_{2^{\frac{1}{\min\mb{c}}}C_0 r_D^{\frac{1}{\max\mb{c}}}e^{-(s+(k-1)T_{D})}}^{L,\mb{c}}a^{-(s+kT_D)}\\
&= B_{2^{\frac{1}{\min\mb{c}}}C_0 r_D^{\frac{1}{\max\mb{c}}}e^{T_{D}}}^{L,\mb{c}} \subset B_{\min(r_0,\frac{1}{2}r(Q_\infty))}^L.
\end{split}
\]
Thus it follows from \eqref{eqbilip} and above observations that
\[
\begin{split}
2r_D \geq d_Y (a^{s+kT_D}x_1,a^{s+kT_D}x_2) &= d_L(a^{s+kT_D}h_1h_2^{-1}a^{-(s+kT_D)},id) \\
&\geq \frac{1}{C_0} d_\infty (a^{s+kT_D}h_1h_2^{-1}a^{-(s+kT_D)},id)\\
&= \frac{1}{C_0} \max_{i=1,\dots,\dim\mathfrak{l}} e^{c_i(s+kT_D)} |(\log h_1h_2^{-1})_i|, 
\end{split}
\]
where $(\log h_1h_2^{-1})_i$ is the $i$-th coordinate of $\log h_1h_2^{-1}$ with respect to the standard basis $\{e_i: 1\leq i\leq \dim\mathfrak{l}\}$.
Since $L=U$ or $L=W$, i.e. commutative subgroups of $G$, for each $i=1,\dots,\dim\mathfrak{l}$, we have
\[
|(\log h_1h_2^{-1})_i|=|(\log h_1 - \log h_2)_i|\leq 2r_D C_0 e^{-c_i(s+kT_D)}.
\]
Note that
\[
d_{L,\mb{c}}(x_1,x_2)=d_{L,\mb{c}}(h_1,h_2) =\max_{i=1,\dots,\dim\mathfrak{l}} |(\log h_1 - \log h_2)_i|^{\frac{1}{c_i}}
\]
Therefore, we have
\[
d_{L,\mb{c}}(x_1,x_2) \leq  \max_{i=1,\dots,\dim\mathfrak{l}} (2r_D C_0)^{\frac{1}{c_i}}e^{-(s+kT_D)}\leq 
2^{\frac{1}{\min\mb{c}}}C_0 r_D^{\frac{1}{\max\mb{c}}}e^{-(s+kT_D)}. 
\]
It follows from the claim that $E_{y,k}^s \cap B_j$ is contained in a single $d_{L,\mb{c}}$-ball of radius $C_0 r_D^{\frac{1}{\max\mb{c}}} e^{-(s+kT_D)}$ 
for each $j=1,\dots,N_{k-1}$.
Hence $E_{y,k}^s$ can be covered by $N_k=N_{k-1}$ $d_{L,\mb{c}}$-balls of radius $C_0 r_D^{\frac{1}{\max\mb{c}}} e^{-(s+kT_D)}$.
\end{proof}

Now, for any non-negative integer $T$, we can find $s\in\{0,\dots,T_{D}-1\}$ and $k\in\bZ_{\geq 0}$ such that 
\eq{
T_{D}|I\cap I_{s,k}(T_{D})| \leq |I\cap \{1,\dots,T\}|\quad\text{and}\quad T-T_{D}<s+kT_{D}\leq T
} from the pigeon hole principle. By the above observation, $E_{y,T}\subset E_{y,k}^{s}$ can be covered by 
$C_{s}e^{(J_L(T_{D}-1)+D)|I\cap I_{s,k}(T_{D})|}$ $d_{\mb{c}}$-balls of radius $C_0 r_D^{\frac{1}{\max\mb{c}}} e^{-(s+kT_D)}$. 
Since $T-T_{D}<s+kT_{D}\leq T$ and $D>J_L$, $E_{y,T}$ can be covered by 
$(\max_{s}C_s) e^{D|I\cap \{1,\dots.T\}|}$ $d_{\mb{c}}$-balls of radius $C_0 e^{T_D} r_D^{\frac{1}{\max\mb{c}}} e^{-T}$. 
Hence there exists a constant $C>0$ depending on $Q_\infty^0$, $r$, and $D$, but independent of $T$
such that $E_{y,T}$ can be covered by 
$C e^{D|I\cap \{1,\dots.T\}|}$ $d_{\mb{c}}$-balls of radius $r_D^{\frac{1}{\max\mb{c}}} e^{-T}$.
\end{proof}

\section{Upper bound for Hausdorff dimension of $\mb{Bad}_{A}(\eps)$}\label{sec:entropyboundA}
In this section, we will prove Theorem \ref{thmEff1} by constructing $a$-invariant probability measure on $Y$ with large entropy. 
Here and in the next section, we will consider the dynamical entropy of $a$ instead of $a^{-1}$ contrary to Section \ref{sec2}.
Hence let us use the following notation:
For a given partition $\cQ$ of $Y$ and a integer $q\geq 1$, we denote by
\[
\cQ^{(q)}= \bigvee_{i=0}^{q-1} a^{-i}\cQ.
\]
\subsection{Constructing measure with entropy lower bound}\label{sec4.1}
Let us denote by $\ov{X}$ and $\ov{Y}$ the one-point compactifications of $X$ and $Y$, respectively.
Let $\cA$ be a given countably generated $\sigma$-algebra of $X$ or $Y$. We denote by $\overline{\cA}$ the $\sigma$-algebra generated by $\cA$ 
and $\set{\infty}$. The diagonal action $a_t$ is extended to the action on $\ov{X}$ and $\overline{Y}$ by $a_t(\infty)=\infty$ for $t\in\bR$. 
For a finite partition $\cQ=\set{Q_1,\cdots,Q_N,Q_\infty}$ of $Y$ which has only one non-compact element $Q_\infty$, denote by $\overline{\cQ}$ 
the finite partition $\set{Q_1, \cdots, Q_N, \overline{Q_\infty}\defn Q_\infty\cup\set{\infty}}$ of $\overline{Y}$. 
Note that $\overline{\cQ^{(q)}}=\overline{\cQ}^{(q)}$ for any $M\in\N$.
Denote by $\crly{P}(X)$ the space of probability measures on $X$, and use similar notations for $Y$, $\ov{X}$, and $\ov{Y}$.

In this subsection, we construct an $a$-invariant measure on $\overline{Y}$ with a lower bound on the conditional entropy for the proof of 
Theorem \ref{thmEff1}. Here, the conditional entropy will be computed with respect to the $\sigma$-algebras constructed in Section \ref{sec2}. 
If $x_A$ has no escape of mass, such measure was constructed in \cite[Proposition 2.3]{LSS}. 
The following proposition generalizes the measure construction for $x_A$'s with some escape of mass. 

\begin{prop}\label{prop5} For $A\in M_{m,n}(\R)$ fixed, 
let $$\eta_A=\sup\set{\eta:x_A \ \textrm{has} \ \eta\textrm{-escape of mass on average}}.$$ 
Then there exists $\mu_A\in\crly{P}(\overline{X})$ with $\mu_A(X)=1-\eta_A$ such that 
for any $\eps>0$, there exists an $a$-invariant measure $\ov{\mu}\in\crly{P}(\overline{Y})$ satisfying
\begin{enumerate}
  %  \item\label{ainv} $\mu$ is $a_t$-invariant,
    \item\label{supp'} $\Supp{\ov{\mu}}\subset \cL_\eps\cup (\overline{Y}\smallsetminus Y)$,
    \item\label{cusp'} $\pi_*\ov{\mu}=\mu_A$, in particular, there exists $a$-invariant measure $\mu\in\crly{P}(Y)$ such that 
    $$\ov{\mu}=(1-\eta_A)\mu+\eta_A\del_{\infty},$$ where $\del_\infty$ is the dirac delta measure on 
    $\overline{Y}\smallsetminus Y$.
    \item\label{entropy'} Let $\cA^W$ be as in Proposition \ref{algebracst} for $\mu$, $r_0$, and $L=W$, 
    and let $\cA^W_\infty$ be as in \eqref{eq:tailalg}.
Then we have
    $$h_{\ov{\mu}}(a|\overline{\cA^W_\infty})\ge 1-\eta_A-r_{1}(m- \dim_H \mb{Bad}_A(\eps)).$$
\end{enumerate}
\end{prop}
\begin{rem}\
\begin{enumerate}
\item Note that if $\eta_A >0$ then $x_{A}$ has $\eta_A$-escape of mass on average.
\item One can check that $\eta_A=0$ if and only if $x_{A}$ is \textit{heavy}, which is defined in \cite[Definition 1.1]{LSS}.
\end{enumerate}
\end{rem}

\begin{proof}
Since $x_{A}$ has $\eta_A$-escape of mass on average but no more than $\eta_A,$ we may fix an increasing sequence of integers 
$\set{k_i}_{i\geq 1}$ such that
$$\frac{1}{k_i}\displaystyle\sum_{k=0}^{k_i-1}\del_{a^k x_A}\wstar\mu_A\in\crly{P}(\overline{X})$$ with $\mu_A(X)= 1-\eta_A$.

Let us denote by $\bT^{m}= [0,1]^{m}/\!\sim$ the torus in $\bR^{m}$, where the equivalence relation is modulo $1.$ 
%\seon{should it be modulo the lattice $x_A$?} 
Consider the increasing family of sets
$$R^{A,T}:=\set{b\in \bT^m|\forall t\ge T,a_t y_{A,b}\in\cL_\eps}\cap\mb{Bad}_A(\eps).$$ 

By Proposition \ref{prop1} and Remark \ref{rmkdim}, $\displaystyle\bigcup_{T=1}^\infty R^{A,T}$ has Hausdorff dimension equal to $\dim_H \mb{Bad}_A(\eps)$. 
For any $\gamma >0$, it follows that there exists $T_\gamma \in \bN$ satisfying 
$\dim_H R^{A,T_\gamma}\geq \dim_H \mb{Bad}_A(\eps)-\gamma$.

Let $\phi_A:\bT^{m}\to Y$ be the map defined by $\phi_A(b)=y_{A,b}$. Note that $\phi_A$ is a one-to-one Lipschitz map between 
$\bT^m$ and $\phi_A(\bT^m)$, so we may consider a quasinorm on $\phi_A(\bT^m)$ induced from the $\mb{r}$-quasinorm on $\bR^m$ 
and denote it again by $\|\cdot\|_{\mb{r}}$.
%\seon{One-to-one if the torus is defined modulo $A$?}

For each $k_i\ge T_\gamma$, let $S_i$ be a maximal $e^{-k_i}$-separated subset of $R^{A,T_\gamma}$ with respect to the $\mb{r}$-quasinorm. 
By Lemma \ref{relating dimensions}(3),
\eqlabel{Sicount}{\displaystyle\liminf_{i\to\infty}\frac{\log|S_i|}{k_i}\ge \underline{\dim}_{\mb{r}}(R^{A,T_\gamma})\ge 1-r_{1}(m+\gamma- \dim_H \mb{Bad}_A(\eps)).}
Let $\nu_i\defn\frac{1}{|S_i|}\displaystyle\sum_{b\in S_i}\del_{y_{A,b}}$ be the normalized counting measure on the set 
$D_i:=\set{y_{A,b}: b\in S_i}\subset Y$. Extracting a subsequence if necessary, we may assume 
that $$\mu_i\defn\frac{1}{k_i}\displaystyle\sum_{k=0}^{k_i-1}a_*^k\nu_i\wstar\mu^\gamma\in\crly{P}(\overline{Y}).$$
The measure $\mu^{\gamma}$ is $a$-invariant since $a_* \mu_i -\mu_i$ goes to zero measure. 

Choose any sequence of positive real numbers $(\ga_j)_{j\geq 1}$ converging to zero and let $\set{\mu^{\gamma_j}}$ be a family of $a$-invariant probability measures on $\overline{Y}$ obtained from the above construction for each $\gamma_j$.
Extracting a subsequence again if necessary, we may take a $\text{weak}^*$-limit measure $\ov{\mu}\in \crly{P}(\overline{Y})$ of $\set{\mu^{\gamma_j}}$.
We prove that $\ov{\mu}$ is the desired measure. The measure $\ov{\mu}$ is clearly $a$-invariant.\\
%\seon{Explain why?}\\
(\ref{supp'}) We show that for all $\gamma>0$, $\mu^\gamma(Y\setminus\cL_\eps)=0$. For any $b\in S_i\subseteq R^{A,T_\gamma}$, $a_T y_{A,b}\in \cL_\eps$ holds for $T>T_\gamma$. Thus we have
\eq{
\mu_i(Y\setminus\cL_\eps)=\frac{1}{k_i}\displaystyle\sum_{k=0}^{k_i-1}a^k_{*}\nu_i(Y\setminus\cL_\eps)= \frac{1}{k_i}\displaystyle\sum_{k=0}^{T_\gamma}a^k_{*}\nu_i(Y\setminus\cL_\eps)
= \frac{1}{k_{i} |S_{i}|}\sum_{y\in D_{i}, 0\leq k \leq T_{\gamma}} \de_{a^{k}y}(Y\setminus\cL_\eps) \leq \frac{T_\gamma}{k_i}.
}
By taking $k_i\to \infty$, we have $\mu^{\ga}(Y\setminus\cL_\eps)=0$ for arbitrary $\ga>0$, hence
$$\ov{\mu}(Y\setminus\cL_\eps)=\lim_{j\to\infty}\mu^{\gamma_j}(Y\setminus\cL_\eps)=0.$$\\
(\ref{cusp'}) For all $\ga>0$, $\pi_*\mu^\gamma=\mu_A$ since $\pi_*\nu_i=\del_{x_{A}}$ for all $i\ge 1$. 
It follows that $\pi_*\ov{\mu}=\mu_A$. Hence,
$$\ov{\mu}(\overline{Y}\setminus Y)=\lim_{j\to\infty}\mu^{\gamma_j}(\overline{Y}\setminus Y)=\mu_A(\overline{X}\setminus X)=\eta_A,$$
so we have a decomposition $\overline{\mu}=(1-\eta_A)\mu+\eta_A\del_\infty$ for some $a$-invariant $\mu\in\crly{P}(Y)$.\\
(\ref{entropy'}) We first fix any $D>J_W =1$ and $Q^0_\infty \subset X$ such that $X\smallsetminus Q_\infty^0$ has compact closure. 
As in \cite[Proof of Theorem 4.2, Claim 2]{LSS}, we can construct a finite partition $\cQ$ of $Y$ satisfying:
%Suppose that $\cQ$ is any finite partition of $Y$ satisfying:
    \begin{itemize}
        \item $\cQ$ contains an atom $Q_\infty$ of the form $\pi^{-1}(Q_\infty^0)$,
        \item $\forall Q\in \cQ\smallsetminus\set{Q_\infty}$, $\diam Q<r_D=r_D(Q_\infty^0)$, where $r_D$ is from \eqref{eqsmallrad},
        \item $\forall Q\in\cQ,\forall j\geq 1,\; \mu^{\gamma_j}(\partial Q)=0$.
    \end{itemize}
Remark that for all $i\geq 1$, $D_i \subset \phi_A(\bT^m)$, which is a compact set in $Y$,
therefore we can choose $Q^0_{\infty}$ so that 
\eqlabel{eqcusp}{
Q_\infty \cap D_i = \varnothing.
}
    We claim that it suffices to show the following statement. For all $q\ge 1$,
    \eqlabel{sttlowbdd}{\frac{1}{q}H_{\ov{\mu}}(\overline{\cQ}^{(q)}|\overline{\cA^W_\infty})\ge 1-r_{1}(m-\dim_H \mb{Bad}_A(\eps))-D\ov{\mu}(\overline{Q_\infty}).}
    Indeed, by taking $q\to \infty$, we have
%    as in \cite[Proof of Theorem 4.2, Claim 2]{LSS}, we can construct a finite partition $\cQ$ of $Y$ satisfying the bullet-requirements above. Hence,
$$h_{\ov{\mu}}(a|\overline{\cA^W})\ge 1-r_{1}(m- \dim_H \mb{Bad}_A(\eps))-D\ov{\mu}(\overline{Q_\infty}).$$
%for any $Q_\infty$ of $\cQ$ satisfying the bullet-requirements. 
Taking $D\to 1$ and $Q_\infty^0\subset X$ such that $\ov{\mu}(\ov{Q_\infty}) \to \ov{\mu}(\overline{Y}\setminus Y)=\eta_A$ and $D\to 1$, we conclude \eqref{entropy'}.

In the rest of the proof, we show the inequality \eqref{sttlowbdd}.
It is clear if $\ov{\mu}(Q_\infty)=1$, so assume that $\ov{\mu}(Q_\infty)<1$, hence for all large enough $j\geq 1$, 
$\mu^{\ga_j}(Q_\infty)<1$. Now, we fix such $j\geq 1$ and write temporarily $\ga=\ga_j$.

Choose $\beta>0$ such that $\mu^\gamma(Q_\infty)<\beta<1$. For large enough $i\geq 1$, we have
\eq{
\mu_i(Q_\infty)=\frac{1}{k_i|S_i|}\displaystyle\sum_{y\in D_i, 0\le k<k_i}\de_{a^k y}(Q_\infty)
=\frac{1}{k_i}\displaystyle\sum_{0\le k<k_i}\de_{a^k x_{A}}(Q^0_\infty) < \beta.
}
In other words, there exist at most $\beta k_i$ number of $a^k x_A$'s in $Q^0_\infty$, thus for any $y\in D_{i}$, we have
\[
|\{k\in\{0,\dots,k_{i}-1\}: a^{k}y \in Q_{\infty}\}| < \beta k_{i}.
\] 
From Lemma \ref{CovLem} with $L=W$ and \eqref{eqcusp}, if $Q$ is any non-empty atom of $\cQ^{(k_i)}$, 
fixing any $y\in D_i\cap Q$, the set
\eq{
D_{i}\cap Q = D_{i}\cap [y]_{\cQ^{(k_i)}} \subset E_{y,k_{i}-1}
} 
can be covered $Ce^{D\beta k_{i}}$ many $r_D^{1/r_1}e^{-k_{i}}$-balls for $d_{\mb{r}}$, 
where $C$ is a constant depending on $Q_\infty^0$ and $D$, 
but not on $k_i$. Since $D_{i}$ is $e^{-k_{i}}$-separated with respect to $d_{\mb{r}}$ and $r_D^{1/r_{1}}<\frac{1}{2}$, we get 
\eqlabel{eq:count}{\mathrm{Card}(D_i \cap Q)\leq Ce^{D\beta k_i }.}

Now let $\cA^W=(\cP^W)_0^\infty=\bigvee_{i=0}^{\infty}a^i \cP^W$ be as in Proposition \ref{algebracst} for $\mu$, 
$r_0$ and $L=W$, and let $\cA^W_\infty$ be as in \eqref{eq:tailalg}.\vspace{0.3cm}\\
\textbf{Claim}\; $H_{\nu_i}(\cQ^{(k_i)}|\cA^W_\infty) = H_{\nu_i}(\cQ^{(k_i)})$ for all large enough $\ell\geq 1$.
\begin{proof}[Proof of Claim]
Using the continuity of entropy, we have 
\eq{H_{\nu_i}(\cQ^{(k_i)}|\cA^W_\infty)=\lim_{\ell\to\infty}H_{\nu_i}(\cQ^{(k_i)}|(\cP^W)_\ell^\infty).}
Now we show $H_{\nu_i}(\cQ^{(k_i)}|(\cP^W)_\ell^\infty) = H_{\nu_i}(\cQ^{(k_i)})$ for all large enough $\ell\geq 1$.
Let $E_\del$ be the dynamical $\del$-boundary of $\cP$ as in Lemma \ref{Exceptional} for $\mu$ and $r_0$.
As mentioned in Remark \ref{basepoint}, we may assume that there exists $y\in \phi_A(\bT^m)$ such that 
$y\notin \partial \cP$. Since $E_\del = \bigcup_{k=0}^\infty a^k \partial_{d_0e^{-k\al}\del}\cP$, 
there exists $\del>0$ such that $y\in Y \setminus E_{\del}$. For any $\ell \geq 1$, we have
$a^{-\ell}y \in Y\setminus a^{-\ell}E_\del \subset Y\setminus E_\del$.
Hence, it follows from \eqref{excEq} and Proposition \ref{algebracst} that
\[
[y]_{(\cP^W)_\ell^\infty} = a^\ell[a^{-\ell}y]_{(\cP^W)_0^\infty} = a^\ell[a^{-\ell}y]_{\cA^W} \supset a^{\ell} B_{\del}^{W} a^{-\ell} y \supset B_{d_0 e^{\alpha\ell}\del}^{W} y.
\]
Since the support of $\nu_i$ is a set of finite points on a single compact $W$-orbit $\phi_A(\bT^m)$, 
$\nu_i$ is supported on a single atom of $(\cP^W)_\ell^\infty$ for all large enough $\ell\geq 1$.
This proves the claim.
\end{proof}
Combining \eqref{eq:count} and \textbf{Claim}, it follows that
\eqlabel{staticbd3}{
H_{\nu_i}(\cQ^{(k_i)}|\cA^W_\infty)= H_{\nu_i}(\cQ^{(k_i)})
 \geq \log |S_i|-D\beta k_i-\log C.
}

For any $q\ge 1$, write the Euclidean division of large enough $k_i-1$ by $q$ as
$$k_i-1=qk'+s \ \textrm{with} \ s\in\set{0,\cdots,q-1}.$$
By subadditivity of the entropy with respect to the partition, for each $p\in\set{0,\cdots,q-1}$,
$$H_{\nu_i}(\cQ^{(k_i)}|\cA^W_\infty)\leq H_{a^{p}\nu_i}(\cQ^{(q)}|\cA^W_\infty)+\cdots+H_{a^{p+qk'}\nu_i}(\cQ^{(q)}|\cA^W_\infty)+2q\log |\cQ|.$$
Summing those inequalities for $p=0,\cdots,q-1$, and using the concave property of entropy with respect to the measure, we obtain
\eq{qH_{\nu_i}(\cQ^{(k_i)}|\cA^W_\infty)\leq\displaystyle\sum_{k=0}^{k_i-1}H_{a^k \nu_i}(\cQ^{(q)}|\cA^W_\infty)_0^M+2q^2\log |\cQ|\leq k_iH_{\mu_i}(\cQ^{(q)}|\cA^W_\infty)+2q^2\log |\cQ|,
}
and it follows from \eqref{staticbd3} that
\eq{\frac{1}{q}H_{\mu_i}(\cQ^{(q)}|\cA^W_\infty)\ge \frac{1}{k_i}H_{\nu_i}(\cQ^{(k_i)}|\cA^W_\infty)-\frac{2q\log |\cQ|}{k_i}
\ge \frac{1}{k_i}\Bigl(\log |S_i|-D\beta k_i -\log C - 2q\log|\cQ|\Bigr).
}
Now we can take $i\to\infty$ because the atoms $Q$ of $\overline{\cQ}$ and hence of $\overline{\cQ}^{(q)}$, satisfy $\mu^\gamma(\partial Q)=0$. Also, the constants $C$ and $|\cQ|$ are independent to $k_i$. Thus we obtain
\[
    \frac{1}{q}H_{\mu^\gamma}(\overline{\cQ}^{(q)}|\overline{\cA^W_\infty})
    \ge 1-r_{1}(m+\gamma- \dim_H \mb{Bad}_A(\eps))-D\beta,
\]
By taking $\beta \to \ov{\mu}(\overline{Q_\infty})$ and $\gamma=\gamma_j \to 0$, 
the inequality \eqref{sttlowbdd} follows. 
%we have 
%$$\frac{1}{q}H_{\mu^\gamma}(\overline{\cQ}^{(q)}|\overline{\cA^W_\infty})
%    \ge  1-r_{1}(m+\gamma- \dim_H \mb{Bad}_A(\eps))-D\mu^{\gamma}(\overline{Q_\infty}).$$
%Recall that $\ga=\ga_j$, and by taking $j\to \infty$ so that $\ga_j \to 0$, we finally have \eqref{sttlowbdd}, i.e.,
%$$\frac{1}{q}H_{\ov{\mu}}(\overline{\cQ}^{(q)}|\overline{\cA^W_\infty})
%    \ge  1-r_{1}(m- \dim_H \mb{Bad}_A(\eps))-D\ov{\mu}(\overline{Q_\infty}).$$
\end{proof}

\subsection{The proof of Theorem \ref{thmEff1}}
In this subsection, we will estimate the dimension upper bound in Theorem \ref{thmEff1} using $a$-invariant measure with large relative entropy constructed in Proposition \ref{prop5} and the effective variational principle in Proposition \ref{effEL}. To use the effective variational principle, we need the following lemma.

For $x\in X$ and $H\geq 1$ we set:
$$\textrm{ht}(x)\defn\sup\set{\|gv\|^{-1}: x=gSL_d(\bZ), v\in\bZ^d\setminus\set{0}},$$
$$X_{\leq H}\defn\set{x\in X: \textrm{ht}(x)\leq H },\quad Y_{\leq H}\defn\pi^{-1}(X_{\le H}).$$
Note that $\textrm{ht}(x) \geq 1$ for any $x\in X$ by Minkowski's theorem and $X_{\leq H}$ and $Y_{\leq H}$ are compact sets for all $H\geq 1$ by Mahler's compact criterion.

\begin{lem}\label{volest}
Let $\cA$ be a countably generated sub-$\sigma$-algebra of Borel $\sigma$-algbera which is $a^{-1}$-descending and $W$-subordinate. Let us fix $y\in Y_{\leq H}$ and suppose that $B^{W,\mb{r}}_{\del}\cdot y\subset [y]_{\cA}\subset B^{W,\mb{r}}_{r}\cdot y$ for some $0<\del<r$. For any $0<\eps<1$, if $j_1\ge \log((2dH^{d-1})^{\frac{1}{r_m}}\del^{-1})$ and $j_2\ge \log((dH^{d-1})^{\frac{1}{s_n}}\eps^{-\frac{n}{d}})$, then $\tau_y^{a^{j_1}\cA}(a^{-j_2}\cL_\eps)\leq 1-e^{-j_1-j_2}r^{-1}\eps^{\frac{m}{d}}$, where $\tau_{y}^{a^{j_1}\cA}$ is as in Subsection \ref{sec2.4}.
\end{lem}

\begin{proof}
For $x=\pi(y)\in X_{\leq H}$, there exists $g\in SL_d(\bR)$ such that $x=gSL_d(\bZ)$ and $\displaystyle\inf_{v\in\bZ^d\setminus\set{0}}\|gv\|\ge H^{-1}$. By Minkowski's second theorem with a convex body $[-1,1]^d$, we can choose vectors $gv_1,\cdots,gv_d$ in $g\bZ^d$ so that $\displaystyle\prod_{i=1}^{d}\|gv_i\|\leq 1$. Then for any $1\leq i\leq d$, $$\|gv_i\|\leq \displaystyle\prod_{j\neq i}\|gv_j\|^{-1} \leq H^{d-1}.$$ Let $\Delta\subset \bR^d$ be the parallelepiped generated by $gv_1,\cdots, gv_d$, then $\|b\|\leq dH^{d-1}$ for any $b\in \Delta$. It follows that $\|b^+\|_{\mb{r}}\leq (dH^{d-1})^{\frac{1}{r_m}}$ and $\|b^-\|_{\mb{s}}\leq (dH^{d-1})^{\frac{1}{s_n}}$ for any $b=(b^+,b^-)\in\Del$, where $b^+\in\bR^m$ and $b^-\in\bR^n$. Note that the set $\pi^{-1}(x)\subset Y$ is parametrized as follows: $$\pi^{-1}(x)=\set{w(b)g\Gamma\in Y: b\in\Del}.$$ Write $y=w(b_0)g\Gamma$ for some $b_0=(b_0^+,b_0^-)\in \Del$. Denote by $V_y\subset W$ the shape of $\cA$-atom so that $V_y\cdot y=[y]_{a^{j_1}\cA}$, and $\Xi\subset\bR^m$ the corresponding set to $V_y$ containing $0$ given by the canonical bijection between $W$ and $\bR^m$. Since $a^{j_1}$ expands the $\mb{r}$-quasinorm with the ratio $e^{j_1}$, we have $B^{W,\mb{r}}_{e^{j_1}\del}\cdot y\subset [y]_{a^{j_1}\cA}\subset B^{W,\mb{r}}_{e^{j_1}r}\cdot y$, i.e.
$B^{\bR^m,\mb{r}}_{e^{j_1}\del}\subset \Xi\subset B^{\bR^m,\mb{r}}_{e^{j_1}r}.$ Then the atom $[y]_{a^{j_1}\cA}$ is parametrized as follows:
$$[y]_{a^{j_1}\cA}=\set{w(b)g \Gamma: b=(b^+,b^-_0), b^+\in b^+_0+\Xi},$$
and $\tau_y^{a^{j_1}\cA}$ can be considered as the normalized Lebesgue measure on the set $b^+_0+\Xi\subset \bR^m$.

\begin{figure}
\begin{tikzpicture}
\filldraw[red!20!white] (-0.2,-2) rectangle (0.2,3);
\draw[red,thick] (-0.2,-2) rectangle (0.2,3);
\draw[white,thick] (-0.2,-2) -- (0.2,-2);
\draw[->] (-2,0) -- (4,0) node[anchor=west] {\tiny{$\bR^{m}$}};
\draw[->] (0,-2) -- (0,4) node[anchor=south] {\tiny{$\bR^{n}$}};
\draw[gray] (0,0) -- (0.4,-0.4);
\draw[gray] (0,0) -- (2,2);
\draw[gray] (0.4,-0.4) -- (2.4,1.6);
\draw[gray] (2.4,1.6) -- (2,2);
\fill[blue] (1.3,1) circle (1.5pt);
\draw (1.3,1) node[anchor=north,blue] {\tiny{$b_{0}$}};
\draw[very thin,dashed] (2.4,0) node[anchor=north] {\tiny{$dH^{d-1}$}} -- (2.4,2.4);
\draw[very thin,dashed] (0,2.4) -- (2.4,2.4);
\draw (-0.7,3) node[anchor=south,red] {$\Theta^{+}\times\Theta^{-}$};
\draw[blue,very thin] (-0.5,1) -- (3.1,1) node[anchor=west] {$[y]_{a^{j_{1}}\mathcal{A}}$};

\end{tikzpicture}
\caption{Intersection of $\Theta^{+}\times\Theta^{-}$ and $[y]_{a^{j_{1}}\mathcal{A}}$}
\label{atx}
\end{figure}
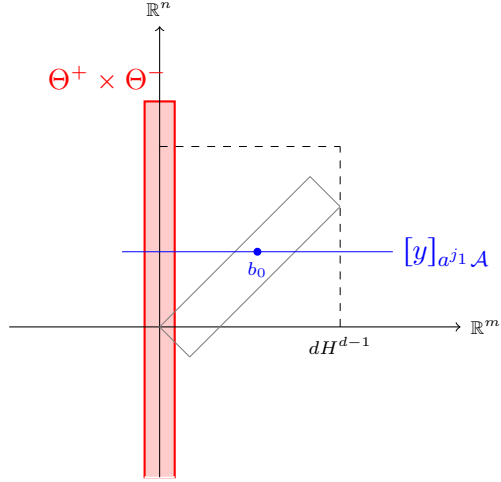

Let us consider the following sets: 
$$\Theta^+\defn\set{b^+\in\bR^m: \|b^+\|_{\mb{r}}\leq e^{-j_2}\eps^{\frac{m}{d}}}
\ \text{ and }\ \Theta^-\defn\set{b^-\in\bR^n: \|b^-\|_{\mb{s}}\leq e^{j_2}\eps^{\frac{n}{d}}}.$$
If $b=(b^+,b^-)\in\Theta^+\times\Theta^-$, then $\|e^{\mb{r}j_2}b^+\|_{\mb{r}}\leq \eps^{\frac{m}{d}}$ and 
$\|e^{-\mb{s}j_2}b^-\|_{\mb{s}}\leq \eps^{\frac{n}{d}}$, where $e^{\mb{r}j_2}b^+$ and $e^{-\mb{s}j_2}b^-$ 
denote the vectors such that $a^{j_2}b=(e^{\mb{r}j_2}b^+,e^{-\mb{s}j_2}b^-)$. It follows that $w(b)g\Gamma\notin a^{-j_2}\cL_\eps$ since 
$$a^{j_2}w(b^+,b^-)g\Gamma=w(e^{\mb{r}j_2}b^+,e^{-\mb{s}j_2}b^-)a^{j_2}g\Gamma\notin\cL_\eps$$ by the definition of $\cL_\eps$. 

Now we claim that the set $\Theta^+\times\set{b_0^{-}}$ is contained in the intersection of $(b_0^++\Xi)\times\set{b_0^{-}}$ and 
$\Theta^+\times\Theta^-$. See Figure \ref{atx}. 
It is enough to show that $\Theta^+ \subset b_0^+ + \Xi$ and $b_0^- \in \Theta^-$.
Since $\|b_0^-\|_s \leq (dH^{d-1})^{\frac{1}{s_n}}$, the latter assertion follows from the assumption 
$j_2\ge \log((dH^{d-1})^{\frac{1}{s_n}}\eps^{-\frac{n}{d}})$. To show the former assertion, fix any $b^+ \in \Theta^+$. 
By the quasi-metric property of $\|\cdot\|_{\mb{r}}$ as in \eqref{quasitriang}, it follows from the assumptions $j_1\ge \log((2dH^{d-1})^{\frac{1}{r_m}}\del^{-1})$ and $j_2\ge \log((dH^{d-1})^{\frac{1}{s_n}}\eps^{-\frac{n}{d}})$ that
\[
\begin{split}
\|b^+ - b_0^+\|_{\mb{r}} &\leq 2^{\frac{1-r_m}{r_m}}(\|b^+\|_{\mb{r}}+\|b_0^+\|_{\mb{r}})\leq 2^{\frac{1-r_m}{r_m}}(e^{-j_2}\eps^{\frac{m}{d}} + (dH^{d-1})^{\frac{1}{r_m}})\\
&\leq 2^{\frac{1-r_m}{r_m}}((dH^{d-1})^{-\frac{1}{s_n}}\eps + (dH^{d-1})^{\frac{1}{r_m}})
\leq 2^{\frac{1-r_m}{r_m}+1}(dH^{d-1})^{\frac{1}{r_m}} \\
&\leq e^{j_1}\del.
\end{split}
\] Thus we have $b^+ \in b_0^+ +B_{e^{j_1}\del}^{\bR^m,\mb{r}} \subset b_0^+ +\Xi $, which concludes the former assertion.

By the above claim, we obtain
\eq{
1-\tau_y^{a^{j_1}\cA}(a^{-j_2}\cL_\eps)=\tau_y^{a^{j_1}\cA}(Y\setminus a^{-j_2}\cL_\eps)
\geq \frac{m_{\bR^m}(\Theta^{+})}{m_{\bR^m}(b_0^{+}+\Xi)}
\geq \frac{m_{\bR^m}(B^{\bR^m,\mb{r}}_{e^{-j_2}\eps^{\frac{m}{d}}})}{m_{\bR^m}(B^{\bR^m,\mb{r}}_{e^{j_1}r})}
=\frac{e^{-j_2}\eps^{\frac{m}{d}}}{e^{j_1}r}.
}
This proves the lemma. 
\end{proof}

\begin{proof}[Proof of Theorem \ref{thmEff1}]
Suppose that $A\in M_{m,n}(\R)$ is not singular on average, and let 
$$\eta_A=\sup\set{\eta:x_A \ \textrm{has} \ \eta\textrm{-escape of mass}}<1.$$ 
By Proposition \ref{prop5}, there is an $a$-invariant measure $\ov{\mu}\in\crly{P}(\overline{Y})$ such that 
$$\Supp\ov{\mu}\subset\cL_\eps\cup(\overline{Y}\setminus Y),\; \pi_*\ov{\mu}=\mu_A\in\crly{P}(\overline{X}),\; 
\text{and}\; \ov{\mu}(\overline{Y}\setminus Y)=\mu_A(\overline{X}\setminus X)=\eta_A.$$
This measure can be represented by the linear combination $$\ov{\mu}=(1-\eta_A)\mu+\eta_A\del_\infty,$$ 
where $\del_\infty$ is the dirac delta measure on $\overline{Y}\setminus Y$ and $\mu\in\crly{P}(Y)$ is $a$-invariant.
There is a compact set $K\subset X$ such that $\mu_A(K)>0.99\mu_A(X)$. We can choose $0<r<1$ such that $Y(r)\supset \pi^{-1}(K)$ 
and $\mu(Y(r))>0.99$. Note that the choice of $r$ is independent of $\eps$ since $\mu_A$ is only determined by fixed $A$. 

Let $\cA^W$ be as in Proposition \ref{algebracst} for $\mu$, $r_0$, and $L=W$, 
and let $\cA^W_\infty$ be as in \eqref{eq:tailalg}.
It follows from \eqref{entropy'} of Proposition \ref{prop5} that
$$h_{\ov{\mu}}(a|\overline{\cA^W_\infty})\ge (1-\eta_A)-r_1(m-\dim_H \mb{Bad}_A(\eps)).$$ 
Since the entropy function is linear with respect to the measure, it follows that
\eq{h_{\mu}(a|\cA^W_\infty)= \frac{1}{1-\eta_A}h_{\ov{\mu}}(a|\overline{\cA^W_\infty})\ge 1-\frac{r_1}{1-\eta_A}(m-\dim_H \mb{Bad}_A(\eps)).}
By Proposition \ref{algexiA}, we obtain
\eqlabel{entest}{H_{\mu}(\cA^W|a\cA^W)\ge 1-\frac{r_1}{1-\eta_A}(m-\dim_H \mb{Bad}_A(\eps)).}

By Lemma \ref{Exceptional}, there exists $0<\del<\min((\frac{cr_0}{16d_0})^2,r)$ such that the dynamical $\del$-boundary 
has measure $\mu(E_\del)<0.01$. 
Note that since $r_0$ depends only on $G$, the constants $C_1,C_2>0$ in Lemma \ref{Exceptional}
depend only on $a$ and $G$, hence $\del$ is independent of $\eps$ even if the set $E_\del$ might depend on $\eps$. 
We write $Z=Y(r)\setminus E_\del$ for simplicity. Note that $\mu(Z)\ge\mu(Y(r))-\mu(E_\del)>0.98$.

To apply Lemma \ref{volest}, choose $H\geq 1$ such that \eqlabel{eqH}{Y(r) \subset Y_{\leq H}.} Note that the constant $H$ depends only on $r$.
Set
$$j_1=\lceil\log((2dH^{d-1})^{\frac{1}{r_m}}{\del'}^{-1})\rceil\quad\text{and}\quad 
j_2=\lceil\log((dH^{d-1})^{\frac{1}{s_n}}\eps^{-\frac{n}{d}})\rceil,$$
where $\del'>0$ will be determined below.

Let $\cA=a^{-k}\cA^W$ for $k=\lceil\log(2^{\frac{1}{r_m}}\eps^{-\frac{m}{d}})\rceil+j_2$. 
By Proposition \ref{algebracst}, $[y]_{\cA^W} \subset B_{r_0}^{W}\cdot y$ for all $y\in Y$, and 
$B_{\del}^{W}\cdot y \subset [y]_{\cA^W}$ for all $y\in Z$ since $\del < r$.
It follows from \eqref{eqbilip} that
\[
\forall y\in Y,\ [y]_{\cA^W}\subset B^{W,d_\infty}_{C_0 r_0}\cdot y \qquad\text{and}\qquad
\forall y\in Z,\ B^{W,d_\infty}_{\del/C_0}\cdot y\subset[y]_{\cA^W},
\]
where $B^{W,d_\infty}_r$ is the $d_\infty$-ball of radius $r$ around the identity in $W$.
For simplicity, we may assume that $r_0 < \frac{1}{C_0}$ by choosing $r_0$ small enough.
This implies that 
\eq{
\forall y\in Y,\ [y]_{\cA^W}\subset B^{W,\mb{r}}_{1}\cdot y \qquad\text{and}\qquad
\forall y\in Z,\ B^{W,\mb{r}}_{(\del/C_0)^{\frac{1}{r_m}}}\cdot y\subset[y]_{\cA^W}.
}
Thus, for any $y\in Y$,
\eqlabel{eqWatombd}{
[y]_{\cA}=a^{-k}[a^k y]_{\cA^W} \subset a^{-k}B^{W,\mb{r}}_{1} a^k \cdot y =B^{W,\mb{r}}_{e^{-k}}\cdot y \subset B^{W,\mb{r}}_{r'}\cdot y,
}
where $r'=2^{-\frac{1}{r_m}}e^{-j_2}\eps^{\frac{m}{d}}$.
Similary, it follows that for any $y\in a^{-k}Z$,
\eqlabel{eqWatombd2}{B^{W,\mb{r}}_{\del'}\cdot y\subset [y]_{\cA}\subset B^{W,\mb{r}}_{r'}\cdot y,}
where $\del'=e^{-1}(\del/C_0)^{\frac{1}{r_m}}r'$.
 
Now we will use Corollary \ref{effELcor} with $L=W$, $K=Y$, and $B=B^{W,\mb{r}}_{r'}$. 
Note that the maximal entropy contribution of $W$ for $a^{j_1}$ is $j_1$, and $\mu$ is supported on $a^{-j_2}\cL_\eps$ since $\Supp\mu\subseteq\cL_\eps$ and $\mu$ is $a$-invariant. 
%Since $(B^{W,\mb{r}}_{r'})^2 \subset B^{W,\mb{r}}_{2^{\frac{1}{r_m}}r'}= B^{W,\mb{r}}_{e^{-j_2}2^{-\frac{1}{r_m}}\eps^{\frac{m}{d}}}$ from \eqref{quasitriang}, 
Thus we have
\eqlabel{eqsuppbd}{
B^{W,\mb{r}}_{r'}\Supp \mu \subset B^{W,\mb{r}}_{r'}a^{-j_2}\cL_\eps
=a^{-j_2}B^{W,\mb{r}}_{e^{j_2}r'}\cL_\eps = a^{-j_2}B^{W,\mb{r}}_{2^{-\frac{1}{r_m}}\eps^{\frac{m}{d}}}\cL_\eps \subset a^{-j_2}\cL_{2^{-\frac{d}{mr_m}}\eps}}
by using the triangular inequality of $\mb{r}$-quasinorm as in \eqref{quasitriang} and the definition of $\cL_\eps$ for the last inclusion.
Using \eqref{eqWatombd}, it follows from \eqref{eqsuppbd} and Corollary \ref{effELcor} with $L=W$, $K=Y$, and $B=B^{W,\mb{r}}_{r'}$ that
\eqlabel{est1}{
H_\mu(\cA|a^{j_1}\cA)\leq j_1+\int_Y\log\tau^{a^{j_1}\cA}_y\left(a^{-j_2}\cL_{2^{-\frac{d}{mr_m}}\eps}\right)d\mu(y)}
Using \eqref{eqWatombd2}, it follows from Lemma \ref{volest} with $\del=\del'$ and $r=r'$ that 
for any $y\in a^{-k} Z\cap Y_{\leq H}$,
$$\tau^{a^{j_1}\cA}_y\left(a^{-j_2}\cL_{2^{-\frac{d}{mr_m}}\eps}\right)\leq 1-2^{-\frac{1}{r_m}}e^{-j_1-j_2}r'^{-1}\eps^{\frac{m}{d}}=1-e^{-j_1},$$
hence $-\log\tau^{a^{j_1}\cA}_y(a^{-j_2}\cL_{2^{-\frac{d}{mr_m}}\eps})\ge e^{-j_1}$.
Since $\mu(a^{-k} Z\cap Y_{\leq H})\geq \frac{1}{2}$, it follows from \eqref{est1} that
\eqlabel{est2}{
\begin{aligned}
1-H_\mu(\cA^W|a\cA^W)&=1-\frac{1}{j_1}H_\mu(\cA^W|a^{j_1}\cA^W)=1-\frac{1}{j_1}H_\mu(\cA|a^{j_1}\cA)\\
&\ge-\frac{1}{j_1}\int_{a^{-k} Z\cap Y_{\leq H}}\log\tau^{a^{j_1}\cA}_y(a^{-j_2}\cL_{2^{-\frac{d}{mr_m}}\eps})d\mu(y)
\ge \frac{e^{-j_1}}{2j_1}.
\end{aligned}}
Recall that $j_1$ is chosen by 
\begin{align*}
j_1=\lceil\log((2dH^{d-1})^{\frac{1}{r_m}} e(\del/C_0)^{-\frac{1}{r_m}}2^{\frac{1}{r_m}}e^{j_2}\eps^{-\frac{m}{d}})\rceil
&\leq\lceil\log((2dH^{d-1})^{\frac{1}{r_m}+\frac{1}{s_n}} e^2 (\del/C_0)^{-\frac{1}{r_m}}2^{\frac{1}{r_m}}
\eps^{-\frac{n}{d}}\eps^{-\frac{m}{d}})\rceil\\
&\leq \log((2dH^{d-1})^{\frac{1}{r_m}+\frac{1}{s_n}} e^3 (\del/C_0)^{-\frac{1}{r_m}}2^{\frac{1}{r_m}})
-\log\eps
\end{align*}
Here, the constants $H$ and $\del$ depend on fixed $A\in M_{m,n}(\bR)$, not on $\eps$. 
Combining \eqref{entest} and \eqref{est2}, we obtain
$$m-\dim_H \mb{Bad}_A(\eps)\geq c(A)\frac{\eps}{\log(1/\eps)},$$
where the constant $c(A)>0$ depends only on $d$, $\mb{r}$, $\mb{s}$, and $A\in M_{m,n}(\bR)$. It completes the proof.
\end{proof}

\section{Upper bound for Hausdorff dimension of $\mb{Bad}^{b}(\eps)$}\label{sec:entropyboundb}
In this section, as explained in the introduction, we only consider the unweighted setting, that is, 
\eq{
\mb{r}=(1/m,\dots,1/m)\quad\text{and}\quad \mb{s}=(1/n,\dots,1/n).
}

\subsection{Constructing measure with entropy lower bound}
Similar to Subsection \ref{sec4.1}, we will construct an $a$-invariant measure on $Y$ with a lower bound on the conditional entropy to the $\sigma$-algebra $\cA^U_\infty$ obtained in \eqref{eq:tailalg} and Proposition \ref{algebracst} with $L=U$. 
To control the amount of escape of mass for the desired measure, we need a modification of \cite[Theorem 1.1]{KKLM} as Proposition \ref{KKLM'} below.

%\begin{thm}\cite[Remark 2.1]{KKLM}\label{KKLM}
%For any $x\in X$, the set
%$$\set{u\in U| ux\ \textrm{has} \ \del\textrm{-escape of mass on average}}$$
%has Hausdorff dimension at most $mn-\frac{\del mn}{m+n}$.
%\end{thm}
%
%\begin{rem}
%For the definition of escape of mass, we follow the definition of singularity on average in \cite{DFSU}, Section 1.4. \cite{KKLM} uses $\lim$ instead of $\liminf$ for the definition of escape of mass.
%\end{rem}
For any compact set $\mathfrak{S}\subset X$ and positive integer $k>0$,
and any $0<\eta<1$, let
\begin{align*}
 %E_\eta&\defn\set{A\in \bT^{mn}\subset M_{m,n}(\bR): x_A \ \textrm{has} \ \eta\textrm{-escape of mass on average}}\\
 %F_{\eta,\mathfrak{S}}&\defn\set{A\in \bT^{mn}\subset M_{m,n}(\bR):\frac{1}{k}\displaystyle\sum_{i=0}^{k-1}\del_{a^i x_A}
 %(X\setminus \mathfrak{S})<\eta \  \textrm{for infinitely many} \  k},\\
 F_{\eta,\mathfrak{S}}^k&\defn\set{A\in \bT^{mn}\subset M_{m,n}(\bR):\frac{1}{k}\displaystyle\sum_{i=0}^{k-1}\del_{a^i x_A}
 (X\setminus \mathfrak{S})<\eta}.
\end{align*}
%We note that for any sequence of increasing compact sets $\set{\mathfrak{S}_j}$ exhausting $X$, $\bT^{mn}\setminus E_\eta=\displaystyle\bigcup_{j=1}^\infty F_{\eta,\mathfrak{S}_j}=\bigcup_{j=1}^{\infty}\limsup_{k\to\infty}F_{\eta,\mathfrak{S}_\eta}^k$ holds.

Given a compact set $\mathfrak{S}$ of $X$, $k\in\bN,\eta\in(0,1)$, and $t\in\bN$, define the set
$$Z(\mathfrak{S},k,t,\eta):=\set{A\in\bT^{mn}:\frac{1}{k}\displaystyle\sum_{i=0}^{k-1}\del_{a^{ti} x_A}
 (X\setminus \mathfrak{S})\ge\eta};$$
in other words, the set of $A\in\bT^{mn}$ such that up to time $k$, the proportion of times $i$ for which the orbit point $a^{ti}x_A$ is in the complement of $\mathfrak{S}$ is at least $\eta$. 
The following theorem is one of the main results in \cite{KKLM}.
%In \cite{KKLM}, Theorem \ref{KKLM} is actually deduced from the following covering result.

\begin{thm}\cite[Theorem 1.5]{KKLM}\label{KKLMcov}
There exists $t_0>0$ and $C>0$ such that the following holds. For any $t>t_0$ there exists a compact set $\mathfrak{S}:=\mathfrak{S}(t)$ of $X$ such that for any $k\in\bN$ and $\eta\in(0,1)$, the set $Z(\mathfrak{S},k,t,\eta)$ can be covered with $Ct^{3k}e^{(m+n-\eta)mntk}$ balls in $\bT^{mn}$ of radius $e^{-(m+n)tk}$.
\end{thm}

The following proposition is a slightly stronger variant of \cite[Theorem 1.1]{KKLM} which will be needed later. We prove this using Theorem \ref{KKLMcov}.

\begin{prop}\label{KKLM'}
There exists a family of compact sets $\set{\mathfrak{S}_\eta}_{0<\eta< 1}$ of $X$ such that the following is true. For any $0<\eta\leq 1$,
\eqlabel{KKLM''}{\dim_H(\bT^{mn}\setminus \limsup_{k\to\infty}\displaystyle\bigcap_{\eta'\ge\eta}F^k_{\eta',\mathfrak{S}_{\eta'}})\leq mn-\frac{\eta mn}{2(m+n)}.}
\end{prop}
\begin{proof}
For $\eta\in(0,1)$, let $t_\eta\ge 4$ be the smallest integer such that $\frac{3\log t_\eta}{t_\eta}\leq \frac{\eta mn}{10}$, and $\mathfrak{S}'_\eta$ be the set $\mathfrak{S}(t_\eta)$ of Theorem \ref{KKLMcov}. For $l\ge 4$, denote by $\eta_l>0$ the smallest real number such that $t_{\eta_l}=l$. Then $\eta_l\ge\frac{3\eta_{l-1}}{4}$ for any $l\ge 5$. We note that $\mathfrak{S}'_{\eta'} \subseteq \mathfrak{S}'_\eta$ for any $0<\eta\leq\eta'$. 
For $\eta'\in[\eta_{l},\eta_{l-1})$ let us define $\mathfrak{S}''_{\eta'} = \mathfrak{S}'_{\eta_l}$.
% a family of compact sets $\mathfrak{S}''_\eta$ such that $\mathfrak{S}'_{\eta_l} \subseteq \mathfrak{S}''_{\eta'}$ for any $l\ge 4$ and $\eta_{l}\leq\eta'< \eta_{l-1}$. 
For any $\eta\in(0,1)$, we set $\mathfrak{S}_\eta = \bigcup_{-t_{\eta}\leq t \leq t_\eta} a^t \mathfrak{S}_{\eta}''$
so that for any $-t_\eta \leq t\leq t_\eta$ and $x\in\mathfrak{S}''_\eta$, $a^tx\in\mathfrak{S}_\eta$.

Now we prove that this family of compact sets $\set{\mathfrak{S}_\eta}_{0<\eta<1}$ satisfies \eqref{KKLM''}. 
Suppose $A\notin F_{\eta,\mathfrak{S}_\eta}^k$, which implies $\frac{1}{k}\displaystyle\sum_{i=0}^{k-1}\del_{a^i x_A}
 (X\setminus \mathfrak{S}_\eta)\ge \eta$.
For sufficiently large $k$,  $$\frac{1}{\lceil\frac{k}{t_\eta}\rceil}\displaystyle\sum_{i=0}^{\lceil\frac{k}{t_\eta}\rceil-1}\del_{a^{t_\eta i} x_A}
 (X\setminus \mathfrak{S}''_\eta)\ge \frac{1}{t_\eta \lceil\frac{k}{t_\eta}\rceil}\displaystyle\sum_{i=0}^{t_\eta(\lceil\frac{k}{t_\eta}\rceil-1)}\del_{a^{i} x_A}
 (X\setminus \mathfrak{S}_\eta)\ge\frac{9}{10}\eta.$$
Hence $\bT^{mn}\setminus F_{\eta,\mathfrak{S}_{\eta}}^{ k}\subseteq Z(\mathfrak{S}''_\eta,\lceil\frac{k}{t_\eta}\rceil,t_\eta,\frac{9}{10}\eta)$ for any $0<\eta<1$ and sufficiently large $k\in\bN$. 
 
 For any $\eta_{l+1}< \eta'\leq \eta_l$, we have $t_{\eta'}=l$ and the set $Z(\mathfrak{S}''_{\eta'},\lceil\frac{k}{t_\eta}\rceil,t_{\eta'},\frac{9}{10}\eta')$ is contained in $Z(\mathfrak{S}'_{\eta_l},\lceil\frac{k}{t_{\eta_l}}\rceil,l,\frac{9}{10}\eta_l)$. It follows that for any $0<\eta<1$
 $$\bT^{mn}\setminus\displaystyle\bigcap_{\eta'\ge\eta}F_{\eta',\mathfrak{S}_{\eta'}^k}^{k}\subseteq \displaystyle\bigcup_{\eta'\ge\eta}Z(\mathfrak{S}''_{\eta'},\lceil\frac{k}{t_{\eta'}}\rceil,t_{\eta'},\frac{9}{10}\eta')\subseteq \displaystyle\bigcup_{l=4}^{t_\eta}Z(\mathfrak{S}'_{\eta_l},\lceil\frac{k}{l}\rceil,l,\frac{9}{10}\eta_l),$$
 hence
 $$\bT^{mn}\setminus\limsup_{k\to\infty}\displaystyle\bigcap_{\eta'\ge\eta}F^{k}_{\eta',\mathfrak{S}_{\eta'}}\subseteq \displaystyle\bigcup_{k_0\ge 1}\displaystyle\bigcap_{k=k_0}^{\infty}\displaystyle\bigcup_{l=4}^{t_\eta}Z(\mathfrak{S}'_{\eta_l},\lceil\frac{k}{l}\rceil,l,\frac{9}{10}\eta_l).$$
 By Theorem \ref{KKLMcov}, the set $\displaystyle\bigcup_{l=4}^{t_\eta}Z(\mathfrak{S}'_{\eta_l},\lceil\frac{k}{l}\rceil,l,\frac{9}{10}\eta_l)$ can be covered with
 \eq{\begin{aligned}
\displaystyle\sum_{l=4}^{t_\eta}Cl^{3\lceil\frac{k}{l}\rceil}e^{(m+n-\frac{9}{10}\eta_l)mn\lceil\frac{k}{l}\rceil l}&\leq  \displaystyle\sum_{l=4}^{t_\eta} Ct_\eta^3e^{\frac{3\log l}{l}k}e^{(m+n-\frac{9}{10}\eta_l)mn(k+t_\eta)}\\
&\leq\displaystyle\sum_{l=4}^{t_\eta}Ct_\eta^3e^{(m+n)mnt_\eta}e^{(m+n-\frac{8}{10}\eta_l)mnk}\\
&\leq Ct_\eta^4e^{(m+n)mnt_\eta}e^{(m+n-\frac{\eta}{2})mnk}
 \end{aligned}}
 balls in $\bT^{mn}$ of radius $e^{-(m+n)k}$. Here we used $\eta_{t_{\eta}}\ge\frac{3\eta}{4}$ which follows from $\eta_l\ge\frac{3\eta_{l-1}}{4}$ for any $l\ge 5$. Thus, for any sufficiently large $k_0\in\bN$,
 \eq{\begin{aligned}
 \dim_H&\left(\displaystyle\bigcap_{k=k_0}^{\infty}\displaystyle\bigcup_{l=4}^{t_\eta}Z(\mathfrak{S}'_{\eta_l},\lceil\frac{k}{l}\rceil,l,\eta_l)\right)\leq\limsup_{k\to\infty}\frac{\log(Ct_\eta^4e^{(m+n)mnt_\eta}e^{(m+n-\frac{\eta}{2})mnk})}{-\log(e^{-(m+n)k})}\\
 &=\limsup_{k\to\infty}\frac{\log(Ct_\eta^4e^{(m+n)mnt_\eta})+(m+n-\frac{\eta}{2})mnk}{(m+n)k}=mn-\frac{\eta mn}{2(m+n)},
 \end{aligned}}
hence we get $\dim_H(\bT^{mn}\setminus\displaystyle\limsup_{k\to\infty}\displaystyle\bigcap_{\eta'\ge\eta}F^{k}_{\eta',\mathfrak{S}_{\eta'}})\leq mn-\frac{\eta mn}{2(m+n)}$.
\end{proof}

In the rest of this subsection, we will prove the following proposition which gives the bound of $\dim_H \mb{Bad}^b(\eps)$. The construction of the $a$-invariant measure with large relative entropy roughly follows the construction in Proposition \ref{prop5}. However, the situation is significantly different, as fixing $b$ does not determine the amount of excursion in the cusp.
The additional step using Proposition \ref{KKLM'} is necessary to control the measure near the cusp allowing a small amount of escape of mass. 

\begin{prop}\label{prop2}
Let $\set{\mathfrak{S}_\eta}_{0<\eta< 1}$ be the family of compact sets of $X$ as in Proposition~\ref{KKLM'}.  For fixed $b\in \bR^m$ and $\eps>0$, assume that $\dim_H \mb{Bad}^{b}(\eps)>\dim_H \mb{Bad}^{0}(\eps)$. Let $\eta_0:=2(m+n)(1-\frac{\dim_H \mb{Bad}^{b}(\eps)}{mn})$. Then there exist an $a$-invariant measure $\ov{\mu}\in \crly{P}(\overline{Y})$ such that
\begin{enumerate}
  %  \item\label{ainv} $\mu^b$ is $a_t$-invariant,
    \item\label{supp} $\Supp{\ov{\mu}}\subseteq \cL_\eps\cup (\overline{Y}\setminus Y)$,
    \item\label{cusp} $\pi_*\ov{\mu}(\overline{X}\setminus\mathfrak{S}_{\eta'})\leq \eta'$ for any $\eta_0\leq\eta'<1$, in particular, there exist $\mu\in\crly{P}(Y)$ and $0\leq \widehat{\eta}\leq\eta_0$ such that
    $$\ov{\mu}=(1-\widehat{\eta})\mu+\widehat{\eta}\del_\infty,$$
    where $\del_\infty$ is the dirac delta measure on $\overline{Y}\setminus Y$.
    \item\label{entropy} Let $\cA^U$ be as in Proposition \ref{algebracst} for $\mu$, $r_0$, and $L=U$, 
    and let $\cA^U_\infty$ be as in \eqref{eq:tailalg}.
Then we have
    $$h_{\ov{\mu}}(a|\overline{\cA^U_\infty})\ge(1-\widehat{\eta}^{\frac{1}{2}})(d-\frac{1}{2}\eta_0-d\widehat{\eta}^{\frac{1}{2}}).$$
\end{enumerate}
\end{prop}
\begin{rem}
We remark that this proposition is valid for the weighted setting except for the construction of 
$\set{\mathfrak{S}_\eta}_{0<\eta< 1}$ since it depends on the unweighted result (Theorem \ref{KKLMcov}) in \cite{KKLM}.
So, we keep the notations $\mb{r}$ and $\mb{s}$ for weights in the following proof.
\end{rem}
%In particular, by taking a sequence of compact sets $\mathfrak{S}_i$ exhausting $X$ and a sequence of positive numbers $\gamma_i$ which converges to $0$,
% $$h_{\mu} (a|\overline{\cA^U}) \geq \frac{m+n}{mn} \dim_H \mb{Bad}^{b}(\eps).$$
% Note that the right handside is the maximal entropy if $\dim_H \mb{Bad}^{b}(\eps) = mn.$
\begin{proof}
%Take a sequence of increasing compact sets $\set{\mathfrak{S}_j}$ exhausting $X$.
%Since $E_{0.01}=\displaystyle\bigcap_{j=1}^{\infty} (\bT^{mn}\setminus F_{0.01,\mathfrak{S}_j})$, $\dim_H E_{0.01}<mn-\frac{0.01mn}{m+n}$, and $\set{F_{0.01,\mathfrak{S}_j}}_{j=1}^{\infty}$ is an increasing sequence, we can find $j_0\in\bN$ such that 
%\eqlabel{0.01dim}{\dim_H (\bT^{mn}\setminus F_{0.01,\mathfrak{S}_{j_0}})<mn-\frac{0.005mn}{m+n}.} We set $\mathfrak{S}=\mathfrak{S}_{j_0}$ and remark that $\mathfrak{S}'$ is chosen independently to $\eps$.

For $\eps>0$, denote by $R$ the set $\mb{Bad}^{b}(\eps)\setminus \mb{Bad}_{0}^{b}(\eps)$, and let 
$$R^{T}:=\set{A\in R \cap \bT^{mn} \subset M_{m,n}(\bR) |\forall t\ge T, a_t x_{A,b}\in \cL_\eps}.$$ The sequence $\set{R^T}_{T\ge 1}$ is increasing, and $R=\displaystyle\bigcup_{T=1}^\infty R^{T}$ by Proposition~\ref{prop1}. Since $\dim_H \mb{Bad}^{b}(\eps)>\dim_H \mb{Bad}^{0}(\eps)\ge\dim_H \mb{Bad}_{0}^{b}(\eps)$, it follows that $\dim_H R=\dim_H \mb{Bad}^{b}(\eps)$. Thus for any $\gamma>0$, there exists $T_\gamma\ge 1$ satisfying
\eqlabel{RTdim}{\dim_H R^{T_\gamma}>\dim_H \mb{Bad}^{b}(\eps)-\gamma.}

Let $\eta=2(m+n)(1-\frac{\dim_H \mb{Bad}^{b}(\eps)-\gamma}{mn})$. 
If $0<\gamma<\frac{mn}{2(m+n)}-(mn-\dim_H \mb{Bad}^{b}(\eps))$, then $0<\eta<1$.
For $k\in\bN$, write $\widetilde{F}_\eta^k:=\displaystyle\bigcap_{\eta'\ge\eta}F^k_{\eta',\mathfrak{S}_{\eta'}}$ for simplicity. Recall that we have 
\eqlabel{Fetadim}{\dim_H (\bT^{mn}\setminus\displaystyle\limsup_{k\to\infty}\widetilde{F}_\eta^k) \leq mn-\frac{\eta mn}{2(m+n)}=\dim_H \mb{Bad}^{b}(\eps)-\gamma} by Theorem~\ref{KKLM'}. It follows from \eqref{RTdim} and \eqref{Fetadim} that $$\dim_H(R^{T_\gamma}\cap \limsup_{k\to\infty}\widetilde{F}_\eta^k)>\dim_H \mb{Bad}^{b}(\eps)-\gamma.$$
%Since $R^{T_\gamma}\cap \displaystyle\limsup_{k\to\infty}\widetilde{F}_\eta^k=\displaystyle\bigcap_{N=1}^\infty \displaystyle\bigcup_{k=N}^{\infty}(R^{T_\gamma}\cap \widetilde{F}_\eta^k)$, 
Thus there is an increasing sequence of positive integers $\set{k_i}\to \infty$ such that
$$\dim_H (R^{T_\gamma}\cap \widetilde{F}_{\eta}^{k_i})> \dim_H \mb{Bad}^{b}(\eps)-\gamma.$$

For each $k_i\ge T_\gamma$ let $S_i$ be a maximal $e^{-k_i}$-separated subset of $R^{T_\gamma}\cap \widetilde{F}_{\eta}^{k_i}$ with respect to the quasi-distance $d_{\mb{r}\otimes\mb{s}}$. By Lemma \ref{relating dimensions},
\eqlabel{eqn1}{\begin{aligned}
  \displaystyle\liminf_{i\to\infty}\frac{\log |S_i|}{k_i}
  \ge\underline{\dim}_{\mb{r}\otimes\mb{s}} (R^{T_\gamma}\cap \widetilde{F}_\eta^{k_i})
  &> m+n-(r_1+s_1)(mn-\dim_H \mb{Bad}^{b}(\eps)+\gamma)\\
  &= m+n-\frac{m+n}{mn}(mn-\dim_H \mb{Bad}^{b}(\eps)+\gamma)\\
  &=\frac{m+n}{mn}(\dim_H \mb{Bad}^{b}(\eps)-\gamma).
\end{aligned}}

Let $\nu_i\defn \frac{1}{|S_i|}\displaystyle\sum_{y\in D_i}\delta_{y}= \frac{1}{|S_i|}\displaystyle\sum_{A\in S_i}\delta_{y_{A,b}}$ be the normalized counting measure on the set $D_i:=\set{y_{A,b}:A\in S_i}\subset Y$ and let $\mu^{\gamma}$ be a weak*-limit of $\mu_i$:
$$\mu_i\defn \frac{1}{k_i}\displaystyle\sum_{k=0}^{k_i-1}a^k_{*}\nu_i \wstar \mu^{\gamma}\in\crly{P}(\overline{Y})$$
By extracting a subsequence if necessary, we may assume that $\mu^{\gamma}$ is a weak*-accumulation point of $\set{\mu_i}$. The measure $\mu^{\gamma}$ is clearly an $a$-invariant measure since $a_*\mu_i-\mu_i$ goes to zero measure.

Choose any sequence of positive real numbers $(\gamma_j)_{j\ge 1}$ converging to zero and $(\eta_j)_{j\ge 1}$ be the corresponding sequence such that $$\eta_j=2(m+n)(1-\frac{\dim_H \mb{Bad}^{b}(\eps)-\gamma_j}{mn}).$$ Let $\set{\mu^{\gamma_j}}$ be a family of $a$-invariant probability measures on $\overline{Y}$ obtained from the above construction for each $\gamma_j$. Extracting a subsequence again if necessary, we may take a weak$^*$-limit measure $\ov{\mu}\in\crly{P}(\overline{Y})$ of $\set{\mu^{\gamma_j}}$. We prove that $\ov{\mu}$ is the desired measure. The measure $\ov{\mu}$ is clearly $a$-invariant.\\ 
(\ref{supp}) We show that for any $\gamma$, $\mu^\gamma(Y\setminus\cL_\eps)=0$. For any $A\in S_i\subseteq R^{T_\gamma}$, $a^T y_{A,b}\in \cL_\eps$ holds for $T>T_\gamma$. Thus
$$\mu_i(Y\setminus\cL_\eps)=\frac{1}{k_i}\displaystyle\sum_{k=0}^{k_i-1}(a^k)_{*}\nu_i(Y\setminus\cL_\eps)=\frac{1}{k_i}\displaystyle\sum_{k=0}^{T_\gamma}(a^k)_{*}\nu_i(Y\setminus\cL_\eps)\leq\frac{T_\gamma}{k_i}.$$ By taking limit for $k_i\to \infty$, we have $\mu^\gamma(Y\setminus\cL_\eps)=0$ for arbitrary $\gamma$, hence,
$$\ov{\mu}(Y\setminus\cL_\eps)=\lim_{j\to\infty}\mu^{\gamma_j}(Y\setminus\cL_\eps)=0.$$\\
(\ref{cusp}) For any $\gamma=\gamma_j$, if $A\in S_i\subset \widetilde{F}_{\eta_j}^{k_i}=\displaystyle\bigcap_{\eta'\ge\eta_j}F_{\eta',\mathfrak{S}_{\eta'}}^{k_i}$, then for all $i \in \mathbb{N}$ and $\eta_j \leq\eta'\leq 1$,  $\frac{1}{k_i}\displaystyle\sum_{k=0}^{k_i-1}\delta_{a^k x_A}(X\setminus \mathfrak{S}_{\eta'})<\eta'$. Therefore for all $i\in\bN$ and $\eta_j \leq\eta'\leq 1$,
\begin{align*}
    \pi_*\mu_i(X\setminus \mathfrak{S}_{\eta'})
    %&=\frac{1}{|S_i|}\displaystyle\sum_{y\in D_i} \frac{1}{k_i}\displaystyle\sum_{k=0}^{k_i-1} \pi_*\delta_{a^k y}(X\setminus \mathfrak{S}_{\eta'})\\
    &=\frac{1}{|S_i|}\displaystyle\sum_{A\in S_i} \frac{1}{k_i}\displaystyle\sum_{k=0}^{k_i-1} \delta_{a^k x_A}(X\setminus \mathfrak{S}_{\eta'})
    <\eta',
\end{align*}
hence $\pi_*\mu^{\gamma_j}(\overline{X}\setminus \mathfrak{S}_{\eta'})=\displaystyle\lim_{i\to\infty}\pi_*\mu_i(X\setminus \mathfrak{S}_{\eta'})\leq\eta'$. 
Since $\eta_j$ converges to $\eta_0$ as $j\to \infty$, we have
$$\pi_*\ov{\mu}(\overline{X}\setminus \mathfrak{S}_{\eta'})\leq \eta'$$ for any $\eta' \geq \eta_0$. Hence,
$$\ov{\mu}(\overline{Y}\setminus Y)\leq \lim_{\eta'\to\eta_0}\pi_*\ov{\mu}(\overline{X}\setminus \mathfrak{S}_{\eta'})\leq \eta_0,$$
so we have a decomposition $\ov{\mu}=(1-\widehat{\eta})\mu+\widehat{\eta}\del_\infty$ for some $\mu\in\crly{P}(Y)$ and $0\leq\widehat{\eta}\leq\eta_0$.

For the rest of the proof, let us check the condition (\ref{entropy}).\\
(\ref{entropy})
We first fix any $D>J_U =m+n$. As in the proof of Proposition \ref{prop5}, there exists a finite partition $\cQ$ of $Y$ satisfying:
    \begin{itemize}
        \item $\cQ$ contains an atom $Q_\infty$ of the form $\pi^{-1}(Q_\infty^0)$, where $X\smallsetminus Q_\infty^0$ has compact closure.
        \item $\forall Q\in \cQ\smallsetminus\set{Q_\infty}$, $\diam Q<r_D=r_D(Q_\infty^0)$, where $r_D$ is as in Subsection \ref{sec3.3},
        \item $\forall Q\in\cQ,\forall j\geq 1,\; \mu^{\gamma_j}(\partial Q)=0$.
    \end{itemize}
Remark that for all $i\geq 1$, $D_i \subset \{y_{A,b}:A\in[0,1]^{mn}, b\in[0,1]^m\}$, which is a compact set in $Y$,
therefore we can choose $Q^0_{\infty}$ so that 
\eqlabel{eqcsp}{
Q_\infty \cap D_i = \varnothing.
}
To prove \eqref{entropy}, it suffices to prove that for all $q\ge 1$,
    \eqlabel{sttlowbdd'}{\frac{1}{q}H_{\ov{\mu}}(\overline{\cQ}^{(q)}|\overline{\cA^U_\infty})
    \ge (1-\ov{\mu}(\overline{Q_\infty})^{\frac{1}{2}})\left(\frac{m+n}{mn}\dim_H \mb{Bad}^{b}(\eps)
    -D\ov{\mu}(\overline{Q_\infty})^{\frac{1}{2}}\right).}
Indeed, taking $D\to m+n$ and $Q^0_\infty\subset X$ such that $\ov{\mu}(\ov{Q_\infty}) \to \widehat{\eta}$, it follows that
\eq{
h_{\ov{\mu}}(a|\overline{\cA^U_\infty})\ge(m+n)(1-\widehat{\eta}^{\frac{1}{2}})(\frac{1}{mn}\dim_H \mb{Bad}^{b}(\eps)-\widehat{\eta}^{\frac{1}{2}})
=(1-\widehat{\eta}^{\frac{1}{2}})(d-\frac{1}{2}\eta_0-d\widehat{\eta}^{\frac{1}{2}}).
}

It remains to prove \eqref{sttlowbdd'}.
It is trivial if $\ov{\mu}(\ov{Q_\infty})=1$, so assume that $\ov{\mu}(\ov{Q_\infty})<1$, hence for all large enough $j\ge1$, $\mu^{\gamma_j}(\ov{Q_\infty})<1$. Now we fix such $j\ge 1$ and write temporarily $\gamma=\gamma_j$.

Choose $\beta>0$ such that $\mu^\gamma(\ov{Q_\infty}) < \beta < 1$. Then for large enough $i$,
$$\mu_i(Q_\infty)=\frac{1}{k_i|S_i|}\displaystyle\sum_{y\in D_i, 0\le k<k_i}\de_{a^k y}(Q_\infty) < \beta.$$ In other words, there exist at most $\beta k_i|S_i|$ number of $a^k y$'s in $Q_\infty$ with $y\in D_i$ and $0\leq k<k_i$.

Let $S'_i\subset S_i$ be the set of $A\in S_i$'s such that
\eqlabel{eq:cuspcount}{|\{0\leq k< k_i: a^ky_{A,b} \in Q_\infty \} |\leq \beta^{\frac{1}{2}}k_i.}
Thus we have $|S_i\setminus S_i'|\leq \beta^{\frac{1}{2}}|S_i|$, hence 
\eqlabel{Sicount}{|S'_i|\ge (1-\beta^{\frac{1}{2}})|S_i|.} 
Let $\nu'_i\defn \frac{1}{|S'_i|}\displaystyle\sum_{y\in S'_i}\de_y$ be the normalized counting measure on $D'_i$, where $D'_i:=\set{y_{A,b}:A\in S'_i}\subset Y$. By definition, 
$\nu_i(Q)\ge \frac{|S'_i|}{|S_i|}\nu'_i(Q)$ for all measurable set $Q\subseteq Y$. 
Thus, for any arbitrary countable partition $\cQ$ fo $Y$, 
\eqlabel{eqn2}{
\begin{aligned}
H_{\nu_i}(\cQ)&=-\displaystyle\sum_{\nu_i(Q)\leq\frac{1}{e}}\log(\nu_i(Q))\nu_i(Q)-\displaystyle\sum_{\nu_i(Q)>\frac{1}{e}}\log(\nu_i(Q))\nu_i(Q)\\&\ge -\displaystyle\sum_{\nu_i(Q)\leq\frac{1}{e}}\log(\frac{|S'_i|}{|S_i|}\nu'_i(Q))\frac{|S'_i|}{|S_i|}\nu'_i(Q)
\\&=-\frac{|S'_i|}{|S_i|}\displaystyle\sum_{\nu_i(Q)\leq\frac{1}{e}}\log(\nu'_i(Q))\nu'_i(Q)-\frac{|S'_i|}{|S_i|}\log{\frac{|S'_i|}{|S_i|}}\displaystyle\sum_{\nu_i(Q)\leq\frac{1}{e}}\nu'_i(Q)
\\&\geq \frac{|S'_i|}{|S_i|}\Bigl\{H_{\nu'_i}(\cQ)+\displaystyle\sum_{\nu_i(Q)>\frac{1}{e}}\log(\nu'_i(Q))\nu'_i(Q)\Bigr\}\\&\ge (1-\beta^{\frac{1}{2}})(H_{\nu'_i}(\cQ)-\frac{2}{e}).
\end{aligned}
}
In the last inequality, we use the fact that $\nu'_i$ is a probability measure, thus there can be at most two elements $Q$ of the partition for which $\nu'_i (Q) > \frac{1}{e}$.

To compute $H_{\nu_i'}(\cQ^{(k_i)})$ note that for any $y \in D_i'$, $y \notin Q_\infty$.
From Lemma \ref{CovLem} with $L=U$, \eqref{eqcsp}, and \eqref{eq:cuspcount}, if $Q$ is any non-empty atom of $\cQ^{(k_i)}$, fixing any $y\in D'_{i}\cap Q$, the set
\eq{
D'_{i}\cap Q = D'_{i}\cap [y]_{\cQ^{(k_i)}} \subset E_{y,k_{i}-1}
} 
can be covered by $Ce^{D\sqrt{\beta} k_{i}}$ $d_{\mb{r}\otimes\mb{s}}$-balls of radius $r_D^{\frac{1}{r_1 +s_1}}e^{-k_{i}}$,
where $C$ is a constant depending on $Q^0_\infty$ and $D$, but not on $k_i$. 
Since $D'_{i}$ is $e^{-k_{i}}$-separated with respect to $d_{\mb{r}\otimes\mb{s}}$ and $r_D^{\frac{1}{r_1 +s_1}}<\frac{1}{2}$, 
we get 
$$|S_i'|\nu_i'(Q) = \mathrm{Card}(D'_i \cap Q)\leq Ce^{D\sqrt{\beta} k_i },$$
hence we have
\eqlabel{staticbdU2}{
H_{\nu_i'}(\cQ^{(k_i)})\ge\log |S_i'|-D\beta^{\frac{1}{2}} k_i-\log C.}

Now let $\cA^U=(\cP^U)_0^\infty=\bigvee_{i=0}^{\infty}a^i \cP^U$ be as in Proposition \ref{algebracst} for $\mu$, 
$r_0$ and $L=U$, and let $\cA^U_\infty$ be as in \eqref{eq:tailalg}.\vspace{0.3cm}\\
\textbf{Claim}\; $H_{\nu_i}(\cQ^{(k_i)}|\cA^U_\infty) = H_{\nu_i}(\cQ^{(k_i)})$ for all large enough $\ell\geq 1$.
\begin{proof}[Proof of Claim]
Using the continuity of entropy, we have 
\eq{H_{\nu_i}(\cQ^{(k_i)}|\cA^U_\infty)=\lim_{\ell\to\infty}H_{\nu_i}(\cQ^{(k_i)}|(\cP^U)_\ell^\infty).}
Now we show $H_{\nu_i}(\cQ^{(k_i)}|(\cP^U)_\ell^\infty) = H_{\nu_i}(\cQ^{(k_i)})$ for all large enough $\ell\geq 1$.
Let $\cP$ and $E_{\del}$ be as in Lemma \ref{Exceptional} for $\mu$ and $r_0$.
As mentioned in Remark \ref{basepoint}, we may assume that there exists $y\in \{y_{A,b}:A\in \bT^{mn}\subset M_{m,n}(\bR)\}$
such that $y\notin \partial \cP$.
Since $E_\del = \bigcup_{k=0}^\infty a^k \partial_{d_0e^{-k\al}\del}\cP$, $y\in Y \setminus E_{\del}$
for some small enough $\del>0$, which implies that $a^{-\ell}y\in Y\setminus a^{-\ell}E_\del \subset Y\setminus E_\del$. 
Hence, it follows from \eqref{excEq} and Proposition \ref{algebracst} that
\[
[y]_{(\cP^U)_\ell^\infty} = a^\ell[a^{-\ell}y]_{(\cP^U)_0^\infty} = a^\ell[a^{-\ell}y]_{\cA^U} \supset a^{\ell} B_{\del}^{U} a^{-\ell} y \supset B_{d_0 e^{\alpha\ell}\del}^{U} y.
\]
Since the support of $\nu_i$ is a set of finite points on a single compact $U$-orbit, 
$\nu_i$ is supported on a single atom of $(\cP^W)_\ell^\infty$ for all large enough $\ell\geq 1$.
This proves the claim.
\end{proof}

Combining \eqref{Sicount}, \eqref{eqn2}, \eqref{staticbdU2}, and \textbf{Claim}, we have 
\eqlabel{sttbdd2}{\begin{split}
H_{\nu_i}(\cQ^{(k_i)}|\cA^U_\infty)&= H_{\nu_i}(\cQ^{(k_i)})
\ge (1-\beta^{\frac{1}{2}})(H_{\nu_i'}(\cQ^{(k_i)})-\frac{2}{e})\\
&\ge (1-\beta^{\frac{1}{2}})(\log |S_i|-D\beta^{\frac{1}{2}} k_i-\log C-\frac{2}{e}+\log(1-\beta^{\frac{1}{2}})).
\end{split}}

%For any $q\ge 1$, write the Euclidean division of large enough $k_i-1$ by $q$ as
%$$k_i-1=qk'+s \ \textrm{with} \ s\in\set{0,\cdots,q-1}.$$
%By subadditivity of the entropy with respect to the partition, for each $p\in\set{0,\cdots,q-1}$,
%$$H_{\nu_i}(\cQ^{(k_i)}|\cA^U_\infty)\leq H_{a^{p}\nu_i}(\cQ^{(q)}|\cA^U_\infty)+\cdots+H_{a^{p+qk'}\nu_i}(\cQ^{(q)}|\cA^U_\infty)+2q\log |\cQ|.$$
%Summing those inequalities for $p=0,\cdots,q-1$, and using the concave property of entropy with respect to the measure, we obtain
%\begin{align*}
%  qH_{\nu_i}(\cQ^{(k_i)}|\cA^U_\infty)
%  &\leq\displaystyle\sum_{k=0}^{k_i-1}H_{a^k \nu_i}(\cQ^{(q)}|\cA^U_\infty)+2q^2\log |\cQ|\\
%  &\leq k_i H_{\mu_i}(\cQ^{(q)}|\cA^U_\infty)+2q^2\log |\cQ|
%\end{align*}
As in the proof of Proposition \ref{prop5}, it follows from \eqref{sttbdd2} that
\begin{align*}
    \frac{1}{q}&H_{\mu_i}(\cQ^{(q)}|\cA^U_\infty)
    \ge \frac{1}{k_i}H_{\nu_i}(\cQ^{(k_i)}|\cA^U_\infty)-\frac{2q\log |\cQ|}{k_i}\\
    &\ge \frac{1}{k_i}\Bigl((1-\beta^{\frac{1}{2}})(\log|S_i|-D\beta^{\frac{1}{2}}k_i-\log{C}-\frac{2}{e}+\log(1-\beta^{\frac{1}{2}}))
-2q\log|\cQ|\Bigr).
\end{align*}
Now we can take $i\to\infty$ because the atoms $Q$ of $\cQ$ and hence of $\cQ^{(q)}$, satisfy $\mu^\gamma(\partial Q)=0$. Also, the constants $C$, $\beta$, and $|\cQ|$ are independent to $k_i$. Thus it follows from the inequality \eqref{eqn1} that
\begin{align*}
    \frac{1}{q}H_{\mu^\gamma}(\overline{\cQ}^{(q)}|\overline{\cA^U_\infty})
    &\ge (1-\beta^{\frac{1}{2}})\left(\frac{m+n}{mn}(\dim_H \mb{Bad}^{b}(\eps)-\gamma)-D\beta^{\frac{1}{2}}\right).
\end{align*}
%and by taking $\rho\to 0$, we have
%$$\frac{1}{q}H_{\mu^\gamma}(\overline{\cQ}^{(q)}|\overline{\cA^U_\infty})
%    \ge (1-\mu^\gamma(\overline{Q_\infty})^{\frac{1}{2}})\left(\frac{m+n}{mn}(\dim_H \mb{Bad}^{b}(\eps)-\gamma)
%    -D\mu^\gamma(\overline{Q_\infty})^{\frac{1}{2}}\right).$$
By taking $\beta\to\ov{\mu}(\overline{Q_\infty})$ and $\gamma=\gamma_j \to 0$, the inequality \eqref{sttlowbdd'} follows.
%$$\frac{1}{q}H_{\ov{\mu}}(\overline{\cQ}^{(q)}|\overline{\cA^U_\infty})
%    \ge (1-\ov{\mu}(\overline{Q_\infty})^{\frac{1}{2}})\left(\frac{m+n}{mn}\dim_H \mb{Bad}^{b}(\eps)
%    -D\ov{\mu}(\overline{Q_\infty})^{\frac{1}{2}}\right).$$
\end{proof}

\subsection{Effective equidistribution and the proof of Theorem \ref{corb1}}
In this subsection, we recall some effective equidistribution results which are necessary for the proof of Theorem \ref{corb1}. Let $\mathfrak{g}=\operatorname{Lie}\,G(\bR)$ and choose an orthonormal basis for $\mathfrak{g}$. Define the (left) differentiation action of $\mathfrak{g}$ on $C_c^\infty(X)$ by $Zf(x)=\frac{d}{dt}f(\textrm{exp}(tZ)x)|_{t=0}$ for $f\in C_c^\infty(X)$ and $Z$ in the orthonormal basis. This also defines for any $l\in\bN$, $L^2$-Sobolev norms $\cS_l$ on $C^\infty_c(Y)$:
\eqlabel{Sobol}{\cS_l(f)^2\defn\displaystyle\sum_{\cD}\|\textrm{ht}\circ \pi ^{l}\cD(f)\|^2_{L^2},}
where $\cD$ ranges over all the monomials in the chosen basis of degree $\leq l$ and $\textrm{ht } \circ \pi$ is the function assigning 1 over the smallest length of a vector in the lattice corresponding to the given grid.
Let us define the function $\zeta : (\bT^{d}\setminus \bQ^d)\times\bR^{+}\to\bN$ measuring the Diophantine property of $b$:
\eq{
\zeta(b,T):= \min\left\{N\in\bN : \min_{1\leq q\leq N}\|qb\|_{\bZ}\leq\frac{T^2}{N} \right\}.
}
Then there exists a sufficiently large $l\in\bN$ such that the following equidistribution theorems hold.

\begin{thm}\cite[Theorem 1.3]{Kim21}\label{Teffirr}
Let $K$ be a bounded subset in $\SL_d(\bR)$ and $V\subset U$ be a fixed neighborhood of the identity in $U$ with smooth boundary and compact closure. Then, for any $t\ge 0$, $f\in C_c^\infty(Y)$, and $y=gw(b)\Gamma$ with $g\in K$ and $b\in\bT^d\setminus\bQ^d$, there exists a constant $\alpha_1>0$ depending only on $d$ and $V$ so that
\eqlabel{effirr}{\frac{1}{m_U(V)}\int_V f(a_tuy)dm_U(u)=\int_Y fdm_Y+O(\cS_l(f)\zeta(b,e^{\frac{t}{2m}})^{-\alpha_1}).}
The implied constant in \eqref{effirr} depends only on $d$, $V$, and $K$. 
\end{thm}

For $q\in\bN$, define 
\begin{align*}
X_q &:= \left\{gw(\mb{p}/q)\Gamma\in Y : g\in \SL_{d}(\bR), \mb{p}\in\bZ^d, \gcd(\mb{p},q)=1\right\},\\
\Gamma_q &:= \{\gamma\in SL_{d}(\bZ): \gamma e_1 \equiv e_1 \;(\bmod\; q)\}.
\end{align*}
\begin{lem}\label{IdLem}
The subspace $X_q\subset Y$ can be identified with the quotient space $\SL_d(\bR)/\Gamma_{q}$. In particular, this identification is locally bi-Lipschitz.
\end{lem}
\begin{proof}
The action $\SL_d(\bR)$ on $X_q$ by the left multiplication is transitive and $\operatorname{Stab}_{\SL_d(\bR)}(w(e_1 /q)\Gamma)=\Gamma_q$. To see the transitivity, it is enough to show the transitivity on each fiber, i.e., $$\SL_d(\bZ)e_1 \equiv \{\mb{p}\in\bZ^{d}: \gcd(\mb{p},q)=1\} \;(\bmod\; q).$$ Write $D=\gcd(\mb{p})$ and $\mb{p}'=\mb{p}/D$. Since $\gcd(D,q)=1$, there are $a,b\in\bZ$ such that $aD+bq=1$. Take $A\in M_{d,d}(\bZ)$ such that $\det(A)=D$ and $Ae_1=\mb{p}$. If we set $\mb{u}=b\mb{p}'+(a-1)Ae_2$, then by direct calculation, we have $\mb{p}+q\mb{u}=(A+\mb{u}\times{^{t}}(qe_1 + e_2))e_1$ and $A+\mb{u}\times{^{t}}(qe_1 + e_2)\in \SL_d(\bZ)$, which concludes the transitivity. Bi-Lipshitz property of the identification follows trivially since both $X_q$ and $\SL_d(\bR)/\Gamma_q$ are locally isometric to $\SL_d(\R)$.
\end{proof}

\begin{thm}\cite[Theorem 2.3]{KM12}\label{Teffrat}
For $q\in\bN$, let $\SL_d(\bR)/\Gamma_q\simeq X_q\subset Y$. Let $K$ and $V$ be as in Theorem \ref{Teffirr}. Then, for any $t\ge 0$, $f\in C_c^\infty(Y)$, and $y=gw(\frac{\bp}{q})\Gamma$ with $g\in K$ and $\bp\in\bZ^d$, there exists a constant $\alpha_2>0$ depending only on $d$ and $V$ so that
\eqlabel{effrat}{\frac{1}{m_U(V)}\int_V f(a_tuy)dm_U(u)=\int_{X_q} fdm_{X_q}+O(\cS_l(f)[\Gamma_1:\Gamma_q]^{\frac{1}{2}}e^{-\alpha_2 t}).}
The implied constant in \eqref{effrat} depends only on $d$, $V$, and $K$.
\end{thm}
\begin{proof}
This result was obtained in \cite[Theorem 2.3]{KM12} in the case $q=1$. For general $q$, we refer the reader to \cite[Theorem 5.4]{KM21} which gave a sketch of required modification. \cite[Theorem 5.4]{KM21} is actually stated for different congruence subgroups from our $\Gamma_q$, but the modification still works.
\end{proof}

Since we assume the unweighted setting, $\cL_{\eps}=\{y\in Y : \forall v\in \Lambda_{y},\ \|v\|\geq \eps^{1/d} \}$.

\begin{lem}\label{bvol}
For any small enough $\eps>0$ and $q\in\bN$, $m_Y(Y_{\leq\eps^{-1}}\setminus\cL_\eps)\asymp\eps$ and $m_{X_q}(Y_{\leq\eps^{-1}}\setminus\cL_\eps)\gg q^{-d}\eps$.
\end{lem}
\begin{proof}
Using Siegel integral formula \cite[Lemma 2.1]{MM11} with $f=\mathds{1}_{B_{\eps^{1/d}}(0)}$, which is the indicator function on $\eps^{1/d}$-ball centered at $0$ in $\bR^d$, we have $m_Y(Y_{\leq\eps^{-1}}\setminus\cL_\eps)\ll\eps$. On the other hands, by \cite[Theorem 1]{A15} with $A=B_{\eps^{1/d}}(0)$, we have $m_{Y}(\cL_\eps)< \frac{1}{1+2^d \eps}$. It follows from Siegel integral formula on $X$ that $m_{Y}(Y_{>\eps^{-1}})=m_{X}(X_{>\eps^{-1}})\leq 2^d\eps^d$. Since $d\geq 2$, we have
\eq{
m_Y (Y_{\leq\eps^{-1}}\setminus\cL_\eps) \geq m_Y (Y\setminus\cL_\eps)-m_Y (Y_{>\eps^{-1}}) > \frac{2^d \eps}{1+2^d \eps}-2^d \eps^d \gg \eps
} for small enough $\eps >0$, which concludes the first assertion.

To prove the second assertion, observe that for any $x\in X_{>\eps^{-1/d}}$, 
there exists $g\in \SL_d(\bR)$ such that $x=g\SL_d(\bZ)$ and $\|ge_1\|\leq \eps^{1/d}$.
Then $gw(\frac{e_1}{q})\Gamma \in \pi_{q}^{-1}(x)\cap (Y\setminus\cL_\eps)$, where $\pi_q : X_q \to X$ is the natural projection. Since $|\pi_{q}^{-1}(x)|\leq q^{d}$ and $m_X (x\in X: \eps^{-1/d}<\textrm{ht}(x)\leq \eps^{-1}) \asymp \eps$, we have 
\eq{
m_{X_q}(Y_{\leq\eps^{-1}}\setminus\cL_\eps) \geq \frac{|\pi_{q}^{-1}(x)\cap (Y\setminus\cL_\eps)|}{|\pi_{q}^{-1}(x)|} m_X (x\in X: \eps^{-1/d}<\textrm{ht}(x)\leq \eps^{-1}) \gg q^{-d}\eps.
}

\end{proof}

\begin{prop}\label{btauest}
Let $\cA$ be a countably generated sub-$\sigma$-algebra of the Borel   
$\sigma$-algebra which is $a^{-1}$-descending and $U$-subordinate. Fix a compact set $K\subset Y$. 
Let $1<R'<R$, $k=\lfloor\frac{mn\log R'}{4d}\rfloor$. Suppose that $y\in a^{4k}K$ satisfies 
$B^{U,d_\infty}_{R'}\cdot y\subset[y]_\cA\subset B^{U,d_\infty}_{R}\cdot y$, where $B^{U,d_\infty}_r$ is the 
$d_\infty$-ball of radius $r$ around the identity in $U$.
For $\eps>0$, let $\Om\subset Y$ be a set satisfying $\Om\cup a^{-3k}\Om\subseteq\cL_{\frac{\eps}{2}}$. There exist $M,M'>0$ such that the following holds. 
If $R'\ge\eps^{-M'}$, then $$1-\tau^{\cA}_y(\Om)\gg \left(\frac{R'}{R}\right)^{mn}\eps^{dM+1},$$ 
where the implied constant depends only on $K$. 
\end{prop}

\begin{proof}
Denote by $V_y\subset U$ the shape of $\cA$-atom of $y$ so that $V_y\cdot y=[y]_{\cA}$. Set $V=B^{U,d_\infty}_{1}$. 
Since $\frac{mn\log R'}{d}-4\leq4k\leq \frac{mn\log R'}{d}$, we have
$$B^{U,d_\infty}_{e^{-\frac{4d}{mn}}R'} \subseteq a^{4k}Va^{-4k} = B_{e^{\frac{d}{mn}4k}}^{U,d_\infty}\subseteq 
B_{R'}^{U,d_\infty} \subseteq V_y.$$ 
It follows that
\eqlabel{tauest1}{
\begin{aligned}
1-&\tau_y^{\cA}(\Om)=\frac{1}{m_U(V_y)}\int_{V_y} \mathds{1}_{Y\setminus\Om}(uy)dm_U(u)\geq \frac{1}{m_U(B_R^{U,d_\infty})}\int_{a^{4k}Va^{-4k}} \mathds{1}_{Y\setminus\Om}(uy)dm_U(u)\\
&\geq e^{-4d}\left(\frac{R'}{R}\right)^{mn}\left(\frac{1}{m_U(a^{4k}Va^{-4k})}\int_{a^{4k}Va^{-4k}} \mathds{1}_{Y\setminus\Om}(uy)dm_U(u)\right)\\
&=e^{-4d}\left(\frac{R'}{R}\right)^{mn}\left(\frac{1}{m_U(V)}\int_{V} \mathds{1}_{Y\setminus\Om}(a^{4k}ua^{-4k}y)dm_U(u)\right).
\end{aligned}
}

It remains to show that 
\eqlabel{equidisest}{\frac{1}{m_U(V)}\int_{V} \mathds{1}_{Y\setminus\Om}(a^{4k}ua^{-4k}y)dm_U(u) \gg \eps^{dM+1}.}
We will approximate the characteristic function in the above integrand by a smooth function $\psi$ and use effective equidistribution results from Theorem \ref{Teffirr} and \ref{Teffrat}. Since $\pi(K) \subset X$ is compact, we can choose $g_0 \in SL_d(\bR)$ such that $\|g_0\|<C_K$ with a constant $C_K>0$ 
depending only on $K$, and $a^{-4k}y=g_0w(b_0)\Gamma$ with $b_0 \in \bR^d$.
For the constants $\alpha_1$ in Theorem \ref{Teffirr} and $\alpha_2$ in Theorem \ref{Teffrat}, 
let $\alpha=\min(\alpha_1,\alpha_2)$ and $M=\frac{1}{\alpha}\left(2+l+\frac{\dim G}{2d}\right)$. 
By \cite[Lemma 2.4.7(b)]{KM96} with $r=C\eps^{\frac{1}{d}}<1$, we can take the approximation function 
$\theta\in C_{c}^{\infty}(G)$ of the identity such that $\theta\ge 0$, $\Supp \theta\subseteq B^{G}_r(id)$, $\int_G \theta=1$, 
and $\cS_l(\theta)\ll \eps^{-\frac{1}{d}(l+\frac{\dim G}{2})}$. 
Let $\psi=\theta*\mathds{1}_{Y_{\leq \eps^{-1}}\setminus\cL_{\frac{\eps}{4}}}$, then we have 
$\mathds{1}_{Y_{\leq (2\eps)^{-1}}\setminus\cL_{\frac{\eps}{8}}}\leq\psi\leq\mathds{1}_{Y_{\leq 2\eps^{-1}}
\setminus\cL_{\frac{\eps}{2}}}$. Moreover, using Young's inequality, its Sobolev norm is bounded as follows:
\eqlabel{psiSobol}{
\begin{aligned}
\cS_l(\psi)^2&=\displaystyle\sum_{\cD}\|(\textrm{ht}\circ\pi)^l\cD(\psi)\|^2_{L^2}
\ll\eps^{-l}\displaystyle\sum_{\cD}\|\cD(\theta)*\mathds{1}_{Y_{\leq \eps^{-1}}\setminus\cL_{\frac{\eps}{4}}}\|^2_{L^2}\\
&\ll\eps^{-l}\|\mathds{1}_{Y_{\leq \eps^{-1}}\setminus\cL_{\frac{\eps}{4}}}\|^2_{L^1}\displaystyle\sum_{\cD}\|\cD(\theta)\|^2_{L^2}\ll\eps^{-l}\cS_l(\theta)^2,
\end{aligned}}
hence $\cS_l(\psi)\ll \eps^{-\frac{l}{2}}\cS_l(\theta)\leq \eps^{-(l+\frac{\dim G}{2d})}$.

We will prove \eqref{equidisest} applying Theorem \ref{Teffirr} and \ref{Teffrat} 
to the following two cases, respectively:
\[
\textbf{Case (i)}\quad \zeta(b_0,e^{\frac{2k}{m}})\ge\frac{r_0}{C_K C_0}\eps^{-M}\qquad\text{and}\qquad
\textbf{Case (ii)}\quad \zeta(b_0,e^{\frac{2k}{m}})<\frac{r_0}{C_K C_0}\eps^{-M}.\qquad
\]

%\begin{enumerate}[label=(\roman*)]x
%\item[\textbf{Case (i)}] $\zeta(b_0,e^{\frac{2k}{m}})\ge\frac{r_0}{C_K C_0}\eps^{-M}$,
%\item[\textbf{Case (ii)}] $\zeta(b_0,e^{\frac{2k}{m}})<\frac{r_0}{C_K C_0}\eps^{-M}$.
%\end{enumerate}

\textbf{Case (i):}
Applying Theorem \ref{Teffirr}, we have
\eq{
\begin{aligned}
&\frac{1}{m_U(V)}\int_V\mathds{1}_{Y\setminus\Om}(a^{4k}ua^{-4k}y)dm_U(u)
\geq\frac{1}{m_U(V)}\int_V \psi(a^{4k}ua^{-4k}y)dm_U(u)\\
&=\frac{1}{m_U(V)}\int_V \psi(a^{4k}ug_0w(b_0)\Gamma)dm_U(u)
=\int_Y\psi dm_Y+O(\cS_l(\psi)\zeta(b_0,e^{\frac{2k}{m}})^{-\alpha})\\
&\geq m_Y(Y_{\leq (2\eps)^{-1}}\setminus\cL_{\frac{\eps}{8}})+O(\eps^{-(l+\frac{\dim G}{2d})}\eps^{M\alpha}).
\end{aligned}}
It follows from Lemma \ref{bvol} and $M\alpha=2+(l+\frac{\dim G}{2d})$ that
\eq{
\frac{1}{m_U(V)}\int_{V} \mathds{1}_{Y\setminus\Om}(a^{4k}ua^{-4k}y)dm_U(u)\geq 
m_Y(Y_{\leq (2\eps)^{-1}}\setminus\cL_{\frac{\eps}{2}})+O(\eps^2)\asymp \eps \geq \eps^{dM+1}.}
%Hence, $1-\tau^{\cA}_y(\Om)\gg\eps\left(\frac{R'}{R}\right)^{mn}$ by \eqref{tauest1} and \eqref{tauest3}.

\textbf{Case (ii):}
The assumption $\zeta(b_0,e^{\frac{2k}{m}})<\frac{r_0}{C_K C_0}\eps^{-M}$ implies that there exists 
$q\leq \frac{r_0}{C_K C_0} \eps^{-M}$ such that $\|qb_0\|_{\bZ}\leq q^{2}e^{-\frac{2k}{m}}$, whence 
\eqlabel{eqdistbq}{\|b_0-\frac{\bp}{q}\|\leq qe^{-\frac{2k}{m}}\leq \frac{r_0}{C_K C_0}\eps^{-M}e^{-\frac{2k}{m}}}
for some $\bp\in\bZ^d$. Let $y'=a^{4k}g_0w(\frac{\bp}{q})\Gamma$. Then for any $u\in V$,
\eq{\begin{aligned}
d_Y(a^kua^{-4k}y, &a^kua^{-4k}y') \leq d_G(a^k u g_0 w(b_0), a^k u g_0 w(\frac{\mb{p}}{q}))=d_G \left(\left(\begin{matrix} I_d & a^k u g_0 (b_0 - \frac{\mb{p}}{q}) \\ & 1 \\ \end{matrix}\right), id\right)\\
& \leq C_0 d_\infty \left(\left(\begin{matrix} I_d & a^k u g_0 (b_0 - \frac{p}{q}) \\ & 1 \\ \end{matrix}\right), id\right)
\leq C_0 e^{\frac{k}{m}} \|g_0\| \|b_0-\frac{\mb{p}}{q}\| \leq r_0 \eps^{-M}e^{-\frac{k}{m}}.
\end{aligned}}
by \eqref{eqbilip} and \eqref{eqdistbq}.
%Since by \eqref{eqdistbq},
%$$d_\infty \left(\left(\begin{matrix} I_d & a^k u g_0 (b_0 - \frac{p}{q}) \\ & 1 \\ \end{matrix}\right), id\right)  = e^{\frac{k}{m}}\|g_0\|
%\|b_0-\frac{\mb{p}}{q}\| < \frac{r_0}{C_0},$$
%it follows from \eqref{eqbilip}, \eqref{eqdistbq}, and \eqref{eqredtoG} that
%\eq{
%\begin{split}
%d_Y(a^kua^{-4k}y,a^kua^{-4k}y')
%&\leq C_0 d_\infty \left(\left(\begin{matrix} I_d & a^k u g_0 (b_0 - \frac{p}{q}) \\ & 1 \\ \end{matrix}\right), id\right)\\
%&\leq C_0 e^{\frac{k}{m}} \|g_0\| \|b_0-\frac{\mb{p}}{q}\| \leq r_0 \eps^{-M}e^{-\frac{k}{m}}.
%\end{split}
%}
Hence, we have
\eqlabel{psidiff}{\begin{aligned}
|\psi(a^kua^{-4k}y)-\psi(a^kua^{-4k}y')|&\ll\cS_l(\psi)d_Y(a^kua^{-4k}y,a^kua^{-4k}y')\ll \cS_l(\psi)\eps^{-M}e^{-\frac{k}{m}}.
\end{aligned}}
It follows from the assumption $a^{-3k}\Om\subseteq\cL_{\frac{\eps}{2}}$, \eqref{psidiff}, and Theorem \ref{Teffrat} that
\[
\begin{split}
\frac{1}{m_U(V)}\int_V&\mathds{1}_{Y\setminus\Om}(a^{4k}ua^{-4k}y)dm_U(u)
=\frac{1}{m_U(V)}\int_V\mathds{1}_{Y\setminus a^{-3k}\Om}(a^{k}ua^{-4k}y)dm_U(u)\\
&\geq\frac{1}{m_U(V)}\int_V \psi(a^{k}ua^{-4k}y)dm_U(u)\\
&=\frac{1}{m_U(V)}\int_V \psi(a^{k}ua^{-4k}y')dm_U(u)+O(\cS_l(\psi)\eps^{-M}e^{-\frac{k}{m}})\\
&=\int_{X_q}\psi dm_Y+O(\cS_l(\psi)q^{\frac{d}{2}}e^{-\alpha k}+\cS_l(\psi)\eps^{-M}e^{-\frac{k}{m}})\\
&\geq m_{X_q}(Y_{\leq (2\eps)^{-1}}\setminus\cL_{\frac{\eps}{8}})+O(\eps^{-(l+\frac{\dim G}{2d})-\frac{dM}{2}}e^{-\alpha k}
+\eps^{-(l+\frac{\dim G}{2d})-M}e^{-\frac{k}{m}}).
\end{split}
\]
%We are using \eqref{psidiff} for the fourth line, and Theorem \ref{Teffrat} for the fifth line.
Let $M'=\min\left(\frac{4d}{\alpha}(l+\frac{\dim G}{2d}+\frac{3dM}{2}+2), 4dm(l+\frac{\dim G}{2d}+(d+1)M+2)\right)$. 
If $R'>\eps^{-M'}$, then $e^{-4dk}<e^{4d}\eps^{M'}$, so $\eps^{-(l+\frac{\dim G}{2d})-\frac{dM}{2}}e^{-\alpha k}\ll \eps^{dM+2}$
and $\eps^{-(l+\frac{\dim G}{2d})-M}e^{-\frac{k}{m}}\ll \eps^{dM+2}$. Combining this with Lemma \ref{bvol}, it follows that
\eq{
\frac{1}{m_U(V)}\int_V\mathds{1}_{Y\setminus\Om}(a^{4k}ua^{-4k}y)dm_U(u)\gg q^{-d}\eps+O(\eps^{dM+2})
\gg \eps^{dM+1}+O(\eps^{dM+2})\gg \eps^{dM+1}.
}
\end{proof}

\begin{proof}[Proof of Theorem \ref{corb1}]
 For fixed $b$, let $\eta_0=2(m+n)(1-\frac{\dim_H \mb{Bad}^b(\eps)}{mn})$ as in Proposition \ref{prop2}. 
 It is enough to consider the case when $\mb{Bad}^b(\eps)$ is sufficiently close to the full dimension $mn$, 
 so we may assume $\dim_H\mb{Bad}^b(\eps)>\dim_H\mb{Bad}^0(\eps)$ and $\eta_0\leq 0.01$. 
 By Proposition \ref{prop2}, there is an $a$-invariant measure $\ov{\mu}\in\crly{P}(\overline{Y})$ such that 
 $\Supp\ov{\mu}\subseteq \cL_\eps\cup(\overline{Y}\setminus Y)$, and 
 $\pi_*\ov{\mu}(\overline{X}\setminus\mathfrak{S}_{\eta'})\leq \eta'$ for any $\eta_0\leq\eta'\leq 1$. 
 We also have $a$-invariant $\mu\in\crly{P}(Y)$ and $0\leq\widehat{\eta}\leq\eta_0$ such that
 $$\ov{\mu}=(1-\widehat{\eta})\mu+\widehat{\eta}\del_\infty.$$
 In particular, for $\eta'=0.01$, we have $\mu(\pi^{-1}(\mathfrak{S}_{0.01}))\geq 0.99$. 
 We can choose $0<r<1$ such that $Y(r)\supset\pi^{-1}(\mathfrak{S}_{0.01})$. 
 Note that the choice of $r$ is independent of $\eps$ and $b$ since $\mathfrak{S}_{0.01}$ is constructed in 
 Proposition \ref{KKLM'} independent to $\eps$ and $b$.

Let $\cA^U$ be as in Proposition \ref{algebracst} for $\mu$, $r_0$, and $L=U$, 
and let $\cA^U_\infty$ be as in \eqref{eq:tailalg}.
It follows from (\ref{entropy}) of Proposition \ref{prop2} that
\eq{h_{\ov{\mu}}(a|\overline{\cA^U_\infty})\ge(1-\widehat{\eta}^{\frac{1}{2}})(d-\frac{1}{2}\eta_0-d\widehat{\eta}^{\frac{1}{2}}).}
By the linearlity of the entropy function with respect to the measure, we have
\eqlabel{entlow}{
h_\mu(a|\cA^U_\infty)\ge(1+\widehat{\eta}^{\frac{1}{2}})^{-1}(d-\frac{1}{2}\eta_0-d\widehat{\eta}^{\frac{1}{2}})
\ge d-2d\widehat{\eta}^{\frac{1}{2}}-\frac{1}{2}\eta_0.
}

On the other hand, we shall get an upper bound of $h_{\mu}(a|\cA^U_\infty)$ from Proposition \ref{algexiA} and 
Corollary \ref{effELcor}. By Lemma \ref{Exceptional}, there exists $0<\del<\min((\frac{cr_0}{16d_0})^2,r)$ such that 
$\mu(E_\del)<0.01$. Note that since $r_0$ depends only on $G$, the constants $C_1,C_2>0$ in Lemma \ref{Exceptional}
depend only on $a$ and $G$, hence $\del$ is independent of $\eps$ even if the set $E_\del$ depends on $\eps$. 
We write $Z=Y(r)\setminus E_\del$ for simplicity. Note that $\mu(Z)\ge\mu(Y(r))-\mu(E_\del)>0.98$. 

By Proposition \ref{algebracst}, $[y]_{\cA^U}\subset B^U_{r_0}\cdot y$ for all $y\in Y$, and
$B_\del^U\cdot y\subset[y]_{\cA^U}$ for all $y\in Z$ since $\del <r$. 
It follows from \eqref{eqbilip} that 
\eqlabel{eqUatombound}{\forall y \in Y,\ [y]_{\cA^U}\subset B^{U,d_\infty}_{C_0 r_0}\cdot y \qquad\text{ and }
\qquad \forall y \in Z,\ B^{U,d_\infty}_{\del/C_0}\cdot y\subset[y]_{\cA^U},} 
%\[
%B^{U,d_\infty}_{\del/C_0}\cdot y\subset[y]_{\cA^U}\subset B^{U,d_\infty}_{C_0 r_0}\cdot y,
%\]
where $B^{U,d_\infty}_r$ is the $d_\infty$-ball of radius $r$ around the identity in $U$.
For simplicity, we may assume that $r_0 < \frac{1}{C_0}$ by choosing $r_0$ small enough.

Let $M$ and $M'$ be the constants in Proposition \ref{btauest}, $r'=1-\frac{1}{2^{1/d}}$, $R'=\eps^{-M'}$, 
$R=e^{\frac{mn}{d}}\frac{C_0}{\del}R'$, and $k=\lfloor\frac{mn\log R'}{4d}\rfloor$. Let $\cA_1=a^{-j_1}\cA^U$ and 
$\cA_2=a^{j_2}\cA^U$, where
\[
j_1=\lceil-\frac{mn}{d}\log r' \rceil \qquad\text{and}\qquad 
j_2=\lceil-\frac{mn}{d}\log\frac{\del}{C_0 R'}\rceil.
\]
%\eq{\begin{aligned}
%j_1&=\lceil-\frac{mn}{d}\log\left(r^{-1}(1-2^{\frac{1}{d}})\eps^{\frac{1}{d}}\right)\rceil,\\ j_2&=\lceil-\frac{mn}{d}\log(\del\eps^{M'})\rceil.
%\end{aligned}}
By \eqref{eqUatombound}, we have that for any $y\in Y$,
\eqlabel{eqUatomupper}{
[y]_{\cA_1}=a^{-j_1}[a^{j_1}y]_{\cA^U} \subset a^{-j_1}B_{1}^{U,d_\infty}a^{j_1} \cdot y \subset 
B_{r'}^{U,d_\infty}\cdot y.
}
%Hence, we have $[y]_{\cA_1} \subset B_{r'}^{U,d_\infty} \cdot y$ for any $y\in a^{-j_1}Z$.
Similarly, it follows from \eqref{eqUatombound} that 
%Then for $y\in Z$, the atoms with respect to $\cA_1$ and $\cA_2$ satisfy
%$$[y]_{\cA_1}\subset B^{U}_{r'} \cdot y,$$
$B^{U,d_\infty}_{R'}\cdot y\subset[y]_{\cA_2}\subset B^{U,d_\infty}_R\cdot y$ for any $y\in a^{j_2}Z$.

%Observe that for any $z\in Z$,
%\[
%[a^{-j_1}z]_{\cA_1}=a^{-j_1}[z]_{\cA^U} \subset a^{-j_1}B_{1}^{U,d_\infty}a^{j_1} \cdot a^{-j_1}z \subset 
%B_{r'}^{U,d_\infty}\cdot a^{-j_1}z.
%\]
%Hence, we have $[y]_{\cA_1} \subset B_{r'}^{U,d_\infty} \cdot y$ for any $y\in a^{-j_1}Z$.
%Similarly, it follows that 
%%Then for $y\in Z$, the atoms with respect to $\cA_1$ and $\cA_2$ satisfy
%%$$[y]_{\cA_1}\subset B^{U}_{r'} \cdot y,$$
%$B^{U,d_\infty}_{R'}\cdot y\subset[y]_{\cA_2}\subset B^{U,d_\infty}_R\cdot y$ for any $y\in a^{j_2}Z$.

Let $\Om=B^{U,d_\infty}_{r'}\Supp\mu$. 
For any $v\in \bR^d$ with $\|v\|\geq \eps^{1/d}$ and $u \in B^{U,d_\infty}_{r'}$, 
$$\|uv\| \geq \|v\| - \|(u-id)v\| \geq (1-r')\eps^{1/d} = (\eps/2)^{1/d},$$ hence
$\Om\subseteq B^{U,d_\infty}_{r'} \cL_\eps\subseteq\cL_{\frac{\eps}{2}}$. 
Since $\Supp\mu$ is an $a$-invariant set, we also have
$$a^{-3k}\Om=(a^{-3k}B^{U,d_\infty}_{r'}a^{3k})a^{-3k}\Supp\mu\subseteq(a^{-3k}B^{U,d_\infty}_{r'}a^{3k})
\cL_\eps\subseteq\cL_{\frac{\eps}{2}}.$$
Applying Proposition \ref{btauest} with $K=Y(r)$, $\cA=\cA_2$, and the same $R'$, $R$, $\Om$ as we just defined, 
for any $\eps>0$ and $y\in a^{4k}Y(r)\cap a^{j_2}Z$,
\eqlabel{taulow}{1-\tau_y^{\cA_2}(\Om)\gg \eps^{dM+1}}
since $\frac{R'}{R}$ is bounded below by a constant independent of $\eps$. 
 
By Proposition \ref{algexiA}, we have
\eqlabel{eqest1}{
(j_1+j_2)(d-h_\mu(a|\cA^U_\infty)) = (j_1+j_2)(d-H_\mu(\cA^U | a\cA^U))
= (j_1+j_2)d-H_\mu (\cA_1 | \cA_2).
}
Note that the maximal entropy contribution of $U$ for $a^{j_1+j_2}$ is $(j_1+j_2)d$.
Using \eqref{eqUatomupper},
it follows from Corollary \ref{effELcor} with $\cA=\cA_1$, $K=Y$, and $B=B_{r'}^{U,d_\infty}$ that 
\eqlabel{eqest2}{
(j_1+j_2)d-H_\mu (\cA_1 | \cA_2) \geq - \int_Y \log\tau_y^{\cA_2}(\Om)d\mu(y)
}
Combining \eqref{taulow}, \eqref{eqest1}, and \eqref{eqest2}, since $\mu(a^{4k}Y(r)\cap a^{j_2}Z)\geq \frac{1}{2}$, we have
\eq{
(j_1+j_2)(d-h_\mu(a|\cA^U_\infty))
\ge\int_{a^{4k}Y(r) \cap a^{j_2}Z}(1-\tau_y^{\cA_2}(\Om))d\mu(y)
\gg \frac{1}{2}\eps^{dM+1}.
}
It follows from \eqref{entlow} and $j_1+j_2\asymp \log(1/\eps)$ that
\eq{\eta_0^{\frac{1}{2}}\gg 2d\widehat{\eta}^{\frac{1}{2}}+\frac{1}{2}\eta_0 \geq d-h_\mu(a|\cA^U_\infty)\gg\eps^{dM+2}.}
Since $\eta_0=2(m+n)(1-\frac{\dim_H \mb{Bad}^b(\eps)}{mn})$, we have
$$mn-\dim_H \mb{Bad}'(\eps)\ge c_0\eps^{2(dM+2)}$$
for some constant $c_0>0$ depending only on $d$.
\end{proof}

\section{Characterization of singular on average property and Dimension esitimates}\label{sec6}
In this section, we will show (\ref{S3})$\implies$(\ref{S1}) in Theorem \ref{thmA1}.
Let $A\in M_{m,n}$ and consider two subgroups
\eq{
G(A)\defn A\bZ^n + \bZ^m \subset \bR^m \quad\text{and}\quad G({^{t}A})\defn {^{t}A}\bZ^m + \bZ^n \subset \bR^n.
}
If we view alternatively $G(A)$ as a subgroup of classes modulo $\bZ^m$, lying in the $m$-dimensional torus $\bT^m$,
Kronecker's theorem asserts that $G(A)$ is dense in $\bT^m$ if and only if the group $G({^{t}A})$ has maximal rank $m+n$ over $\bZ$ (See \cite[Chapter \rom{3}, Theorem \rom{4}]{Cas57}). Thus, if $\text{rank}_\bZ (G({^{t}A}))<m+n$, then $\text{Bad}_A(\eps)$ has full Hausdorff dimension for any $\eps>0$. Hence, throughout this section, we consider only matrices $A$ for which $\text{rank}_\bZ (G({^{t}A}))=m+n$.
\subsection{Best approximations}
We set up a weighted version of the best approximations following \cite{CGGMS}. (See also \cite{BL} and \cite{BKLR} for the unweighted setting.)

Given $A\in M_{m,n}$, we denote 
\[ M(\mb{y})= \inf_{\mb{q}\in\bZ^n} \|{^{t}A}\mb{y}-\mb{q}\|_\mb{s}.\]
Our assumption that $\text{rank}_\bZ (G({^{t}A}))$ equals $m+n$ guarantees that $M(\mb{y})>0$ for all non-zero $\mb{y}\in\bZ^m$. One can construct a sequence of $\mb{y}_i\in\bZ^n$ called \textit{a sequence of weighted best approximations to $^{t}A$}, which satisfies the following properties:
\begin{enumerate}
\item\label{Best1} Setting $Y_i=\|\mb{y}_i\|_\mb{r}$ and $M_i=M(\mb{y}_i)$, we have \[ Y_1<Y_2<\cdots\quad \text{and}\quad  M_1>M_2>\cdots, \]
\item\label{Best2} $M(\mb{y})\geq M_i$ for all non-zero $\mb{y}\in\bZ^m$ with $\|\mb{y}\|_\mb{r}<Y_{i+1}$.
\end{enumerate}

The sequence $(Y_i)_{i\geq1}$ has at least geometric growth. 

\begin{lem}\cite[Proof of Lemma 4.3]{CGGMS}\label{lem:CGGMS}
There exists a positive integer $V$ such that for all $i\geq 1,$ \[ Y_{i+V}\geq 2Y_i.\] 
\end{lem}
 In particular, there exist $c>0$ and $\ga>1$ such that  for all $i\geq1$ $Y_{i}\geq c\ga^i$.

\begin{rem}\label{rem:DT} 
%\item The lemma in the above lemma can be found in the proof of \cite[Lemma 4.3]{CGGMS}.
From the weighted Dirichlet's Theorem (see \cite[Theorem 2.2]{Kle98}), one can check that $M_{k}Y_{k+1}\leq 1 $ for all $k\geq1$.
\end{rem}

\subsection{Characterization of singular on average property}

In this section, we will characterize the singular on average property in terms of best approximations. At first, we will show $A$ is singular on average if and only if $^{t}A$ is singular on average. To do this, following \cite[Chapter \rom{5}]{Cas57}, we prove a transference principle between two homogeneous approximations with weights. See also \cite{GE15,Ger20}.

\begin{defi}\label{Parall}
Given positive numbers $\lam_1,\dots,\lam_d$, consider the parallelepiped 
\eq{
\cP=\left\{ \mb{z}=(z_1,\dots,z_d)\in \bR^d:|z_i|\leq \lam_i,\ i=1,\dots,d \right\}.
} We call the parallelepiped 
\eq{
\cP^{*}=\left\{\mb{z}=(z_1,\dots,z_d)\in \bR^d:|z_i|\leq\frac{1}{\lam_i}\prod_{j=1}^{d}\lam_j,\ i=1,\dots,d \right\}
} the \textit{pseudo-compound} of $\cP$.
\end{defi}

\begin{thm}\label{TranThm}\cite{GE15}
Let $\cP$ be as in Definition \ref{Parall} and let $\Lambda$ be a full-rank lattice in $\bR^d$. Then 
\eq{
\cP^{*}\cap \Lambda^{*} \neq \{\mb{0}\} \implies c\cP\cap\Lambda \neq\{\mb{0}\},
} where $c=d^{\frac{1}{2(d-1)}}$ and $\Lambda^{*}$ is the dual lattice of $\Lambda$, i.e., $\Lambda^*=\{x\in\bR^d : x\cdot y 
\in \bZ \text{ for all } y\in\Lambda\}.$
\end{thm}

\begin{cor}\label{TranCor1}
For positive integer $m,n$ let $d= m+n$ and let $A\in M_{m,n}$ and $0<\eps<1$ be given. For all large enough $X\geq 1$, if there exists a nonzero $\mb{q}\in \bZ^n$ such that
\eqlabel{Asol}{
\idist{A\mb{q}}_{\mb{r}}\leq \eps T^{-1} \quad \text{and}\quad \|\mb{q}\|_{\mb{s}} \leq T,
} then there exists a nonzero $\mb{y}\in \bZ^m$ such that 
\eqlabel{TranAsol}{
\idist{^{t}A\mb{y}}_{\mb{s}} \leq c^{(\frac{1}{r_m}+\frac{1}{s_n})}\eps^{\frac{r_m s_n}{s_n +r_1 (1-s_n)}}T_1^{-1} \quad\text{and}\quad \|\mb{y}\|_\mb{r} \leq T_1,
} where $c$ is as in Theorem \ref{TranThm} and $T_1=c^{\frac{1}{r_m}}\eps^{-\frac{r_m (1-s_n)}{s_n+r_1 (1-s_n)}}T$.
\end{cor}

\begin{proof}
Consider the following two parallelepipeds:
\eq{\begin{split}
\cQ&=\left\{\mb{z}=(z_1,\dots,z_d)\in \bR^{d}: \begin{split} &|z_i| \leq \eps^{r_i}T^{-r_i}, \quad i=1,\dots,m \\ &|z_{m+j}| \leq T^{s_j}, \quad j=1,\dots,n \end{split} \right\},\\
\cP&=\left\{\mb{z}=(z_1,\dots,z_d)\in \bR^{d}: \begin{split} &|z_i| \leq Z^{r_i}, \quad i=1,\dots,m \\ &|z_{m+j}| \leq \de^{s_j}Z^{-s_j}, \quad j=1,\dots,n \end{split} \right\},
\end{split}
} where \eq{
\de=\eps^{\frac{r_m s_n}{s_n +r_1 (1-s_n)}}\quad\text{and}\quad Z=\eps^{-\frac{r_m (1-s_n)}{s_n +r_1 (1-s_n)}}T.
}
Observe that the pseudo-compound of $\cP$ is given by 
\eq{
\cP^{*}=\left\{\mb{z}=(z_1,\dots,z_d)\in \bR^{d}: \begin{split} &|z_i| \leq \de Z^{-r_i}, \quad i=1,\dots,m \\ &|z_{m+j}| \leq \de^{1-s_j}Z^{s_j}, \quad j=1,\dots,n \end{split} \right\}
} and that $\cQ \subset \cP^{*}$ since $\eps^{r_i}T^{-r_i}\leq \de Z^{-r_i}$ and $T^{s_j}\leq \de^{1-s_j}Z^{s_j}$ for all $i=1,\dots,m$ and $j=1,\dots,n$.

Now, the existence of a nonzero solution $\mb{q}\in R_v^n$ of the inequalities \eqref{Asol} implies that 
$\left(\begin{matrix} I_m & A \\ & I_n \\ \end{matrix}\right) \bZ^d$ intersects $\cQ$, thus $\cP^{*}$.
%\eq{
%\cQ\cap \left(\begin{matrix} I_m & A \\ & I_n \\ \end{matrix}\right) \bZ^d \neq \{\mb{0}\},
%} which implies that
%\eq{
%\cP^{*}\cap \left(\begin{matrix} I_m & A \\ & I_n \\ \end{matrix}\right) \bZ^d \neq \{\mb{0}\}.
%} 
By Theorem \ref{TranThm}, $\left(\begin{matrix} I_m &  \\ -{^{t}A} & I_n \\ \end{matrix}\right) \bZ^d $ intersects
$c\cP$,
%\eq{
%c\cP\cap \left(\begin{matrix} I_m &  \\ -{^{t}A} & I_n \\ \end{matrix}\right) \bZ^d \neq \{\mb{0}\},
%} 
which concludes the proof of Corollary \ref{TranCor1}.

\end{proof}

\begin{cor}\label{eqCor}
Let $m,n$ be positive integers and $A\in M_{m,n}$. Then $A$ is singular on average if and only if $^{t}A$ is singular on average.
\end{cor}
\begin{proof}
It follows from Corollary \ref{TranCor1}.
\end{proof}

Now, we will characterize the singular on average property in terms of best approximation. Let $A\in M_{m,n}$ be a matrix and $(\mb{y}_k)_{k\geq 1}$ be a sequence of weighted best approximations to $^{t}A$ and write 
\[Y_k=\|\mb{y}_k\|_\mb{r},\quad M_k=\inf_{\mb{q}\in\bZ^n} \|{^{t}A}\mb{y}_k-\mb{q}\|_\mb{s}. \]

\begin{prop}\label{propA1}
Let $A\in M_{m,n}$ be a matrix and let $(\mb{y}_k)_{k\geq 1}$ be a sequence of best approximations to $^{t}A$. Then the following are equivalent:
\begin{enumerate}
\item\label{SS3} $^{t}A$ is singular on average.
\item\label{SS2} For all $\eps>0$, \eq{\lim\limits_{k\to\infty}\frac{1}{\log Y_{k}}\left|\{i\leq k:M_{i}Y_{i+1}>\eps\}\right|=0.} 
\end{enumerate}
\end{prop}

\begin{proof}
($\ref{SS3})\implies(\ref{SS2})$ : Let $0<\eps<1$. Observe that for each integer $X$ with $Y_{k} \leq T < Y_{k+1}$, the inequalities 
\eqlabel{inequal}{
\|{^{t}A}\mb{p}-\mb{q}\|_\mb{s} \leq \eps T^{-1}\quad \text{and}\quad 0 < \|\mb{p}\|_\mb{r} \leq T
}
have a solution if and only if $T\leq \frac{\eps}{M_k}$. Thus, for each integer $\ell\in[\log_2{Y_k},\log_2{Y_{k+1}})$ the inequalities \eqref{inequal} have no solutions for $T=2^\ell$ if and only if 
\eqlabel{NoSol}{
\log_2{\eps}-\log_2{M_k}<\ell<\log_2{Y_{k+1}}.
}

Now we assume that $^{t}A$ is singular on average. For given $\de>0$, if the set $\{k\in\bN:M_k Y_{k+1}>\de\}$ is finite, then it is done. Suppose the set $\{k\in\bN:M_k Y_{k+1}>\de\}$ is infinite and let \eq{\{k\in\bN:M_{k}Y_{k+1}>\de \}=\left\{j(1)<j(2)<\cdots<j(k)<\cdots:k\in\bN\right\}.} 
Set $\eps=\de/2$ and fix a positive integer $V$ in Lemma \ref{lem:CGGMS}. 
For an integer $\ell$ in $[\log_2{Y_{j(k)+1}}-1,\log_2{Y_{j(k)+1}})$, observe that
\[
\log_2{\eps}-\log_2{M_{j(k)}} < \log_2{Y_{j(k)+1}}-1.
\]
Hence the inequalities \eqref{inequal} have no solutions for $T=2^\ell$ by \eqref{NoSol}.
By Lemma \ref{lem:CGGMS}, $\log_2{Y_{j(k)+1+V}}-1\geq\log_2{Y_{j(k)+1}}$. So, we have $\log_2{Y_{j(k+V)+1}}-1\geq\log_2{Y_{j(k)+1}}$. 
Now fix $i=0,\cdots,V-1$. Then the intervals 
\[[\log_2{Y_{j(i+sV)+1}}-1,\log_2{Y_{j(i+sV)+1}}), \quad s=1,\cdots,k \]
are disjoint. Thus, for an integer $N\in[\log_2{Y_{j(i+kV)+1}},\log_2{Y_{j(i+(k+1)V)+1}})$, the number of $\ell$ in $\left\{1,\cdots,N\right\}$ such that \eqref{inequal} have no solutions for $T=2^\ell$ is at least $k$. Since $^{t}A$ is singular on average, 
\[
\frac{k}{\log_{2}{Y_{j(i+(k+1)V)+1}}}\leq\frac{1}{N}\left|\left\{\ell\in\left\{1,\cdots,N\right\}:\eqref{inequal}\text{ have no solutions for }T=2^\ell \right\}\right|
\]
tends to $0$ with $k$, which gives $\frac{i+1+kV}{\log_{2}{Y_{j(i+1+kV)}}}$ tends to $0$ with $k$ for all $i=0,\cdots,V-1$. Thus, we have $\frac{k}{\log_{2}{Y_{j(k)}}}$ tends to $0$ with $k$. 

For any $k\geq 1$, there is an unique positive integer $s_k$ such that \[j(s_k)\leq k < j(s_k +1),\] and observe that $s_k=|\{i\leq k : M_i Y_{i+1}>\de \}|$. Thus, by the monotonicity of $Y_k$, we have
\eq{
\lim_{k\to\infty}\frac{1}{\log_{2}Y_{k}}|\{i\leq k:M_{i}Y_{i+1}>\de\}|\leq \lim_{k\to\infty}\frac{s_k}{\log_{2} Y_{j(s_k)}}=0.
}
($\ref{SS2})\implies(\ref{SS3})$ : Given $0<\eps<1$, the number of integers $\ell$ in $[\log_2{Y_k},\log_2{Y_{k+1}})$ such that \eqref{inequal} have no solutions for $T=2^\ell$ is at most
\[
\lceil \log_2{M_{k}Y_{k+1}}-\log_2{\eps} \rceil \leq \log_2{M_{k}Y_{k+1}}-\log_2{\eps}+1.
\]
Thus, for an integer $N$ in $[\log_2{Y_k},\log_2{Y_{k+1}})$, we have
\begin{align*}
    \frac{1}{N} |\{\ell\in\{1,\cdots,N\}&:\eqref{inequal}\ \text{have no solutions for } T=2^\ell\}| \\
    &\leq\frac{1}{N}\sum_{i=1}^{k}\max\left(0,\log_2{M_{i}Y_{i+1}}-\log_2{\eps}+1\right)\\
    &\leq\frac{1}{\log_2{Y_k}}\sum_{i=1}^{k}\max\left(0,\log_2{M_{i}Y_{i+1}}-\log_2{\eps}+1\right).
\end{align*}
Since $M_{i}Y_{i+1}\leq1$ for each $i\geq1$,
\begin{align*}
    \frac{1}{\log_2{Y_k}}\sum_{i=1}^{k}&\max\left(0,\log_2{M_{i}Y_{i+1}}-\log_2{\eps}+1\right)\\
    &\leq\frac{1}{\log_2{Y_k}}\left(-\log_2{\eps}+1\right)|\{i\leq k : M_i Y_{i+1}>\eps/2\}|.  
\end{align*}
Therefore, $^{t}A$ is singular on average.

\end{proof}

\subsection{Modified Bugeaud-Laurent sequence}
%Let us start with the original version of the Bugeaud-Laurent sequence.
%\begin{lem}\cite[Section 5]{BL}\label{lem:BL}
%Let $A$ be an $m\times n$ matrix and let $(\mb{y}_k)_{k\geq 1}$ be a sequence of weighted best approximations to $^{t}A$.
%For each $R>1$, there exists an increasing function $\vphi:\bZ_{\geq1}\to\bZ_{\geq1}$ satisfying $\vphi(1)=1$ and, for any integer $i\geq 1$,
%\eqlabel{BLS}{
%Y_{\vphi(i+1)}\geq R^{1/2}Y_{\vphi(i)}\quad \text{and}\quad M_{\vphi(i)}Y_{\vphi(i+1)}\leq R.
%}
%\end{lem}

%In fact, instead of the second condition $M_{\vphi(i)}Y_{\vphi(i+1)}\leq R$, the condition $Y_{\vphi(i)+1}\geq R^{-1}Y_{\vphi(i+1)}$ is stated in \cite{BL}, but since $M_{k}Y_{k+1}\leq 1$, the second condition follows.
In this subsection we construct the following modified Bugeaud-Laurent sequence assuming the singular on average property. We refer the reader to \cite[Section 5]{BL} for the original version of the Bugeaud-Laurent sequence.

\begin{prop}\label{lem:MBL}
Let $A\in M_{m,n}$ be such that $^{t}A$ is singular on average and let $(\mb{y}_k)_{k\geq 1}$ be a sequence of weighted best approximations to $^{t}A$. For each $S>R>1$, there exists an increasing function $\vphi:\bZ_{\geq1}\to\bZ_{\geq1}$ satisfying the following properties:
\begin{enumerate}
    \item for any integer $i\geq1$, 
    \eqlabel{Mgrow}{
    Y_{\vphi(i+1)}\geq RY_{\vphi(i)}\quad \text{and} \quad M_{\vphi(i)}Y_{\vphi(i+1)}\leq R.
    }
    \item
    \eqlabel{Mdensity}{
    \limsup_{k\to\infty}\frac{k}{\log{Y_{\vphi(k)}}}\leq\frac{1}{\log{S}}.
    }
\end{enumerate}
\end{prop}

%the assumption is equivalent to the following statement by Theorem \ref{thm:sing}: For all $\eps>0$, \eq{\lim\limits_{k\to\infty}\frac{1}{\log_{2}Y_{k}}\left|\{i\leq k:M_{i}Y_{i+1}>\eps\}\right|=0,}
%if $\lim\limits_{k\to\infty}Y_{k}^{1/k}=\infty$, the lemma follows from \cite[Theorem 2.2]{BKLR}.

\begin{proof}
The function $\vphi$ is constructed in the following way. Fix a positive integer $V$ in Lemma \ref{lem:CGGMS} and let $\cJ=\{j\in\bZ_{\geq 1}:M_j Y_{j+1}\leq R/S^3\}$. Since $^{t}A$ is singular on average, by Proposition \ref{propA1} with $\eps=R/S^3$, we have  \eqlabel{assumption}{\lim_{k\to\infty}\frac{1}{\log Y_{k}}\left|\{i\leq k:i\in\cJ^{c}\}\right|=0.}

If the set $\cJ$ is finite, then we have $\lim\limits_{k\to\infty}Y_{k}^{1/k}=\infty$ by \eqref{assumption}, hence the proof of \cite[Theorem 2.2]{BKLR} implies that there exists a function $\vphi:\bZ_{\geq 1}\to \bZ_{\geq 1}$ for which
\[
Y_{\vphi(i+1)}\geq RY_{\vphi(i)}\quad\text{and}\quad Y_{\vphi(i)+1}\geq R^{-1}Y_{\vphi(i+1)}.
\]  
The fact that $M_{i}Y_{i+1}\leq 1$ for all $i\geq 1$ implies $M_{\vphi(i)}Y_{\vphi(i+1)}\leq R$. Equation \eqref{Mdensity} follows from $\lim\limits_{k\to\infty}Y_{k}^{1/k}=\infty$, which concludes the proof of Proposition \ref{lem:MBL}. 

Now, suppose that $\cJ$ is infinite. Then there are two possible cases:
\begin{enumerate}[label=(\roman*)]
\item $\cJ$ contains all sufficiently large positive integers.
\item There are infinitely many positive integers in $\cJ^{c}$. 
\end{enumerate}
\textbf{Case (i).} Assume the first case and let $\psi(1)=\min\{j:\cJ\supset\bZ_{\geq j}\}$. Define the auxiliary increasing sequence $(\psi(i))_{i\geq 1}$ by
\eq{
\psi(i+1)=\min\{j\in\bZ_{\geq 1}:SY_{\psi(i)}\leq Y_j\},
} which is well defined since $(Y_i)_{i\geq 1}$ is increasing. Note that $\psi(i+1) \leq \psi(i)+\lceil\log_2{S}\rceil V$ since $Y_{\psi(i)+\lceil\log_2{S}\rceil V}\geq SY_{\psi(i)}$ by Lemma \ref{lem:CGGMS}. Let us now define the sequence $(\vphi(i))_{i\geq 1}$ by, for each $i\geq 1$, 
\[ \vphi(i) = 
\begin{cases}
\psi(i) &\quad \text{if } M_{\psi(i)}Y_{\psi(i+1)} \leq R/S,\\
\psi(i+1)-1 &\quad \text{otherwise}.
\end{cases}
\]
Then the sequence $(\vphi(i))_{i\geq 1}$ is increasing and $\vphi \geq \psi$. 

Now we claim that for each $i \geq 1$,
\eqlabel{MSgrow}{
Y_{\vphi(i+1)}\geq SY_{\vphi(i)}\quad\text{and}\quad M_{\vphi(i)}Y_{\vphi(i+1)}\leq R,
} which implies Equation \eqref{Mdensity} since $Y_{\vphi(k)}\geq S^{k-1}Y_{\vphi(1)}$ for all $k\geq 1$. Thus, the claim concludes the proof of Proposition \ref{lem:MBL}.

\begin{proof}[Proof of Equation \eqref{MSgrow}]
There are four possible cases on the values of $\vphi(i)$ and $\vphi(i+1)$. 

\medskip
$\bullet$~ Assume that $\varphi(i)=\psi(i)$ and $\varphi(i+1)
=\psi(i+1)$. By the definition of $\psi(i+1)$, we have
$$
Y_{\varphi(i+1)}=Y_{\psi(i+1)}\geq SY_{\psi(i)}=S Y_{\varphi(i)}.
$$
If $\psi(i)\neq \psi(i+1)-1$, then by the definition of $\varphi(i)$,
we have
$$
M_{\varphi(i)} Y_{\varphi(i+1)}=M_{\psi(i)} Y_{\psi(i+1)}\leq R/S \leq R.
$$
If $\psi(i)= \psi(i+1)-1$, then $\varphi(i+1)= \varphi(i)+1$, hence
$$
M_{\varphi(i)}Y_{\varphi(i+1)}=M_{\varphi(i)}Y_{\varphi(i)+1}\leq 1 \leq R.
$$
This proves Equation \eqref{MSgrow}.

\medskip
$\bullet$~ Assume that $\varphi(i)=\psi(i)$ and $\varphi(i+1)=
\psi(i+2)-1$. By the definition of $\psi(i+1)$, we have
$$
Y_{\varphi(i+1)}=Y_{\psi(i+2)-1}\geq Y_{\psi(i+1)}\geq SY_{\psi(i)}= S Y_{\varphi(i)}.
$$
It follows from the minimality of $\psi(i+2)$ that $SY_{\psi(i+1)}> Y_{\psi(i+2)-1}$. If $\psi(i+1)>\psi(i)+1$, then $M_{\psi(i)}
Y_{\psi(i+1)}\leq R/S$ by the definition of $\varphi(i)$. Hence, we have
$$
M_{\varphi(i)}Y_{\varphi(i+1)}= M_{\psi(i)}Y_{\psi(i+2)-1}\leq SM_{\psi(i)}Y_{\psi(i+1)}\leq R.
$$
If $\psi(i+1)=\psi(i)+1$, then $M_{\psi(i)} Y_{\psi(i)+1} \leq R/S^3$ since $\psi(i)\in\cJ$. Hence,
$$
M_{\varphi(i)}Y_{\varphi(i+1)}= M_{\psi(i)} Y_{\psi(i+2)-1}\leq SM_{\psi(i)} Y_{\psi(i)+1} \leq R/S^2 \leq R.
$$
This proves Equation \eqref{MSgrow}.

\medskip
$\bullet$~ Assume that $\varphi(i)=\psi(i+1)-1$ and $\varphi(i+1)=
\psi(i+1)$. Since $\psi(i+1)-1\in\cJ$, we have
$$
M_{\varphi(i)}Y_{\varphi(i+1)}=M_{\psi(i+1)-1}Y_{\psi(i+1)}\leq R/S^3 \leq R.
$$
If $\psi(i+1)-1=\psi(i)$, then by the definition of $\psi(i+1)$, we
have
$$
\frac{Y_{\varphi(i+1)}}{Y_{\varphi(i)}}=\frac{Y_{\psi(i+1)}}{Y_{\psi(i+1)-1}}
=\frac{Y_{\psi(i+1)}}{Y_{\psi(i)}}\geq S.
$$
If $\psi(i+1)-1>\psi(i)$, then we have $M_{\psi(i)}Y_{\psi(i+1)}>R/S$ by the definition of $\varphi(i)$, and we have $Y_{\psi(i+1)-1} < SY_{\psi(i)} \leq SY_{\psi(i)+1}$ from the minimality of $\psi(i+1)$. We also have $M_{\psi(i)}Y_{\psi(i)+1} \leq R/S^3$ since $\psi(i)\in\cJ$. Therefore
$$
\frac{Y_{\varphi(i+1)}}{Y_{\varphi(i)}} =
\frac{Y_{\psi(i+1)}}{Y_{\psi(i+1)-1}}=
\frac{M_{\psi(i)}Y_{\psi(i+1)}}{M_{\psi(i)}Y_{\psi(i+1)-1}}
\geq\frac{R/S}{SM_{\psi(i)}Y_{\psi(i)+1}}
\geq\frac{R/S}{R/S^2}=S.
$$
This proves Equation \eqref{MSgrow}.

\medskip
$\bullet$~ Assume that $\varphi(i)=\psi(i+1)-1$ and $\varphi(i+1)=\psi(i+2)-1$.
By the previous case computations, we have
$$
\frac{Y_{\varphi(i+1)}}{Y_{\varphi(i)}}
=\frac{Y_{\psi(i+2)-1}}{Y_{\psi(i+1)-1}}
\geq\frac{Y_{\psi(i+1)}}{Y_{\psi(i+1)-1}}\geq S.
$$
We have $SY_{\psi(i+1)}> Y_{\psi(i+2)-1}$ from the minimality of $\psi(i+2)$. Thus since $\psi(i+1)-1\in\cJ$, we have
$$
M_{\varphi(i)}Y_{\varphi(i+1)}=M_{\psi(i+1)-1}Y_{\psi(i+2)-1}=M_{\psi(i+1)-1}Y_{\psi(i+1)}
\left(\frac{Y_{\psi(i+2)-1}}{Y_{\psi(i+1)}}\right) \leq R.
$$
This proves Equation \eqref{MSgrow}.
\end{proof}
\noindent \textbf{Case (ii).} Now we assume the second case and let $j_0 =\min\cJ$. Partition $\bZ_{\geq j_0}$ into disjoint subset
\[
\bZ_{\geq j_0}= C_1 \sqcup D_1 \sqcup C_2 \sqcup D_2 \sqcup \cdots
\]
where $C_i\subset\cJ$ and $D_j\subset\cJ^{c}$ are sets of consecutive integers with 
\[
\max C_i < \min D_i \leq \max D_i < \min C_{i+1}
\]
for all $i\geq 1$. We consider the following two subcases.
\\ 
\textbf{(ii) - 1.} If there is $i_0 \geq 1$ such that $|C_i|< 3\lceil\log_2{S}\rceil V$ for all $i\geq i_0$, then we have, for $k_0 = \min C_{i_0}$,  
\eq{
\frac{k}{\log Y_k}\leq \frac{k_0+\left(3\lceil\log_2{S}\rceil V +1\right)|\{i\leq k: i\in\cJ^{c}\}|}{\log Y_{k}},
}
since there exists an element of $\cJ^c$ in any finite sequence of $3\lceil\log_2{S}\rceil V +1$ consecutive integers at least $k_0$. Therefore $\lim\limits_{k\to\infty}Y_{k}^{1/k}=\infty$ by \eqref{assumption} and this concludes the proof of Proposition \ref{lem:MBL} following the proof when $\cJ$ is finite at the beginning. 
\\
\textbf{(ii) - 2.} The remaining case is that the set
\eq{
\{i:|C_{i}|\geq 3\lceil\log_2{S}\rceil V \}=\{i(1)<i(2)<\cdots<i(k)<\cdots:k\in\bN\}
}
is infinte. 

For each $k\geq 1$, let us define an increasing finite sequence $(\psi_k (i))_{1\leq i \leq m_k +1}$ of positive integers by setting $\psi_k (1)=\min C_{i(k)}$ and by induction 
\[
\psi_k (i+1) = \min \{j\in C_{i(k)}: SY_{\psi_k (i)} \leq Y_j \},
\] as long as this set is nonempty. Since $C_{i(k)}$ is a finite sequence of consecutive positive integers with length at least $3\lceil\log_2{S}\rceil V$ and $Y_{i+\lceil\log_2{S}\rceil V}\geq SY_{i}$ for every $i\geq 1$ by Lemma \ref{lem:CGGMS}, there exists an integer $m_k \geq 2$ such that $\psi_k (i)$ is defined for $i=1,\dots,m_k +1$. Note that $\psi_k (i)$ belongs to $\cJ$ since $C_{i(k)}\subset \cJ$.

As in \textbf{Case (i)}, let us define an increasing finite sequence $(\vphi_k (i))_{1\leq i\leq m_k}$ of positive integers by
\[
\vphi_k (i) =\begin{cases}
\psi_k (i) &\quad \text{if } M_{\psi_k (i)}Y_{\psi_k (i+1)} \leq R/S,\\
\psi_k (i+1)-1 &\quad \text{otherwise}.
\end{cases}
\]
Following the proof of \textbf{Case (i)}, we have for each $i=1,\dots,m_k -1$,
\eqlabel{kMSgrow}{
Y_{\vphi_k (i+1)}\geq S Y_{\vphi_k (i)} \quad\text{and}\quad M_{\vphi_k (i)}Y_{\vphi_k (i+1)} \leq R.
}

Note that $\vphi_k (m_k)<\vphi_{k+1}(1)$. Let us define an increasing finite sequence $(\vphi_k'(i))_{1\leq i\leq n_k +1}$ of positve integers to interpolate between $\vphi_k (m_k)$ and $\vphi_{k+1}(1)$. Let $j_0 =\vphi_{k+1}(1)$. If the set $\{j\in\bZ_{\geq \vphi_k (m_k)}: Y_{j_0}\geq RY_j\}$ is empty, then we set $n_k =0$ and $\vphi_k'(1)=j_0 = \vphi_{k+1}(1)$. Otherwise, follwing \cite[Theorem 2.2]{BKLR}, by decreasing induction, let $n_k \in\bZ_{\geq 1}$ be the maximal positive integer such that there exists $j_1,\dots, j_{n_k}\in\bZ_{\geq 1}$ such that for $\ell=1,\dots,n_k$, the set $\{j\in\bZ_{\geq \vphi_k (m_k)}: Y_{j_{\ell-1}}\geq RY_j\}$ is nonempty and for $\ell=1,\dots,n_k +1$, the integer $j_\ell$ is its largest element. Set $\vphi_k'(i)=j_{n_k +1 -i}$ for $i=1,\dots, n_k +1$. Then the sequence $(\vphi_k'(i))_{1\leq i\leq n_k +1}$ is contained in $[\vphi_k(m_k),\vphi_{k+1}(1)]$ and satisfies that for $i=1,\dots,n_k$,
\eqlabel{kMRgrow}{
Y_{\vphi_k' (i+1)}\geq RY_{\vphi_k'(i)} \quad\text{and}\quad M_{\vphi_k' (i)}Y_{\vphi_k'(i+1)}\leq R
} from the proof of \cite[Theorem 2.2]{BKLR}.

Now, putting alternatively together the sequences $(\vphi_k (i))_{1\leq i\leq m_k -1}$ and $(\vphi_k' (i))_{1\leq i\leq r_k}$ as $k$ ranges over $\bZ_{\geq 1}$, we define $N_k = \sum_{\ell =1}^{k-1}(m_\ell -1 +n_\ell)$ and 
\[
\vphi (i) =\begin{cases}
\vphi_k (i-N_k) &\quad \text{if } 1+N_k \leq i\leq m_k -1 +N_k,\\
\vphi_k' (i+1-m_k -N_k) &\quad \text{if } m_k +N_k \leq i\leq r_k -1 +m_k +N_k.
\end{cases}
\]
Here, we use the standard convention that an empty sum is zero.
With Equation \eqref{kMSgrow} for $i=1,\dots,m_k -2$ and Equation \eqref{kMRgrow} for $i=1,\dots,n_k$, since $\vphi_k'(n_k +1) = \vphi_{k+1}(1)$, it is enough to show the following lemma to prove that the map $\vphi$ satisfies Equation \eqref{Mgrow}.

\begin{lem}\label{lem:kMgrow}
For every $k\in\bZ_{\geq 1}$, we have
\eqlabel{kMgrow}{
Y_{\vphi_k'(1)}\geq RY_{\vphi_{k}(m_k -1)}\quad\text{and}\quad M_{\vphi_k(m_k-1)}Y_{\vphi_k'(1)}\leq R.
}

\end{lem}
\begin{proof}
Since $\vphi_k'(1)\geq \vphi_k(m_k)$ and Equation \eqref{kMSgrow} with $i=m_k-1$, we have
\[
Y_{\vphi_k'(1)}\geq Y_{\vphi_k(m_k)} \geq SY_{\vphi_k(m_k-1)} \geq RY_{\vphi_{k}(m_k -1)},
\] which prove the left hand side of Equation \eqref{kMgrow}.
If $\vphi_k'(1)=\vphi_k(m_k)$, then Equation \eqref{kMSgrow} with $i=m_k-1$ gives the right hand side of Equation \eqref{kMgrow}.

Now assume that $\vphi_k'(1)>\vphi_k(m_k)$. By the maximality of $n_k$, we have $Y_{\vphi_k'(1)}\leq RY_{\vphi_k(m_k)}$. First, we will prove that $\vphi_k(m_k)=\psi_k (m_k)$. For a contradiction, assume that $\vphi_k(m_k)=\psi_k(m_k+1)-1>\phi_k(m_k)$. Following the third subcase of the proof of Equation \eqref{MSgrow}, we have 
\[
\frac{Y_{\psi_k (m_k+1)}}{Y_{\psi_k(m_k+1)-1}}=\frac{M_{\psi_k(m_k)}Y_{\psi_k (m_k+1)}}{M_{\psi_k(m_k)}Y_{\psi_k(m_k+1)-1}}\geq S.
\]   
Hence by the construction of $\vphi_k'(1)$, we have $\vphi_k'(1)=\vphi_k(m_k)$, which is a contradiction to our assumption $\vphi_k'(1)>\vphi_k(m_k)$.

To show the right hand side of Equation \eqref{kMgrow}, we consider two possible values of $\vphi_k(m_k-1)$.

Assume that $\vphi_k(m_k-1)=\psi_k(m_k-1)$. If $\psi_k(m_k-1)>\psi_k(m_k)-1$, then by the definition of $\vphi_k(m_k-1)$, we have $M_{\psi_k(m_k-1)}Y_{\psi_k(m_k)}\leq R/S$. If $\psi_k(m_k-1)=\psi_k(m_k)-1$, then $M_{\psi_k(m_k-1)}Y_{\psi_k(m_k)}\leq R/S^3 \leq R/S$ since $\psi_k(m_k)-1\in \cJ$. Since $\vphi_k(m_k)=\psi_k(m_k)$, we have
\[
M_{\vphi_k(m_k-1)}Y_{\vphi_k'(1)} = M_{\psi_k(m_k-1)}Y_{\psi_k(m_k)}\left(\frac{Y_{\vphi_k'(1)}}{Y_{\vphi_k(m_k)}}\right) \leq R,
\]
which proves the right hand side of Equation \eqref{kMgrow}.

Assume that $\vphi_k(m_k-1)=\psi_k(m_k)-1$. Since $\vphi_k(m_k)=\psi_k(m_k)$ and $\psi_k(m_k)-1\in \cJ$, we have
\[
M_{\vphi_k(m_k-1)}Y_{\vphi_k'(1)} = M_{\psi_k(m_k)-1}Y_{\psi_k(m_k)}\left(\frac{Y_{\vphi_k'(1)}}{Y_{\vphi_k(m_k)}}\right) \leq R,
\]
which proves the right hand side of Equation \eqref{kMgrow}, and concludes the proof of Lemma \ref{lem:kMgrow}.
\end{proof}

Finally, we will show Equation \eqref{Mdensity} for the map $\vphi$. Since there exists an element of $\cJ^c$ in any finite sequence of $3\lceil\log_2{S}\rceil V +1$ consecutive integers in the complement of $\bigcup_{k\geq 1} C_{i(k)}$, there exists $c_0 \geq 0$ such that for every $k\geq 1$, we have
\[
\frac{|\{j\leq \vphi(k): j\notin \bigcup_{k\geq 1} C_{i(k)} \}|}{\log Y_{\vphi(k)}} \leq 
\frac{c_0+\left(3\lceil\log_2{S}\rceil V +1\right)|\{j\leq \vphi(k): j \in\cJ^{c}\}|}{\log Y_{\vphi(k)}},
\] which converges to $0$ as $k\to +\infty$ by \eqref{assumption}. Let us define 
$$n(k)= |\{i\leq k : Y_{\vphi(i)} \geq SY_{\vphi(i+1)}\}|.$$
For each integer $\ell\geq 1$, since $Y_{i+\lceil\log_2{S}\rceil V}\geq SY_{i}$ for every $i\geq 1$ by Lemma \ref{lem:CGGMS}, and by the maximality of $m_\ell$ in the construction of $(\vphi_\ell (i))_{1\leq i\leq m_\ell}$, we have $|\{j\in C_{i(\ell)}: j\geq \vphi_{\ell}(m_\ell)\}|\leq 2\lceil\log_2{S}\rceil V$. If $\vphi(i)$ belongs to $C_{i(\ell)}$ but $\vphi(i+1)$ does not, then $\vphi(i)\geq \vphi_\ell (m_\ell)$. If $\vphi(i)$ and $\vphi(i+1)$ belong to $C_i(\ell)$, then $\vphi$ and $\vphi_\ell$ coincide on $i$ and $i+1$. Thus, by Equation \eqref{kMSgrow}, we have
\[
\begin{split}
k-n(k)&=|\{i\leq k : Y_{\vphi(i)} < SY_{\vphi(i+1)}\}|\\
&\leq \left(2\lceil\log_2{S}\rceil V\right) \big|\{j\leq \vphi(k): j\notin \bigcup_{k\geq 1} C_{i(k)}\}\big|.
\end{split}
\]
Therefore, we have
\[
\begin{split}
\limsup_{k\to\infty}\frac{k}{\log Y_{\vphi(k)}}&=\limsup_{k\to\infty}\frac{n(k)+k-n(k)}{\log Y_{\vphi(k)}} = \limsup \frac{n(k)}{\log Y_{\vphi(k)}}\\
&\leq \limsup_{k\to\infty}\frac{n(k)}{\log S^{n(k)-1}Y_{\vphi(1)}}=\frac{1}{\log S}.
\end{split}
\]
This proves Equation \eqref{Mdensity} and concludes the proof of Proposition \ref{lem:MBL}.

\end{proof}

\subsection{Dimension estimates}
Following the notation in \cite{BHKV10}, given a sequence $\{\mb{y}_i\}$ in $\bZ^m\setminus \{\mb{0}\}$ and $\al\in(0,1/2)$, let
\[
\text{Bad}_{\{\mb{y}_i\}}^\al \defn\{\mb{\theta}\in\bR^m:|\mb{\theta}\cdot\mb{y}_{i}|_{\bZ}\geq\al\ \text{for all}\ i\geq1\}.
\]

\begin{prop}\cite{CGGMS}\label{prop:CGGMS}
Let $A\in M_{m,n}$ be a matrix and let $(\mb{y}_k)_{k\geq 1}$ be a sequence of weighted best approximations to $^{t}A$ and let $R>1$ and $\al\in(0,1/2)$ be given. 
Suppose that there exists an increasing function $\vphi:\bZ_{\geq1}\to\bZ_{\geq1}$ such that for any integer $i\geq1$ 
\[
M_{\vphi(i)}Y_{\vphi(i+1)}\leq R.
\] 
Then $\textup{Bad}_{\{\mb{y}_{\vphi(i)}\}}^\al$ is a subset of $\textup{Bad}_A(\eps)$ where $\eps=\frac{1}{R}\left(\frac{\al^2}{4mn}\right)^{1/\del}$ and $\de=\min\{r_i,s_j:1\leq i\leq m, 1\leq j\leq n\}$.
\end{prop}
\begin{proof}
In the proof of \cite[Theorem 1.11]{CGGMS}, the condition $Y_{\vphi(i)+1}\geq R^{-1}Y_{\vphi(i+1)}$ is used. However, the assumption $M_{\vphi(i)}Y_{\vphi(i+1)}\leq R$ also implies the same conclusion. 
\end{proof}

\begin{prop}\cite{CGGMS}\label{prop:CGGMS_2}
For any $\al\in(0,1/2)$, there exists $R(\al)>1$ with the following property. Let $(\mb{y}_k)_{k\geq1}$ be a sequence in $\bZ^m\setminus\{\mb{0}\}$ such that $\|\mb{y}_{k+1}\|_{\mb{r}}/\|\mb{y}_{k}\|_{\mb{r}}\geq R(\al)$ for all $k\geq1$. Then 
\[
\textup{dim}_{H}\left(\textup{Bad}_{\{\mb{y}_i\}}^\al\right) \geq m-C\limsup_{k\to\infty}\frac{k}{\log{\|\mb{y}_{k}\|_{\mb{r}}}}.
\]
for some positive constant $C=C(\al)$. 
\end{prop}
\begin{proof}
The proof of \cite[Theorem 6.1]{CGGMS} concludes this proposition.
\end{proof}

The two propositions are used in \cite[Theorem 5.1]{BKLR} in the unweighted setting.

%\begin{proof}[Proof of Theorem \ref{thmF1}]
%By Lemma \ref{lem:BL}, Proposition \ref{prop:CGGMS}, and Proposition \ref{prop:CGGMS_2}, for each $R\geq R(\al)>1$,
%\begin{align*}
%    \textup{dim}_{H}\left(\textup{Bad}_{\mb{r},\mb{s}}(A)\right)	
%    &\geq \textup{dim}_{H}\left(\textup{Bad}_{\mb{r},\mb{s}}^\eps (A)\right) \\
%    &\geq \textup{dim}_{H}\left(\textup{Bad}_{\{\mb{y}_{\vphi(i)}\}}^\al\right)\\
%    &\geq m- C\limsup_{k\to\infty}\frac{k}{\log{Y_{\vphi(k)}}}\\
%    &\geq m- C\limsup_{k\to\infty}\frac{k}{\log{(R^{(k-1)/2}Y_{\vphi(1)})}}\\
%    &=m-\frac{2C}{\log{R}} 
%\end{align*}
%where $\eps=\frac{1}{R}\left(\frac{\al^2}{4mn}\right)^{1/\del}$. Taking $R\to\infty$, we have $\textup{dim}_{H}\left(\textup{Bad}_{\mb{r},\mb{s}}(A)\right)=m$. 
%\end{proof}

\begin{proof}[Proof of Theorem \ref{thmA1} (\ref{S3})$\implies$(\ref{S1})]

Suppose $A$ is singular on average. By Corollary \ref{eqCor}, $^{t}A$ is also singular on average. Let $(\mb{y}_k)_{k\geq 1}$ be a sequence of weighted best approximations to $^{t}A$. Then, by Proposition \ref{lem:MBL}, Proposition \ref{prop:CGGMS}, and Proposition \ref{prop:CGGMS_2}, for each $S>R(\al)>1$, we have
\begin{align*}
    \textup{dim}_{H}\left(\textup{Bad}_A(\eps)\right) 
    &\geq \textup{dim}_{H}\left(\textup{Bad}_{\{\mb{y}_{\vphi(i)}\}}^\al\right)\\
    &\geq m- C\limsup_{k\to\infty}\frac{k}{\log{Y_{\vphi(k)}}}\\
    &\geq m- \frac{C}{\log{S}}     
\end{align*}
where $\eps=\frac{1}{R(\al)}\left(\frac{\al^2}{4mn}\right)^{1/\del}$. Taking $S\to\infty$, we have $\textup{dim}_{H}\left(\textup{Bad}_A(\eps)\right)=m$ for $\eps=\frac{1}{R(\al)}\left(\frac{\al^2}{4mn}\right)^{1/\del}$.

\end{proof}

\section*{Acknowledgement}   %%d
We would like to thank Manfred Einsiedler and Fr\'{e}d\'{e}ric Paulin for helpful discussions and valuable comments.
SL is an associate member of KIAS. SL and TK were supported by Samsung Science and Technology Foundation under Project No. SSTF-BA1601-03 and the National Research Foundation of Korea under Project
Number NRF-2020R1A2C1A01011543.
TK was supported by 
the National Research Foundation of Korea under Project Number NRF-2021R1A6A3A13039948. WK was supported by the Korea Foundation for
Advanced Studies.

%\bibliographystyle{amsalpha}
%\bibliography{uribib}
\def\cprime{$'$} \def\cprime{$'$} \def\cprime{$'$}
\providecommand{\bysame}{\leavevmode\hbox to3em{\hrulefill}\thinspace}
\providecommand{\MR}{\relax\ifhmode\unskip\space\fi MR }
\providecommand{\MRhref}[2]{%
  \href{http://www.ams.org/mathscinet-getitem?mr=#1}{#2}
}
\providecommand{\href}[2]{#2}

\end{document}